\documentclass[12pt]{amsart}
\usepackage{amsfonts,mathrsfs,amsmath,amssymb,amscd,mathtools}
\usepackage{graphicx}
\usepackage{caption}
\usepackage{color,xcolor}
\usepackage[percent]{overpic}
\usepackage[all]{xy}
\usepackage{hyperref}
\usepackage{leftidx}

\newtheorem{dummy}{dummy}[section]
\newtheorem{lemma}[dummy]{Lemma}
\newtheorem{proposition}[dummy]{Proposition}
\newtheorem{corollary}[dummy]{Corollary}
\newtheorem{theorem}[dummy]{Theorem}

\theoremstyle{definition}
\newtheorem{definition}[dummy]{Definition}
\newtheorem{definition/proposition}[dummy]{Definition/Proposition}

\newtheorem{remark}[dummy]{Remark}
\newtheorem{example}[dummy]{Example}

\numberwithin{equation}{subsection}
\numberwithin{figure}{section}

\usepackage{tikz}

\usepackage{epstopdf}

\newcommand{\mc}{\mathcal}
\newcommand{\mr}{\mathrm}
\newcommand{\mb}{\mathbb}

\usepackage{enumerate}
\usepackage{enumitem}

\setlist[description]{style=nextline}

\usepackage{cite}

\title{Ruling polynomials and augmentations for Legendrian tangles}
\author[Tao Su]{Tao Su}
\email{taosu@math.berkeley.edu}

\address{Department of Mathematics, University of California, Berkeley}

\begin{document}

\begin{abstract}
Associated to Legendrian links in the standard contact three-space, Ruling polynomials are Legendrian isotopy invariants, which also compute augmentation numbers, that is, the points-counting of augmentation varieties for Legendrian links (up to a normalized factor) \cite{HR15}. In this article, we generalize this picture to Legendrian tangles, which are morally the pieces obtained by cutting Legendrian link fronts along 2 vertical lines. Moreover, we show that the Ruling polynomials for Legendrian tangles satisfy the composition axiom. In the special case of Legendrian knots, our arguments provide new proofs to the main results in \cite{HR15}. In the end, we also introduce generalized Ruling polynomials for Legendrian tangles, to account for non-acyclic augmentations in the ``Ruling polynomials compute augmentation numbers" picture.
\end{abstract}

\maketitle

\tableofcontents

\section*{Introduction}

Similar to smooth knot theory, there's a parallel study of Legendrian knots in contact three manifolds. The fundamental case is the Legendrian knots in the standard contact three space. The classical Legendrian invariants are the topological knot type, the Thurston-Bennequin number and the rotation number. They determine a complete set of invariants for some Legendrian knots, including the unknots, torus knots and the figure eight knots \cite{EF95,EH00}. However, in general they do not determine a complete Legendrian knot invariant, as shown by the Chekanov pairs \cite{Che02}. They have the same classical invariants, but are distinguished by a stronger invariant, the Chekanov-Eliashberg differential graded algebra. The Chekanov-Eliashberg DGAs are special cases of Legendrian contact homology differential graded algebras (LCH DGAs). Morally, associate to any pair $(V, \Lambda)$ of a Legendrian submanifold $\Lambda$ contained in a contact manifold $V$ the LCH DGAs $(\mc{A}(V,\Lambda), \partial)$ are defined via Floer theory \cite{Eli98,EGH00}. The generators are indexed by the Reeb chords of $\Lambda$. The differential counts holomorphic disks in the symplectization $\mb{R}\times V$, with boundaries along the Lagrangian cylinder $\mb{R}\times\Lambda$, and meeting the Reeb chords at positive or negative infinity. The algebras are Legendrian isotopy invariants, up to homotopy equivalence.
In the case for Legendrian knots $\Lambda$ in the standard contact three space, the LCH DGA $(\mc{A}(\Lambda), \partial)$ can also be defined purely combinatorially \cite{Che02,ENS02}.

Usually, the LCH DGAs $(\mc{A}(\Lambda), \partial)$ associated to Legendrian knots are hard to work with directly. Instead, one tries to extract some numerical invariants from them. One fundamental idea in \cite{Che02} is to consider the functor of points of $(\mc{A}(\Lambda),\partial)$:
$$\text{commutative ring } \mathbf{r}\rightarrow \{\text{DGA morphisms }(\mc{A},\partial)\rightarrow \mathbf{r}\}/\text{DGA homotopy}$$
and count the points appropriately over a finite field. One way to do so is as follows. Let $r$ be the gcd of the rotation numbers of the connected components of $\Lambda$, which ensures the existence of a $\mb{Z}/2r$-valued Maslov potential $\mu$, the DGA $\mc{A}(\Lambda)$ associated to $(\Lambda,\mu)$ is then naturally $\mb{Z}/2r$-graded. Start with a nonnegative integer $m$ dividing $2r$, we consider the space of $\mathbb{Z}/m$-graded augmentations (``$m$-graded points") valued in any finite field $k$. This defines an algebraic variety $\mr{A}ug_m(\Lambda, k)$, called the augmentation variety (See \ref{subsec:augmentations for tangles} for more details). Now, a normalized count of the points of $\mr{A}ug_m(\Lambda, \mb{F}_q)$ over a finite field $\mb{F}_q$ gives the augmentation number $aug_m(\Lambda,q):=q^{-\mr{dim}_{\mb{C}}\mr{A}ug(\Lambda,\mb{C})}|\mr{A}ug(\Lambda,\mb{F}_q)|$. This is a Legendrian isotopy invariant \cite[Thm.3.2]{HR15} and in fact distinguishes the Chekanov pairs.

More recently, some categorical Legendrian isotopy invariants, the augmentation categories $\mc{A}ug_+(\Lambda,k), \mc{A}ug_-(\Lambda,k)$, and also some of their equivalent versions or generalizations are constructed \cite{STZ14,BC14,NRSSZ15}. The augmentation categories are $A_{\infty}$ categories, and up to $A_{\infty}$-equivalence, are invariant under Legendrian isotopy of $\Lambda$. They can be viewed as the categorical refinement of augmentation varieties, in the sense that augmentation varieties only encodes the 0-th order information (points) of the LCH DGA $\mc{A}(\Lambda)$, while the augmentation categories defined also encode the higher order information (tangent spaces with additional structures). It's expected that, a refined counting of points using the augmentation category (homotopy cardinality) may give a more natural way to count augmentations (See \cite{NRSS15}).

On the other hand, similar to knot projections in smooth knot theory, the Legendrian knots admit and are determined by the front projections. By considering the types of the decomposition of the front diagrams, one leads to the notion of normal rulings \cite{CP05,Fuc03}. In \cite{CP05}, for each $m\geq 0$ as above, it's shown that a weighed count of the ($m$-graded) normal rulings of the front diagram for $\Lambda$, gives a Legendrian isotopy invariant $R_{\Lambda}^m(z)$, called the $m$-graded ruling polynomials. It turns out, the ruling polynomials can also be used to distinguish the Chekanov pairs. Ruling polynomials are the analogue of Jones polynomials in smooth knot theory, in the sense that they can also be characterized by skein relations \cite{Rut05}.

Moreover, the Ruling polynomials also admit a contact geometry interpretation in terms of augmentation numbers:

\begin{proposition}[{\cite[Thm.1.1]{HR15}}]\label{prop:Rulings and augmentations for Legendrian knots}
The augmentation numbers and the Ruling polynomials of $\Lambda$ determine each other by
\begin{equation*}
\mr{aug}(\Lambda,q)=q^{-\frac{d+l}{2}}z^lR_{\Lambda}^m(z)
\end{equation*}
where $z=q^{\frac{1}{2}}-q^{-\frac{1}{2}}$, $d$ is the maximal degree in $z$ of the Laurent polynomial $R_{\Lambda}^m(z)$, and $l$ is the number of connected components of $\Lambda$.
\end{proposition}

Regarding the structure of the augmentation variety $\mr{A}ug_m(\Lambda, k)$, the following is known:
\begin{proposition}[{\cite[Thm.3.4]{HR15}}]\label{prop:augmentation varieties for Legendrian knots}
Suppose $\Lambda$ has the nearly plat front diagram $\pi_{xz}(\Lambda)$ (see Section \ref{subsubsec:fronts}), with a fixed Maslov potential $\mu$ and $l$ base points such that each connected component has a single base point. Then there's a decomposition of the augmentation variety $\mr{Aug}_m(T;k)$ into subvarieties
\begin{eqnarray*}
\mr{Aug}_m(T;k)=\sqcup_{\rho}\mr{Aug}_m^{\rho}(T;k)
\end{eqnarray*}
where $\rho$ runs over all $m$-graded normal rulings of $\pi_{xz}(\Lambda)$, and
\begin{eqnarray*}
\mr{Aug}_m^{\rho}(T;k)\cong (k^*)^{-\chi(\rho)+l}\times k^{r(\rho)}
\end{eqnarray*}
where $\chi(\rho)=c_R-s(\rho)$, $c_R$ is the number of right cusps in $\pi_{xz}(\Lambda)$ and $s(\rho)$ is the number of switches of $\rho$ (See Definition \ref{def:switches and returns}). Finally, $r(\rho)$ is the number of $m$-graded returns if $m\neq 1$ and the number of $m$-graded returns and right cusps if $m=1$.
\end{proposition}

\addtocontents{toc}{\protect\setcounter{tocdepth}{1}}
\subsection*{Main results}
In this article, we will generalize the previous results to Legendrian tangles\footnote{Throughout the context, Legendrian tangles are assumed to be oriented.}. Legendrian tangles are (special) Legendrian submanifolds in $J^1U\subset \mb{R}_{x,y,z}^3$ transverse to the boundary $\partial \overline{J^1U}$, for some open interval $U$ in $\mb{R}_x$. Similar to Legendrian knots, one can consider the types of the decompositions of Legendrian tangle fronts. As a generalization, this leads to normal rulings and Ruling polynomials for Legendrian tangles. In this case, the boundaries (some set of labeled endpoints) of the tangles are also invariant during a Legendrian isotopy. Hence one can in fact define Ruling polynomials $<\rho_L|R_T^m(z)|\rho_R>$ with fixed boundary conditions, for a Legendrian tangle $T$ with a Maslov potential $\mu$ (see Section \ref{sec:Ruling polynomials for tangles}). Here $\rho_L$ (resp. $\rho_R$) is a given $m$-graded normal Ruling on the left (resp. right) piece ($=$ parallel strands) $T_L$ (resp. $T_R$) of $T$. As the first result, we show the Legendrian invariance and composition axiom for Ruling polynomials:

\begin{theorem}[See Theorem \ref{thm:invariance and composition of Ruling polynomials}]
The $m$-graded Ruling polynomials $<\rho_L|R_T^m(z)|\rho_R>$ are Legendrian isotopy invariants for $(T,\mu)$.\\
Moreover, suppose $T=T_1\circ T_2$ is the composition of two Legendrian tangles $T_1, T_2$, that is, $(T_1)_R=(T_2)_L$ and $T=T_1\cup_{(T_1)_R}T_2$, then the \emph{composition axiom} for Ruling polynomials holds:
\begin{equation*}
<\rho_L|R_T^m(z)|\rho_R>=
\sum_{\rho_I}<\rho_L|R_{T_1}^m(z)|\rho_I><\rho_I|R_{T_2}^m(z)|\rho_R>
\end{equation*}
where $\rho_I$ runs over all the $m$-graded normal rulings of $(T_1)_R=(T_2)_L$.
\end{theorem}

On the other hand, generalizing the LCH DGAs for Legendrian knots, one can construct a (bordered) LCH DGA $\mc{A}(T,\mu,*_1,\ldots,*_B)$, associated to any Legendrian tangle $(T,\mu)$ with base points $*_1,\ldots,*_B$. For example, see \cite{Siv11} and \cite[Section.6]{NRSSZ15} in the case when $T$ has the simple front diagram. As usual, one obtains the homotopy invariance of the DGAs. Hence, by a similar procedure as in the case of Legendrian knots, one can consider the associated augmentation varieties $\mr{Aug}_m(T,\epsilon_{\rho_L},\rho_R;k)$ and augmentation numbers $\mr{aug}_m(T,\rho_L,\rho_R;q)$, with fixed boundary conditions $(\rho_L, \rho_R)$ as above (see Definition \ref{def:aug varieties with boundary conditions}, \ref{def:augmentation number}). These augmentation numbers are again Legendrian isotopy invariants. Moreover, generalizing the previous Proposition \ref{prop:Rulings and augmentations for Legendrian knots} \cite[Thm.1.1]{HR15}, we show that:

\begin{theorem}[See Theorem \ref{thm:counting for tangles}]
Let $T$ be a Legendrian tangle equipped with a $\mb{Z}/2r$-valued Maslov potential $\mu$ and $B$ base points so that each connected component containing a right cusp has at least one base point. Fix a nonnegative integer $m$ dividing $2r$ and $m$-graded normal rulings $\rho_L,\rho_R$ of $T_L,T_R$ respectively, then the augmentation numbers and Ruling polynomials of $(T,\mu)$ are related by
\begin{equation*}
\mr{aug}_m(T,\rho_L,\rho_R;q)=q^{-\frac{d+B}{2}}z^B<\rho_L|R_T^m(z)|\rho_R>
\end{equation*}
where $q$ is the order of a finite field $\mb{F}_q$, $z=q^{\frac{1}{2}}-q^{-\frac{1}{2}}$, $d$ is the maximal degree in $z$ of $<\rho_L|R_T^m(z)|\rho_R>$.
\end{theorem}

\begin{remark}
When $T$ is a Legendrian knot\footnote{Throughout the context, we make no distinction between `Legendrian knot' and `Legendrian link'.}, with $B=l$ base points placed on $T$ so that each connected component of $T$ contains a single base point. The left and right pieces of $T$ are empty tangles, hence the boundary conditions become trivial and the theorem reduces to the previous proposition \ref{prop:Rulings and augmentations for Legendrian knots}. This gives a new proof of Proposition \ref{prop:Rulings and augmentations for Legendrian knots} \cite[Thm.1.1]{HR15}.
\end{remark}

More generally, one can consider the augmentation varieties $\mr{Aug}_m(T,\epsilon_{\rho_L},\rho_R;k)$ with boundary conditions $(\epsilon_L,\rho_R)$, where $\epsilon_L$ is any $m$-graded augmentation defining $\rho_L$ of $T_L$ (all such augmentations form an orbit $\mc{O}_m(\rho_L;k)$ of the canonical one $\epsilon_{\rho_L}$, see Remark \ref{rem:normal rulings via non-degenerate augmentations}).
Similar to Proposition \ref{prop:augmentation varieties for Legendrian knots} \cite[Thm.3.4]{HR15}, we have the following structure theorem for the augmentation varieties $\mr{Aug}_m(T,\epsilon_{\rho_L},\rho_R;k)$, but with a different proof:

\begin{theorem}[See Theorem \ref{thm:augmentation varieties for Legendrian tangles}]
Let $(T,\mu)$ be any Legendrian tangle, with $B$ base points placed on $T$ so that \emph{each right cusp is marked}.
Fix $m$-graded normal rulings $\rho_L,\rho_R$ of $T_L, T_R$ respectively. Fix $\epsilon_L\in\mc{O}_m(\rho_L;k)$. Then
there's a decomposition of augmentation varieties into disjoint union of subvarieties
\begin{eqnarray*}
\mr{Aug}_m(T,\epsilon_L,\rho_R;k)=\sqcup_{\rho}\mr{Aug}_m^{\rho}(T,\epsilon_L,\rho_R;k)
\end{eqnarray*}
where $\rho$ runs over all $m$-graded normal rulings of $T$ such that $\rho|_{T_L}=\rho_L,\rho|_{T_R}=\rho_R$. Moreover,
\begin{eqnarray*}
\mr{Aug}_m^{\rho}(T,\epsilon_L,\rho_R;k)\cong (k^*)^{-\chi(\rho)+B}\times k^{r(\rho)}.
\end{eqnarray*}
\end{theorem}

In the end, we consider \emph{generalized normal rulings} and introduce \emph{generalized Ruling polynomials} for Legendrian tangles, partly suggested by the proofs of the main theorems above. Moreover, the previous main results admit a direct generalization to this setting, and essentially the same arguments apply. It turns out that there're 2 slightly different approaches to do so. The generalized normal rulings as in Definition \ref{def:generalized normal rulings_2}, were firstly introduced in \cite{LR12}. See also \cite{LR12} for some applications of generalized normal rulings to the study of Legendrian knots in $J^1S^1$.

\subsection*{Organization}
In Section \ref{sec:background}, We review the basic backgrounds in Legendrian knot theory. In Section \ref{sec:Ruling polynomials for tangles}, we discuss the basics of Legendrian tangles, define the normal rulings and Ruling polynomials for Legendrian tangles. Then we prove that the Ruling polynomials are Legendrian isotopy invariants and satisfy the composition axiom, the key axiom of a TQFT (Theorem \ref{thm:invariance and composition of Ruling polynomials}). In Section \ref{sec:LCH DGA for tangles}, we discuss the LCH DGAs for any Legendrian tangles (not necessarily with nearly plat fronts) via the front projection and resolution construction. In Section \ref{sec:Augmentations for tangles}, we define augmentation varieties and augmentation numbers for Legendrian tangles (with fixed boundary conditions). After that, we prove an algorithm to compute the augmentation numbers. The invariance of augmentation numbers and the main Theorem \ref{thm:counting for tangles} then follow quickly. The key ingredients of the algorithm are the structures of the augmentation varieties associated to the trivial Legendrian tangle of $n$ parallel strands and elementary Legendrian tangles. The former is a result about the Barannikov normal forms (Lemma \ref{lem:Barannikov normal form}). The latter (Lemma \ref{lem:augmentation varieties for elementary tangles}) is dealt with in Section \ref{sec:Augmentations for elementary Legendrian tangles}, where we also prove a stronger result, which leads to a structure theorem (Theorem \ref{thm:augmentation varieties for Legendrian tangles}) for the augmentation varieties associated to any Legendrian tangles. In Section \ref{sec:Ruling vs aug:general case}, we consider generalized normal rulings and generalized Ruling polynomials for Legendrian tangles, and give 2 possible ways to generalize the main results: Theorem \ref{thm:invariance and composition of Ruling polynomials}, Theorem \ref{thm:counting for tangles} and Theorem \ref{thm:augmentation varieties for Legendrian tangles}.

\subsection*{Acknowledgements}
First of all, I would like to express my deep gratitude to my advisor, Prof. Vivek Shende for numerous invaluable discussions and suggestions throughout this project. Moreover, I'm very grateful to the American Institute of Mathematics for sponsoring a 2017 SQUARE meeting ``Sheaf theory and Legendrian knots", where part of this article was improved, to the other participants of the meeting: Roger Casals, Lenhard Ng, Dan Rutherford, Vivek Shende, especially to Dan Rutherford, for valuable suggestions and comments in helping improving the article. I'm also grateful to my co-advisor, Prof. Richard E.Borcherds for allowing me to try problems according to my own interest. Finally, I would like to thank Kevin Donoghue for useful conversations, and thank Benson Au for teaching me how to use Inkscape.

\addtocontents{toc}{\protect\setcounter{tocdepth}{2}}

\section{Background}\label{sec:background}
\subsection{Legendrian knot basics}
\subsubsection{Contact basics}\label{subsubsec:contact_basics}
Take the standard contact three-space $\mathbb{R}_{x,y,z}^3=J^1(\mathbb{R}_x)=T^*\mathbb{R}_x\times\mathbb{R}_z$ with contact form $\alpha=dz-ydx$. The Reeb vector field of $\alpha$ is then $R_{\alpha}=\partial_z$. We consider a (one-dimensional) Legendrian submanifold (termed as \emph{knot} or \emph{link}) $\Lambda$ in this three space $\mathbb{R}^3$. The \emph{front} and \emph{Lagrangian projections} of $\Lambda$ are $\pi_{xz}(\Lambda)$ and $\pi_{xy}(\Lambda)$ respectively, with the obvious projections $\pi_{xz}:\mathbb{R}_{x,y,z}^3\rightarrow\mathbb{R}_{x,z}^2$ and $\pi_{xy}:\mathbb{R}_{x,y,z}^3\rightarrow\mathbb{R}_{x,y}^2$.

\subsubsection{Front diagrams}\label{subsubsec:fronts}
We will always \emph{assume} the Legendrian link $\Lambda\subset \mathbb{R}^3$ is in a generic position inside its Legendrian isotopy class. So, the front projection $\pi_{xz}(\Lambda)$ gives a \emph{front diagram} (i.e. an immersion of a finite union of circles into $\mathbb{R}_{xz}^2$ away from finitely many points (cusps) having no vertical tangent, which is also an embedding away from finitely many points (cusps and transversal crossings)). The significance of front diagrams is that, any Legendrian link is uniquely determined by its front projection. That is, the $y$-coordinate can be recovered from the $x$ and $z$-coordinates of the front projection as the slope, via the Legendrian condition $dz-ydx=0\Rightarrow y=dz/dx$. In other words, in passing to the front projection, we loss no information. Note also that, near each crossing of a front diagram, the strand of the lesser slope is always the over-strand.

Given a front diagram, the \emph{strands} of $\pi_{xz}(\Lambda)$ are the maximally immersed connected submanifolds, the \emph{arcs} of $\pi_{xz}(\Lambda)$ are the maximally embedded connected submanifolds and the \emph{regions} are the maximal connected components of the complement of $\pi_{xz}(\Lambda)$ in $\mb{R}_{xz}^2$.

We say a front diagram in $\mathbb{R}_{x,z}^2$ is \emph{plat} if the crossings have distinct $x$-coordinates, all the left cusps have the same $x$-coordinate and likewise for the right cusps. We say a front diagram is \emph{nearly plat}, if it's a perturbation of a front diagram, so that the crossings and cusps all have different $x$-coordinates. We can always make the front diagram $\pi_{xz}(\Lambda)$ (nearly) plat by smooth isotopies and Legendrian Reidemeister II moves (see FIGURE \ref{fig:LR}).

We say a front diagram in $\mb{R}_{xz}^2$ is \emph{simple} if it's smooth isotopic to a front whose right cusps have the same $x$-coordinate. For example, any (nearly) plat front diagram is simple.

\subsubsection{Resolution construction}\label{subsubsec:resolution_construction}
In this article, we will use both the front and Lagrangian projections. Hence, it's often necessary to translate between the 2 projections in some simple way. This can be realized by the resolution construction \cite[Prop.2.2]{Ng03}. Given the front diagram $\pi_{xz}(\Lambda)$, we can obtain the Lagrangian projection $\pi_{xy}(\Lambda')$ of a link $\Lambda'$ Legendrian isotopic to $\Lambda$, via a resolution procedure as in FIGURE \ref{fig:Res}. We say that $\Lambda'$ is obtained from $\Lambda$ by \emph{resolution construction}. Note that the same conclusion applies to Legendrian tangles (see Section \ref{subsec:Legendrian tangles} for the definition.)

\begin{figure}[!htbp]
\begin{center}
\minipage{0.8\textwidth}
\includegraphics[width=\linewidth,height=0.5in]{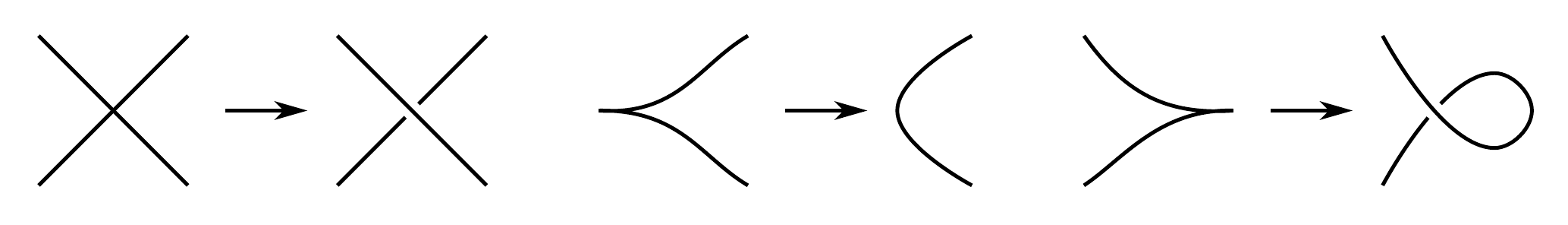}
\endminipage\hfill
\end{center}
\caption{Resolving a front into the Lagrangian projection of a Legendrian isotopic link/tangle.}
\label{fig:Res}
\end{figure}

\subsection{Legendrian Reidemeister moves and Classical invariants}
\subsubsection{Legendrian Reidemeister moves}\label{subsubsec:LR moves}
It's well known that any smooth knot can be represented by a knot diagram, and any 2 knot diagrams represent smoothly isotopic knots if and only if they differ by smooth isotopy and a finite sequence of 3 types of topological Reidemeister moves. There's an analogue for Legendrian knots via front diagrams. That is, 2 front diagrams in $\mb{R}_{xz}^2$ represent the same Legendrian isotopy class of Legendrian knots in $\mb{R}_{x,y,z}^3$ if and only if they differ by a finite sequence of smooth isotopies and the following \emph{Legendrian Reidemeister moves} of 3 types (\cite{Swi92}):

\begin{figure}[!htbp]
\begin{center}
\minipage{0.3\textwidth}
\includegraphics[width = \linewidth]{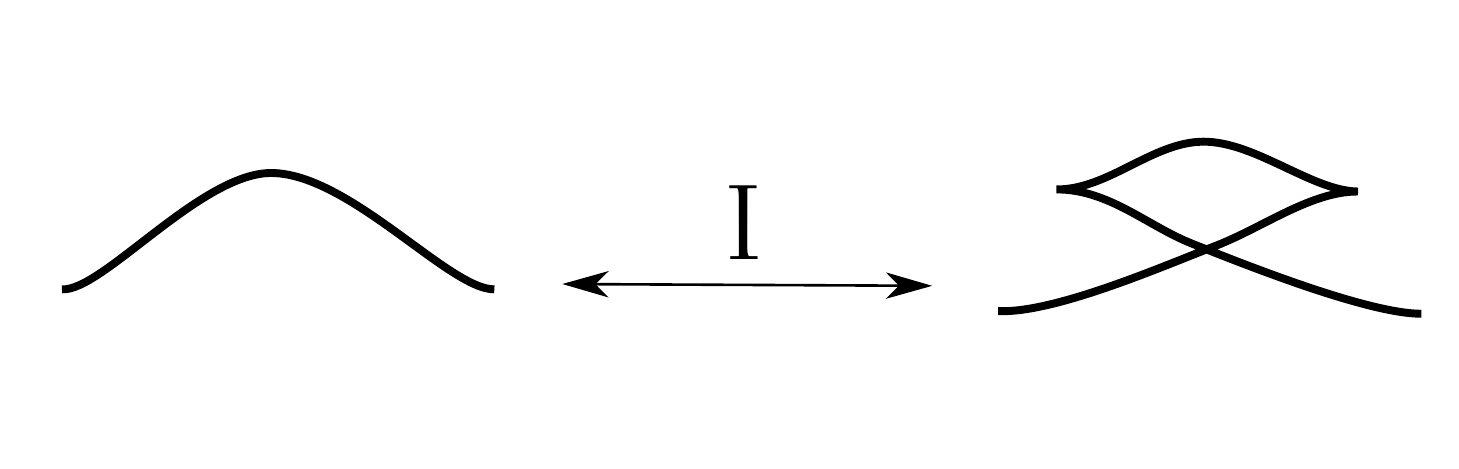}
\endminipage\hfill
\minipage{0.3\textwidth}
\includegraphics[width=\linewidth]{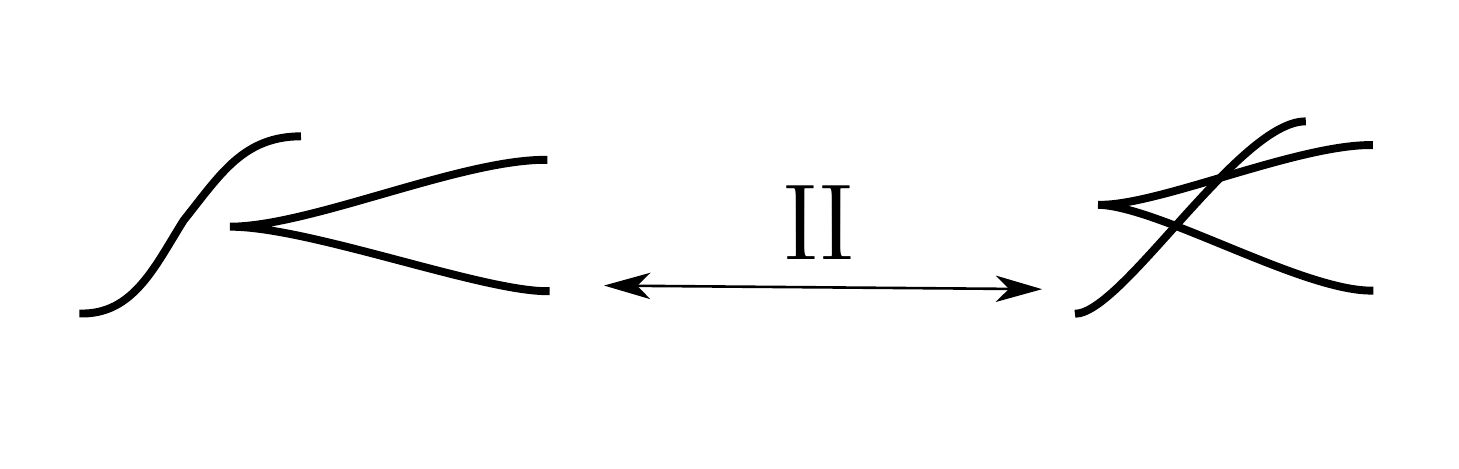}
\endminipage\hfill
\minipage{0.3\textwidth}
\includegraphics[width=\linewidth]{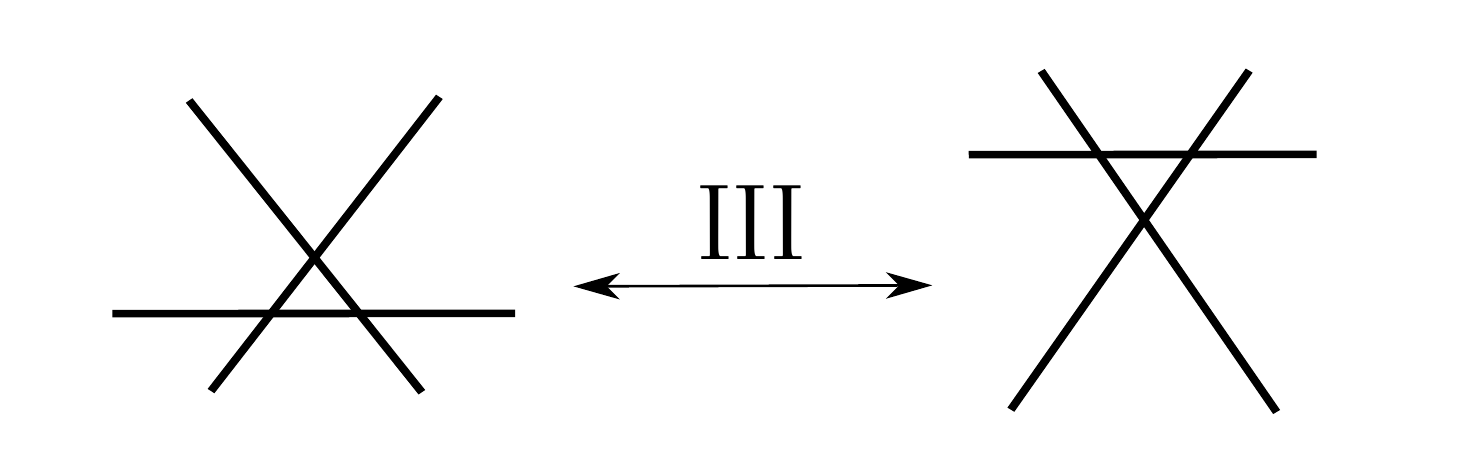}
\endminipage\hfill
\end{center}
\caption{The 3 types of Legendrian Reidemeister moves relating Legendrian-isotopic fronts. Reflections of these moves along the coordinate axes are also allowed.}
\label{fig:LR}
\end{figure}

\subsubsection{Topological knot type}\label{subsubsec:topological_knot_type}
The \emph{topological knot type} of a Legendrian knot is the smooth isotopy class of its underlying smooth knot. Clearly, this defines a Legendrian isotopy invariant. As a consequence, all topological knot invariants (\cite{Jon85, HOMFLY85, Wit89} as well as their ``categorified" versions (\cite{Kho99, Kho07, KR08}) are automatically Legendrian isotopy invariants.

\subsubsection{Thurston-Bennequin number}\label{subsubsec:tb}
Given an oriented Legendrian knot $\Lambda$ in $(\mb{R}_{x,y,z}^3,\alpha=dz-ydx)$, the \emph{Thurston-Bennequin number} (denoted by $tb(\Lambda)$) measures the twisting of the oriented contact plane field along the knot. It can be defined as the linking number of the Legendrian knot and its push-off along the Reeb direction $R_{\alpha}=\partial_z$, that is, $tb(\Lambda):=lk(\Lambda,\Lambda+\epsilon z)$. The geometric definition makes it automatically a Legendrian invariant. On the other hand, project the Legendrian knot down to the front plane $\mb{R}_{xz}^2$, the number can be computed via the front diagram: $tb(\Lambda)=wr(\pi_{xz}(\Lambda))-c(\pi_{xz}(\Lambda))$, where $wr$ is the writhe number and $c$ is the number of right cusps. It's then also easy to check the Legendrian isotopy invariance via Legendrian Reidemeister moves.

\subsubsection{Rotation number}\label{subsubsec:r}
Given an oriented connected Legendrian knot $\Lambda$ in $\mb{R}^3$, the \emph{rotation number} $r(\Lambda)$ is the obstruction to extending the tangent vector field of $\Lambda$ to a nonzero section of the contact plane field over a Seifert surface (an embedded compact oriented surface) bounding $\Lambda$. It can also be computed via the front diagram:
$r(L)=1/2(U(\pi_{xz}(\Lambda)-D(\pi_{xz}(\Lambda))))$, where $U$ (resp. $D$) is the number of \emph{up (resp. down) cusps} ($=$ a cusp near which the orientation of $\pi_{xz}(\Lambda)$ goes up (resp. down)). This is another example where the Legendrian isotopy invariance can be checked via Legendrian Reidemeister moves. By definition, the rotation number depends via a sign on the orientation. For an oriented multi-component Legendrian knot $\Lambda$, we usually define its rotation number $tb(\Lambda)$ as the $gcd$ of the rotation numbers of its components.

\subsubsection{Maslov potential}\label{subsubsec:Maslov}
Given a Legendrian knot $\Lambda$ with front diagram $\pi_{xz}(\Lambda)$. Let $r=|r(\Lambda)|$ and $n$ be a nonnegative integer. A $\mb{Z}/n\mb{Z}$-valued \emph{Maslov potential} of $\pi_{xz}(\Lambda)$ is a map
\begin{equation*}
\mu:\{\text{strands of }\pi_{xz}(\Lambda)\}\rightarrow\mathbb{Z}/n\mb{Z}
\end{equation*}
such that near any cusp, have $\mu(\text{upper strand})=\mu(\text{lower strand})+1$. Such a Maslov potential exists if and only if $2r$ is a multiple of $n$. In particular, the existence of a $\mb{Z}$-valued Maslov potential implies that every component of $\Lambda$ has rotation number 0. We will often fix for $\Lambda$ a $\mb{Z}/2r$-valued Maslov potential $\mu$.

We usually view the topological knot type, the Thurston-Bennequin number and the rotation number as the \emph{classical invariants} of (oriented) Legendrian knots. It turns out, they are not sufficient to classify Legendrian knots up to Legendrian isotopy. In \cite{Che02}, a pair of Legendrian knots (they represent the same class $5_2$ in the classification of topological knots) having the same classical invariants, were shown to be distinguished by a new Legendrian isotopy invariants, the LCH DGA (or Chekanov-Eliashberg DGA) (See Section \ref{subsec:LCH DGA} below).

\subsection{LCH differential graded algebras}\label{subsec:LCH DGA}
\subsubsection{LCH DGA via Lagrangian projection}
Here we recall the Legendrian contact homology differential graded algebra for Legendrian links in $\mathbb{R}^3$ \cite{ENS02}. The version of DGAs we need will also allow an arbitrary number of base points placed on the Legendrian links \cite{NR13,NRSSZ15}. The construction is naturally formulated via Lagrangian projection.

\noindent\emph{Initial data:} Let $\Lambda$ be an oriented Legendrian link in $\mathbb{R}_{x,y,z}^3$, with rotation number $r(\Lambda)$. Take $r=|r(\Lambda)|$ and fix a $\mb{Z}/2r$-valued Maslov potential $\mu$ of $\pi_{xz}(\Lambda)$. Let $*_1,\ldots,*_B$ be the base points placed on $\Lambda$, avoiding the crossings of the Lagrangian projection $\pi_{xy}(\Lambda)$, such that each component of $\Lambda$ contains at least one base point. Denote by $\{a_1,\ldots,a_R\}$ the set of crossings of $\pi_{xy}(\Lambda)$, corresponding to the Reeb chords of the Reeb vector field $R_{\alpha}=\partial_z$.

The $\mathbb{Z}/2r$-graded LCH DGA $\mathcal{A}=\mathcal{A}(\Lambda,\mu,*_1,\ldots,*_B)$ is then defined as follows:

\noindent\emph{As an algebra:} $\mathcal{A}$ is the noncommutative associative unital algebra $\mathbb{Z}[t_1^{\pm 1},\ldots,t_B^{\pm1}]<a_1,\ldots,a_R>$ over the commutative ring $\mathbb{Z}[t_i^{\pm 1},1\leq i\leq B]$, freely generated by $a_j,1\leq j\leq R$. The generators $t_i, t_i^{-1}$ can be regarded as information encoded at the base point $*_i$.

\noindent\emph{The grading:} The algebra $\mathcal{A}$ is assigned a $\mathbb{Z}/2r$-grading via $|x\cdot y|=|x|+|y|, |t_i|=|t_i^{-1}|=0$ and $|a_j|$ is defined as follows. The 2 endpoints of the Reeb chord $a_j$ belong to 2 distinct strands of the front projection $\pi_{xz}(\Lambda)$. Near the upper (lower) endpoint  of $a_j$, the overstrand (understrand) can be parameterized as $x\rightarrow (x,z=f_u(x))$(resp. $(x,z=f_l(x))$), and $f_u'(x(a_j))=f_l'(x(a_j))$. By the generic assumption of $\Lambda$, $f_u-f_l$ attains either a local maximum (or minimum) at $a_j$, accordingly we define
$|a_j|=\mu(\text{over-strand})-\mu(\text{under-strand})+$ either $0($or$-1)$.

\noindent\emph{The differential $\partial$:} To define the differential $\partial$ on $\mathcal{A}$, we firstly impose the \emph{Leibniz Rule} $\partial(x\cdot y)=(\partial x)\cdot y+(-1)^{|x|}\cdot\partial y$ and $\partial(t_i)=\partial(t_i^{-1})=0$. It then suffices to define the differential $\partial a_j$ for each Reeb chord. Intuitively, $\partial a_j$ is a weighted count of boundary-punctured holomorphic disks in the symplectization $(\mathbb{R}_{\tau}\times\mathbb{R}^3_{x,y,z},d(e^{\tau}\alpha)$ with boundary along the Lagrangian $\mathbb{R}_{\tau}\times\Lambda$, with one positive boundary puncture limiting to the Reeb chord $a_j$ at $+\infty$, and several negative boundary punctures limiting to some Reeb chords at $-\infty$.

In our case, we can describe the differential combinatorially. Given Reeb chords $a=a_j$ and $b_1,\ldots,b_n$ for some $n\geq 0$. Let $D_n^2=D^2-\{p,q_1,\ldots,q_n\}$ be a fixed oriented disk with $n+1$ boundary punctures $p,q_1,\ldots,q_n$, arranged in a counterclockwise order.
\begin{definition}[Admissible disks]\label{def:admissible_disks}
Define the moduli space $\Delta(a;b_1,\ldots,b_n)$ to be the space of \emph{admissible disks} $u$ of $\pi_{xy}(\Lambda)$ up to re-parametrization, that is,
\begin{itemize}
\item
$u:(D_n^2,\partial D_n^2)\rightarrow (\mathbb{R}_{x,y}^2,\pi_{xy}(\Lambda))$ is a smooth orientation-preserving immersion, extends continuously to $D^2$;

\item
$u(p)=a, u(q_i)=b_i (1\leq i\leq n)$ and $u$ sends a neighborhood of $p$ (resp. $q_i$) in $D^2$ to a single quadrant of $a$ (resp. $b_i$) with positive (resp. negative) \emph{Reeb sign} (see below).
\end{itemize}
\noindent{}\emph{Reeb signs:} Near a crossing of the Lagrangian projection $\pi_{xy}(\Lambda)$, the 2 quadrants lying in the counterclockwise (resp. clockwise) direction of the over-strand are assigned \emph{positive} (resp. \emph{negative}) \emph{Reeb signs}. See Figure \ref{fig:Orientation signs} (left).
\end{definition}

For each $u\in \Delta(a;b_1,\ldots,b_n)$, walk along $u(\partial D^2)$ starting from $a$, we encounter a sequence $s_1,\ldots,s_N(N\geq n)$ of crossings (excluding $a$) and base points of $\pi_{xy}(\Lambda)$. We then define the weight $w(u)$ of $u$ as follows
\begin{definition}\label{def:weight for disks}
$w(u):=s(u)w(s_1)\ldots w(s_N)$, where
\begin{enumerate}[label=(\roman*)]
\item
$w(s_i)=b_j$ if $s_i$ is the crossing $b_j$.
\item
$w(s_i)=t_j ($resp. $t_j^{-1})$ if $s_i$ is the base point $*_j$, and the boundary orientation of $u(\partial D^2)$ agrees (resp. disagrees) with the orientation of $\Lambda$ near $*_j$.
\item
$s(u)$ is the product of the \emph{orientation signs} (see below) of the quadrants near $a$ and $b_1, \ldots$, $b_n$ occupied by $u$.
\end{enumerate}
\noindent\emph{Orientation signs:} We will use the same convention as \cite{NRSSZ15}. That is, at each crossing $a$ such that $|a|$ is even, we assign negative \emph{orientation signs} to the 2 quadrants that lie on any chosen side of the under-strand of $a$; We assign positive \emph{orientation signs} to all the other quadrants. See Figure \ref{fig:Orientation signs} (right).
\end{definition}

\begin{figure}[!htbp]
\begin{center}
\minipage{0.5\textwidth}
\includegraphics[width = \linewidth, height=0.8in]{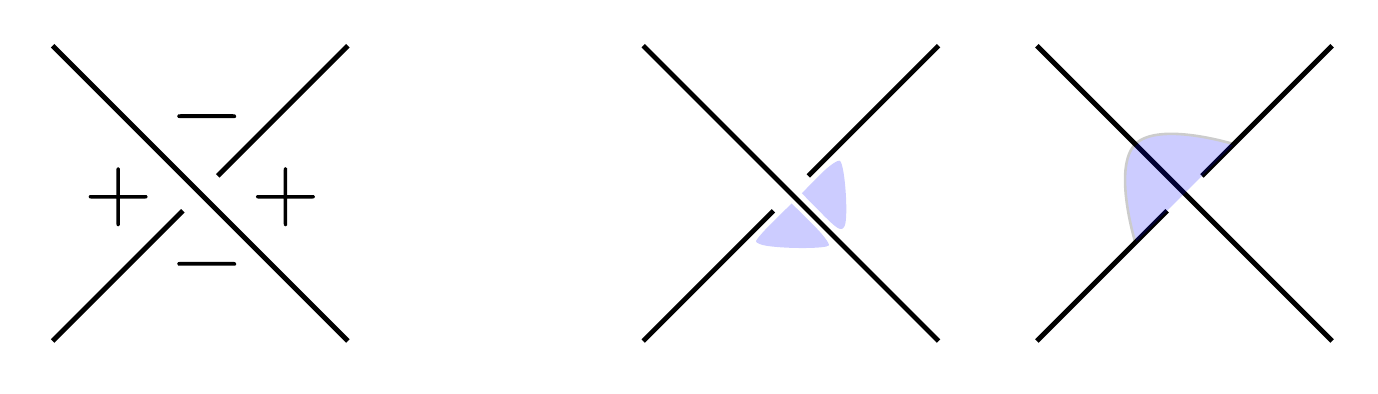}
\endminipage\hfill
\end{center}
\caption{Left: the Reeb signs of the quadrants at a crossing in a Lagrangrian projection. Right: the two possible choices of orientation signs for the quadrants at a crossing of even degree in a Lagrangian projection. The shaded quadrants have negative orientation signs and the unshaded quadrants have positive orientation signs. At a crossing of odd degree, all the four quadrants have positive orientation signs.}
\label{fig:Orientation signs}
\end{figure}

Now we can define the differential of $a=a_j$:
\begin{equation}
\partial a=\sum_{n,b_1,\dots,b_n}\sum_{u\in\Delta(a;b_1,\ldots,b_n)}w(u)
\end{equation}

\begin{theorem}[\cite{Che02,ENS02}]
$(\mathcal{A},\partial)$ is a $\mathbb{Z}/2r$-graded DGA with $\mathrm{deg}(\partial)=-1$.
\end{theorem}

\subsubsection{Invariance of LCH DGA}
We have seen that, the definition of the LCH DGA $\mc{A}(\Lambda)$ associated to a Legendrian link $\Lambda$ depends on several choices: a specific choice of the representative of $\Lambda$ inside its Legendrian isotopy class, and a choice of base points. Here we review that the LCH DGA $\mc{A}(\Lambda)$ is a Legendrian isotopy invariant, up to a \emph{stable isomorphism}, in particular, up to homotopy equivalence of $\mb{Z}/2r$-graded DGAs.

\begin{definition}\label{def:stablization}
An \emph{(algebraic) stabilization} of a $\mb{Z}/2r$-graded DGA $(\mc{A},\partial)$ is a $\mb{Z}/2r$-graded DGA $(S(\mc{A}),\partial')$ obtained by adding 2 new free generators $e$ and $f$, with $|e|=|f|+1$, such that $\partial'|_{\mc{A}}=\partial$ and $\partial' e=f, \partial'f=0$. Two $\mb{Z}/2r$-graded DGAs $(\mc{A},\partial)$ and $(\mc{A}',\partial')$ are \emph{stable isomorphic}, if they are isomorphic as $\mb{Z}/2r$-graded DGAs, after possibly stabilizing each finitely many times.
\end{definition}

\begin{theorem}[{\cite[Thm.2.20]{NR13},\cite[Thm.3.10]{ENS02}}]\label{thm:stable isom}
The isomorphism class of the DGA $(\mc{A}(\Lambda),\partial)$ is independent of the locations of the base points on each connected component of $\Lambda$. The stable isomorphism class of $(\mc{A}(\Lambda),\partial)$ is invariant under Legendrian isotopy of $\Lambda$.
\end{theorem}

\subsubsection{LCH DGA via front projection}\label{subsubsec:DGA via fro}
\emph{Assume} the front projection $\Lambda$ is simple (see Section \ref{subsubsec:fronts} for the definition). Then the LCH DGA also admits a simple front projection description.

The resolution construction of $\Lambda$ gives a Legendrian isotopic link $\Lambda'=\mathrm{Res}(\Lambda)$, whose Reeb chords are in one-to-one correspondence with the crossings and right cusps of $\pi_{xz}(\Lambda)$. We will denote by $\mathcal{A}(\Lambda_{xz},\mu,*_1,\ldots *_B;k)$ the LCH DGA associated to $\mathrm{Res}(\Lambda)$. Denote by $\{a_1,\ldots a_p\}$ (resp. $\{c_1,\ldots c_q\}$) the set of crossings (resp., of right cusps) of $\pi_{xz}(\Lambda)$. Under the correspondence, the algebra is generated over $\mathbb{Z}[t_i^{\pm 1}, 1\leq i\leq B]$ by $\{a_k,1\leq k\leq p, c_k,1\leq k\leq q\}$. The grading is given by: $|t_k^{\pm}|=0, |a_k|=\mu(\text{over-strand})-\mu(\text{under-strand})$ and $|c_k|=1$. One can also translate the definition of the differential for $\mathrm{Res}(\Lambda)$ into the front projection $\pi_{xz}(\Lambda)$. The definition uses the same formula by ``counting" the disks in $\pi_{xz}(\Lambda)$ plus the additional ``invisible disks", one for each right cusp. An ``invisible disk" (See FIGURE \ref{fig:Res}, the last picture) corresponds to a disk with one unique corner on its left at the crossing of $\pi_{xy}(\mathrm{Res}(\Lambda))$ corresponding to the right cusp of $\pi_{xz}(\Lambda)$.

\section{Ruling polynomials for Legendrian tangles}\label{sec:Ruling polynomials for tangles}
\subsection{Legendrian tangles}\label{subsec:Legendrian tangles}
Fix $U=(x_L,x_R)$ to be a open interval in $\mathbb{R}_x$ ($-\infty\leq x_L<x_R\leq\infty$), so the standard contact form $\alpha=dz-ydx$ induces a standard contact structure on $J^1U=U\times\mb{R}_{y,z}^2$. A \emph{Legendrian tangle} $T$ is a Legendrian submanifold in $J^1U$ transverse to the boundary $\partial J^1(\overline{U})$. Typical examples of Legendrian tangles can be obtained from a Legendrian link front by removing the parts outside of a vertical strip in $\mb{R}_{xz}^2$.

\begin{remark}\label{rem:fronts for Legendrian tangles}
In Section \ref{subsubsec:fronts}, notice that the same concepts (\emph{front diagrams, strands, arcs, regions, etc.}) can be introduced for any Legendrian tangle $T$ in $J^1U$. We only have to replace $\Lambda$ by $T$ and $\mb{R}_{xz}^2$ by $U\times\mb{R}_z$ there. Similarly, in Section \ref{subsubsec:Maslov}, the same procedure applies to define \emph{Maslov potentials} for $T$.
\end{remark}

As usual, we will assume $T$ has a generic front projection.\footnote{From now on, we will make no distinction between the Legendrian tangle $T$ and its front projection $\pi_{xz}(T)$.} We equip $T$ with a $\mb{Z}/2r$-valued Maslov potential $\mu$ for some fixed $r\geq 0$. Denote by $n_L$ (resp. $n_R$) the number of left (resp. right) end-points on $T$.

We say $2$ Legendrian tangles in $J^1U$ are \emph{Legendrian isotopic} if there's an isotopy between them along Legendrian tangles in $J^1U$. Note that during the Legendrian isotopy, we require the \emph{ordering} via $z$-coordinates of the end-points is preserved. That is, for two (say, left) end-points $p_1, p_2$, they necessarily have the common $x$-coordinate $x_L$, take any path $\gamma$ in $\partial J^1(\overline{U})$ from $p_2$ to $p_1$, then we say $p_1>p_2$ if $z(p_1)-z(p_2)=\int_{\gamma}\alpha>0$. Then, similar to the case of Legendrian links, two (generic) Legendrian tangle fronts are Legendrian isotopic \emph{if and only if} they differ by a finite sequence of smooth isotopies (preserving the ordering of the end-points) and \emph{Legendrian Reidemeister moves} of the 3 types (see Figure \ref{fig:LR}).

\subsection{Normal Rulings and Ruling polynomials}
Similar to Legendrian knots, we can introduce the notion of $m$-graded normal rulings and Ruling polynomials for any Legendrian tangles. Given a Legendrian tangle $T$, with $\mathbb{Z}/2r$-valued Maslov potential $\mu$ for some fixed $r\geq 0$. Fix a nonnegative integer $m$ dividing $2r$.

\noindent{}\emph{Assume} that the numbers $n_L, n_R$ of the left endpoints and right endpoints of $T$ are both \emph{even}. For example, any Legendrian tangle obtained from cutting a Legendrian link front along 2 vertical lines, satisfies this assumption.

Recall that in Remark \ref{rem:fronts for Legendrian tangles}, we have introduced the notions of \emph{arcs}, \emph{crossings}, \emph{cusps} and \emph{regions} of the front $\pi_{xz}(T)$. In particular, the front diagram is divided into arcs, crossings and cusps.
For example, an arc begins at a cusp, a crossing or an end-point, going from left to right, and ends at another cusp, crossing or end-point, meeting no cusp or crossing in-between. Given a crossing $a$ of the front $T$, its degree is given by $|a|:=\mu(\text{over-strand})-\mu(\text{under-strand})$.

\begin{definition}\label{def:eye}
We say an embedded (closed) disk of $\overline{U}\times\mb{R}_{z}$, is an \emph{eye} of the front $T$, if it is the union
of (the closures of) some \emph{regions}, such that the boundary of the disk in $U\times \mb{R}_z$, being the union of arcs, crossings and cusps, consists of 2 \emph{paths}, starting at the same left cusp or a pair of left end-points, going from left to right through arcs and crossings, meeting no cusps in-between, and ending at the same right cusp or a pair of right end-points.
\end{definition}

\begin{definition}\label{def:normal_ruling}
A \emph{$m$-graded normal Ruling} $\rho$ of $(T,\mu)$ is a \emph{partition} of the set of arcs of the front $T$ into the boundaries in $U\times \mb{R}_z$ of \emph{eyes} (say $e_1,\ldots, e_n$), or in other words,
\begin{equation*}
\sqcup\text{ arcs of }T=\sqcup_{i=1}^{n}(\partial e_i\setminus\{\text{crossings, cusps}\})\cap U\times\mb{R}_z,
\end{equation*}
and such that the following conditions are satisfied:
\begin{enumerate}[label=(\arabic*).]
\item
If some eye $e_i$ starts at a pair of left end-points (resp. ends at a pair of right end-points), we require $\mu(\text{upper-end-point})=\mu(\text{lower-end-point})+1 (\mr{mod} m)$.
\item
Call a crossing $a$ a \emph{switch}, if it's contained in the boundary of some eye $e_i$. In this case, we require $|a|=0 (\mr{mod} m)$.
\item
Each switch $a$ is clearly contained in exactly 2 eyes, say $e_i, e_j$. We require the relative positions of $e_i, e_j$ near $a$ to be in one of the 3 situations in Figure \ref{fig:NR}(top row).
\end{enumerate}

\end{definition}

\begin{figure}[!htbp]
\begin{center}
\minipage{0.8\textwidth}
\includegraphics[width=\linewidth, height=2in]{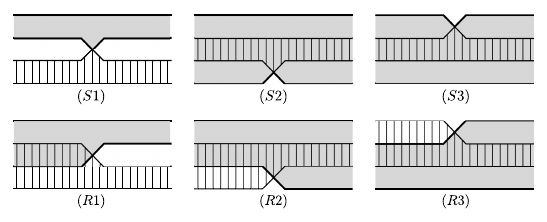}
\endminipage\hfill
\end{center}
\caption{Top row: The behavior (of the 2 eyes $e_i, e_j$) of a $m$-graded normal ruling $\rho$ at a \emph{switch} (where $e_i$ and $e_j$ are the dashed and shadowed regions respectively), where the crossings are required to have degree 0 $(\mr{mod} m)$. Bottom row: The behavior (of the 2 eyes $e_i, e_j$) of $\rho$ at a return. Three more figures omitted: The 3 types of departures obtained by reflecting each of $(R1)$-$(R3)$ with respect to a vertical axis.}
\label{fig:NR}
\end{figure}

\begin{definition}\label{def:switches and returns}
Given a Legendrian tangle $(T,\mu)$, let $\rho$ be a $m$-graded normal ruling of $(T,\mu)$, and let $a$ be a crossing. Then, $a$ is called a \emph{return} if the behavior of $\rho$ at $a$ is as in Figure \ref{fig:NR}(bottom row). $a$ is called a \emph{departure} if the behavior of $\rho$ at $a$ looks like one of the three pictures obtained by reflecting each of $(R1)-(R3)$ in Figure \ref{fig:NR}(bottom row) with respect to a vertical axis. Moreover, returns (resp. departures) of degree 0 modulo $m$ are called \emph{$m$-graded returns} (resp. \emph{$m$-graded departures}) of $\rho$.

\noindent{}\emph{Define} $s(\rho)$ (resp. $d(\rho)$) to be the number of switches (resp. $m$-graded departures) of $\rho$.\\ \noindent{}\emph{Define} $r(\rho)$ to be the number of $m$-graded returns of $\rho$ if $m\neq 1$, and the number of $m$-graded returns and right cusps if $m=1$.
\end{definition}

\begin{remark}\label{rem:ruling by switches}
If we fix the pairing $\rho_L$ of left end-points, a $m$-graded normal Ruling $\rho$ determines and is determined by a subset, denoted by the same symbol $\rho$, of the switches in the set of crossings of $T$. In this case, we will usually make no distinction between a $m$-graded normal Ruling and its set of switches.
\end{remark}

\begin{definition}
Given a $m$-graded normal Ruling $\rho$ of a Legendrian tangle $(T,\mu)$, denote by $e_1,\ldots,e_n$ the eyes in $J^1(\overline{U})$ defined by $\rho$.  The \emph{filling surface} $S_{\rho}$ of $\rho$ is the the disjoint union $\sqcup_{i=1}^ne_i$ of the eyes, glued along the switches via half-twisted strips. This is a compact surface possibly with boundary. See FIGURE \ref{fig:filling_surface} for an example.
\end{definition}

\begin{figure}[!htbp]
\begin{center}
\minipage{0.8\textwidth}
\includegraphics[width = \linewidth, height=1.2in]{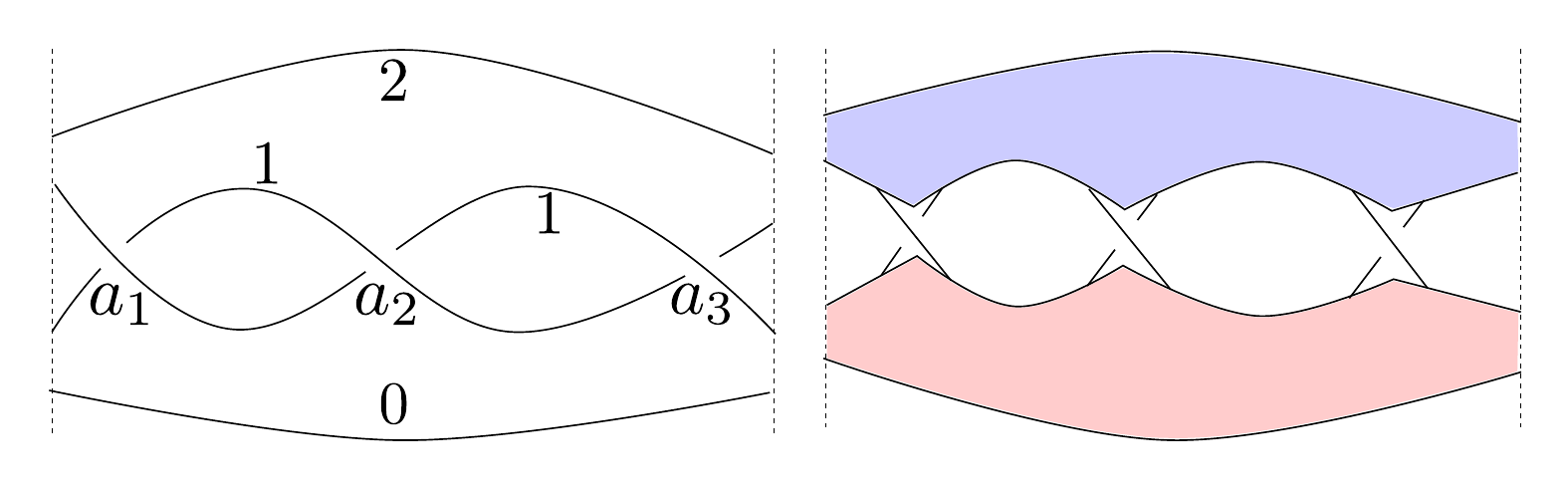}
\endminipage\hfill
\end{center}
\caption{Left: a Legendrian tangle front $T$ with 3 crossings $a_1, a_2, a_3$, the numbers indicate the values of the Maslov potential $\mu$ on each of the 4 strands. Right: the filling surface for a normal ruling of $T$ by gluing the 2 eyes along the 3 switches via half-twisted strips.}
\label{fig:filling_surface}
\end{figure}

Let $T_L$ (resp. $T_R$) be the left (resp. right) pieces $T$ near the left (resp. right) boundary. It's clear that any $m$-graded normal Ruling $\rho$ of $T$ restricts to a $m$-graded normal Ruling of the left piece $T_L$ (resp. of the right piece $T_R$), denoted by \emph{$r_L(\rho)$ or $\rho|_{T_L}$ (resp. $r_R(\rho)$ or $\rho|_{T_R}$)}.
\begin{definition}\label{def:Ruling polynomial}
Fix a $m$-graded normal Ruling $\rho_L$ (resp. $\rho_R$) of $T_L$ (resp. $T_R$). We define a Laurent polynomial $<\rho_L|R_{T,\mu}^m(z)|\rho_R>=<\rho_L|R_{T}^m(z)|\rho_R>$ in $\mb{Z}[z,z^{-1}]$ by
\begin{equation}
<\rho_L|R_{T}^m(z)|\rho_R>:=\sum_{\rho \text{: $r_L(\rho)=\rho_L,r_R(\rho)=\rho_R$}}z^{-\chi(\rho)}
\end{equation}
where the sum is over all $m$-graded normal Rulings $\rho$ such that $r_L(\rho)=\rho_L, r_R(\rho)=\rho_R$. $\chi(\rho)$ is called the \emph{Euler characteristic} of $\rho$ and defined by
\begin{equation}
\chi(\rho):=\chi(S_{\rho})-\chi(S_{\rho}|_{x=x_R}).
\end{equation}
where $x_R$ is the right endpoint of the open interval $U=(x_L,x_R)$ and $\chi(S_{\rho})$ (resp. $\chi(S_{\rho}|_{x=x_R})$) is the usual Euler characteristic of $S_{\rho}$ (resp. $S_{\rho}|_{x=x_R}$). Equivalently, $\chi(\rho)=\chi_c(S_{\rho}|_{x_L\leq x<x_R})$ is the Euler characteristic with compact support of $S_{\rho}|_{x_L\leq x<x_R}$. Also, notice that when $x_R=\infty$, $S_\rho|_{x=x_R}$ is empty with vanishing Euler characteristic.

We will call $<\rho_L|R_T^m(z)|\rho_R>$ the \emph{$m$-graded Ruling polynomial} of $T$ with boundary conditions $(\rho_L,\rho_R)$.
\end{definition}

\begin{remark}\label{rem:filling_surface_computation_formula}
Given a $m$-graded normal Ruling $\rho$, with $n_L=2n_L'$ (resp. $n_R=2n_R'$) left (resp. right) end-points and $c_L$ (resp. $c_R$) left (resp. right) cusps, then $S_{\rho}|_{x=x_R}$ is the disjoint union of $n_R'$ closed line segments and $n=n_L'+c_L=n_R'+c_R$ is the number of eyes in $\rho$. Hence, $\chi(S_{\rho}|_{x=x_R})=n_R'$ is independent of $\rho$ and we get a simple computation formula
\begin{equation}
\chi(\rho)=c_R-s(\rho)
\end{equation}
where $s(\rho)$ is defined in Definition \ref{def:switches and returns}.
In particular, when $T$ is a Legendrian link, the definition here coincides with the usual definition \cite{HR15} of Ruling polynomials for Legendrian links.

\noindent{}Moreover, when $T$ is a trivial Legendrian tangle of $2n$ parallel strands, then 
\begin{eqnarray*}
<\rho_L|R_T^m(z)|\rho_R>=\delta_{\rho_L,\rho_R}.
\end{eqnarray*} 
This may be called the \emph{Identity axiom} for Ruling polynomials, see Remark \ref{rem:TQFT interpretation} below.
\end{remark}

\subsection{Invariance and composition axiom}
Given a Legendrian tangle $T$, let's denote by $\mr{NR}_T^m$ (resp. $\mr{NR}^m_T(\rho_L,\rho_R)$) the set of $m$-graded normal Rulings of $T$ (resp. those with boundary conditions $(\rho_L,\rho_R)$).

\begin{lemma}\label{lem:filling_surface}
Given a Legendrian isotopy $h$ between 2 Legendrian tangles $T$, $T'$, preserving the Maslov potentials $\mu$, $\mu'$, there's a canonical bijection between the set of $m$-graded normal Rulings of $T$ and $T'$
\begin{equation*}
\phi_h: \mr{NR}_T^m \xrightarrow[]{\sim} \mr{NR}^m_{T'}
\end{equation*}
commuting with the restrictions $r_L,r_R$, and such that for any $m$-graded normal Ruling $\rho$, $S_{\rho}$ and $\phi(S_{\rho})$ are homeomorphic, relative to the boundary pieces at  $x=x_L$ and $x=x_R$.
\end{lemma}
Note that for such 2 Legendrian isotopic tangles $(T,\mu), (T',\mu')$, their left and right pieces are necessarily identical: $T_L=T'_L, T_R=T'_R$.

As a consequence of Lemma \ref{lem:filling_surface}, we obtain
\begin{theorem}\label{thm:invariance and composition of Ruling polynomials}
The $m$-graded Ruling polynomials $<\rho_L|R_T^m(z)|\rho_R>$ are Legendrian isotopy invariants for $(T,\mu)$.\\
Moreover, suppose $T=T_1\circ T_2$ is the composition of two Legendrian tangles $T_1, T_2$, that is, $(T_1)_R=(T_2)_L$ and $T=T_1\cup_{(T_1)_R}T_2$, then the \emph{composition axiom} for Ruling polynomials holds:
\begin{equation}
<\rho_L|R_T^m(z)|\rho_R>=\sum_{\rho_I}<\rho_L|R_{T_1}^m(z)|\rho_I><\rho_I|R_{T_2}^m(z)|\rho_R>
\end{equation}
where $\rho_I$ runs over all the $m$-graded normal rulings of $(T_1)_R=(T_2)_L$.
\end{theorem}

\begin{proof}
The invariance of Ruling polynomials follows immediately from Lemma \ref{lem:filling_surface}. As for the composition axiom, let $\rho$ be any $m$-graded normal ruling of $T$ such that $\rho|_{T_L}=\rho_L, \rho|_{T_R}=\rho_R$. Let $T_1, T_2$ be Legendrian tangles over the open intervals $(x_L,x_I), (x_I,x_R)$ respectively. Take $\rho_I=\rho|_{(T_1)_R},\rho_1=\rho|_{T_1},\rho_2=\rho|_{T_2}$.
 Let $S_{\rho}, S_{\rho_1}$,$ S_{\rho_2}$ be the filling surfaces of $\rho, \rho_1$,$\rho_2$ over $T, T_1$, $T_2$ respectively. Then $S_{\rho_1}|_{x=x_I}=S_{\rho_2}|_{x=x_I}$ and $S_{\rho}=S_{\rho_1}\cup_{S_{\rho_1}|_{x=x_I}}S_{\rho_2}$, it follows that $\chi(\rho)=\chi(S_{\rho})-\chi({S_{\rho}|_{x=x_R}})
=\chi(S_{\rho_1})-\chi(S_{\rho_1}|_{x=x_I})+\chi(S_{\rho_2})-\chi(S_{\rho_2}|_{x=x_R})=\chi(\rho_1)+\chi(\rho_2)$.
Now, the composition axiom follows immediately from applying this into Definition \ref{def:Ruling polynomial} of Ruling polynomials.
\end{proof}

\begin{remark}\label{rem:TQFT interpretation}
The previous theorem suggests a ``TQFT" interpretation of Ruling polynomials for Legendrian tangles, strengthen the analogue that Ruling polynomials are Legendrian versions of Jones polynomials, which fits into a TQFT in smooth knot theory.

Morally, we may regard Legendrian tangles $(T,\mu)$ as 1-dimensional cobordisms from the $0$-manifold of left endpoints with additional structures (equivalently, $T_L$) to the $0$-manifold of right endpoints with additional structures (equivalently, $T_R$). In other words, the $0$-manifolds with additional structures (equivalently, trivial Legendrian tangles $(E,\mu_0)$ of even parallel strands) and $1$-dimensional cobordisms (equivalently, Legendrian tangles $(T,\mu)$) form a special 1-dimensional cobordism category $\mc{LT}_1$. Now, we can view Ruling polynomials $R_T^m(z)$ as a ``1-dimensional TQFT" functor $R^m$ from $\mc{LT}_1$ into the category of free modules of finite rank over $K:=\mb{Z}[z,z^{-1}]$ (See \cite{Ati88} for the basic concepts of TQFTs).

More precisely, associate to any any trivial Legendrian tangle $(E,\mu_0)$ of even parallel strands (viewed as an object of $\mc{LC}_1$), $R^m$ assigns the free module $R^m(E)$ over $K$ generated by all the $m$-graded normal rulings of $E$; Associate to any 1-dimensional cobordism $(T,\mu)$, $R^m$ assigns the $K$-module morphism $R^m(T):=R_T^m(z)$ from $R^m(T_L)$ to $R^m(T_R)$, defined by the matrix coefficients $<\rho_L|R_T^m(z)|\rho_R>$. The previous theorem and Remark \ref{rem:filling_surface_computation_formula} shows that $R^m$ is indeed a functor.
\end{remark}

\begin{proof}[Proof of Lemma \ref{lem:filling_surface}]
As any Legendrian isotopy of Legendrian tangles is a composition of a finite sequence of smooth isotopies and Legendrian Reidemeister moves of the 3 types (see FIGURE \ref{fig:LR}), it suffices to show the proposition for a single smooth isotopy or each of the 3 types of Legendrian Reidemeister moves. As always, we assume the $x$-coordinates of the crossings and cusps of the tangle fronts in question are all distinct. The proof is essentially done for each case by pictures.

\emph{If $h$ is a smooth isotopy}. This case is actually not trivial, as the ordering by $x$-coordinates of the crossings and cusps will change during a smooth isotopy, which will affect the set of switches of the normal Rulings. We illustrate only one such a case (FIGURE \ref{fig:filling surface_smooth isotopy}), the other cases are either similar or trivial. Let $a,b$ be 2 neighboring crossings of $T$, say $x(a)<x(b)$, and the smooth isotopy $h$ moves $a$ to the right of $b$ (i.e. $x(a)>x(b)$ after the isotopy), with the remaining part fixed.

Given a $m$-graded normal ruling $\rho$ of $T$ before the smooth isotopy, if $\rho$ has no switches at $a, b$, take $\phi_h(\rho)$ to be the obvious $m$-graded normal ruling corresponding to $\rho$. In particular, $\phi_h(\rho)$ has no switches at $a, b$ either. It's also clear that the filling surfaces $S_{\rho}$ and $S_{\phi_h(\rho)}$ are homeomorphic relative to the boundary pieces at $x=x_L$ and $x=x_R$.

\begin{figure}[!htbp]
\begin{center}
\minipage{0.8\textwidth}
\includegraphics[width=\linewidth,height=1in]{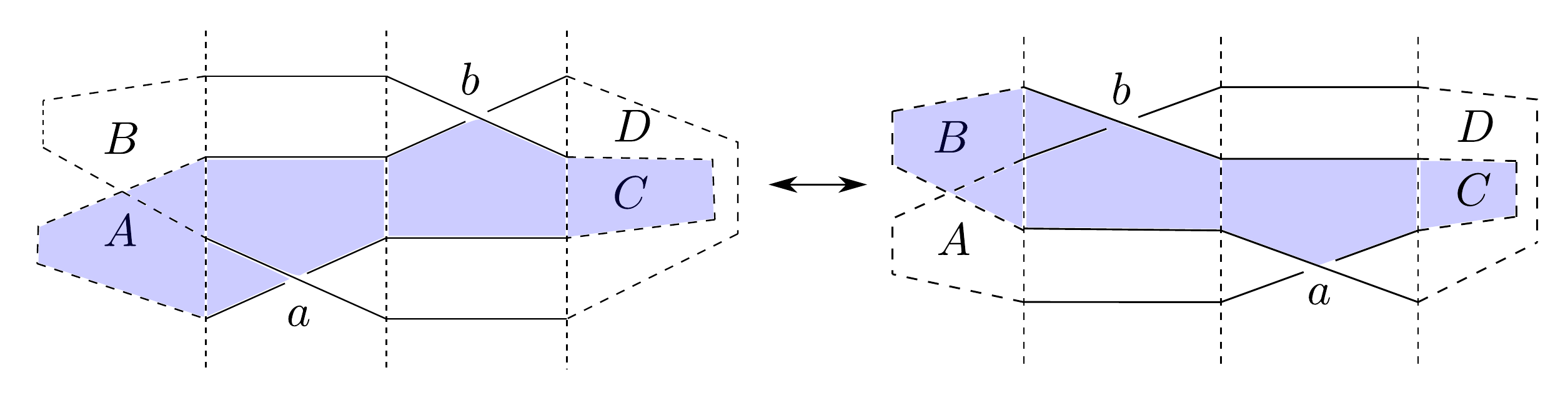}
\endminipage\hfill
\end{center}
\caption{The bijection between $m$-graded normal Rulings under smooth isotopy: The left hand side is a $m$-graded normal ruling $\rho$ with a switch at $b$, instead the corresponding $m$-graded normal ruling $\phi_h(\rho)$ on the right hand side has a switch at $a$. $A,B,C,D$ are the parts of the filling surface $S_{\rho}$ (as well as $S_{\phi_h(\rho)}$) outside the vertical strip drawn in the picture.}
\label{fig:filling surface_smooth isotopy}
\end{figure}

If $\rho$ has a single switch at $a$ or $b$, say $b$, then $|b|=0 (\mr{mod} m)$. The switch $b$ belongs to 2 eyes of $\rho$, say $e_1, e_2$. Recall that each eye has 2 paths (see definition \ref{def:eye}) on the boundary going from left to right, let's call them the \emph{upper-path} and \emph{lower-path} according to their $z$-coordinates. In our case, each of the 2 eyes $e_1, e_2$ has one path containing $b$. If the the remaining 2 companion paths of $e_1, e_2$ contain at most one of the two strands at the crossing $a$, again $\phi_h(\rho)$ is taken to be the obvious normal ruling corresponding to $\rho$ with a switch at $b$, no switch at $a$. The proposition holds trivially. If the remaining 2 companion paths of $e_1, e_2$ restrict to the strands near $a$. By the first condition in the definition \ref{def:normal_ruling} of a $m$-graded normal ruling, we also have that $|a|=0 (\mr{mod} m))$. We look at the relative positions of the 2 eyes $e_1, e_2$ in the vertical strip near $a, b$. We consider only one situation illustrated by FIGURE \ref{fig:filling surface_smooth isotopy}, the others are entirely similar. In this case, the picture on the left gives $\rho$ near $a,b$. We can define $\phi_h(\rho)$ to be the $m$-graded normal ruling (the picture on the right) having a switch at $a$, no switch at $b$ and with the same remaining part as $\rho$. It's easy to see from FIGURE \ref{fig:filling surface_smooth isotopy} that $\phi_h(\rho)$ satisfies the conditions in the definition \ref{def:normal_ruling}. Moreover, denote the remaining parts of the filling surface $S_{\rho}$ by $A, B, C, D$ respectively as in FIGURE \ref{fig:filling surface_smooth isotopy}, then $S_{\rho}$ is: glue $A, C$ by a strip, call the result $A\leftrightarrow C$, glue $B, D$ by a strip, get $B\leftrightarrow D$, then glue $A\leftrightarrow C$ and $B\leftrightarrow D$ by a half-twisted strip, by moving the strips in the gluing, we see the result is simply $A, B, C, D$ with 3 strips (or 1-handles) attached. On the other hand, the filling surface $S_{\phi_h(\rho)}$ is $A\leftrightarrow D$ and $B\leftrightarrow C$ glued via a half-twisted strip (FIGURE \ref{fig:filling surface_smooth isotopy} (right)), again the picture shows $S_{\phi_h(\rho)}$ is $A,B,C,D$ attached with 3 strips. Hence, we conclude that the 2 filling surfaces $S_{\rho}$ and $S_{\phi_h(\rho)}$ are homeomorphic relative to the boundary pieces at $x=x_L$ and $x=x_R$. Note also that, this homeomorphism is orientation-preserving if $S_{\rho}$ is orientable.

When $a,b$ are both the switches of $\rho$, we define $\phi_h(\rho)$ to be the corresponding $m$-graded normal ruling having both $a, b$ as switches. A similar argument proves the proposition.

\emph{If $h$ is a type I Legendrian Reidemeister move (FIGURE \ref{fig:LR} (left)).} Under a type I Legendrian Reidemeister move, the additional crossing necessarily has degree 0. Given a $m$-graded normal ruling $\rho$, we define $\phi_h(\rho)$ to be the corresponding normal ruling having this additional crossing as a switch, and vice versa. The proposition holds trivially, since adding one disk along the boundary doesn't change the topological type (and also the orientability) of a surface.

\emph{If $h$ is a type II Legendrian Reidemeister move (FIGURE \ref{fig:LR} (middle)).} The defining conditions of a normal ruling ensures that the 2 additional crossings can not be switches, so $\phi_h$ is the obvious bijection. Again, the proposition holds trivially.

\emph{If $h$ is a type III Legendrian Reidemeister move (FIGURE \ref{fig:LR} (right)).} Let $a, b, c$ be the 3 crossings involved in the Legendian Reidemeister move III. By a smooth isotopy as proved above, we may assume $a, b, c$ are neighboring crossings (see FIGURE \ref{fig:filling surface_LRIII_1} for an example). Given a $m$-graded normal ruling $\rho$ of $T$ before the move, we need to construction the bijection $\phi_h$ case by case. If $\rho$ has at most one switch at $a, b, c$, $\phi_h(\rho)$ is the obvious normal ruling corresponding to $\rho$, with the same switches at $a, b, c$ as $\rho$. The proposition follows easily.

\begin{figure}[!htbp]
\begin{center}
\minipage{0.8\textwidth}
\includegraphics[width=\linewidth, height=1.2in]{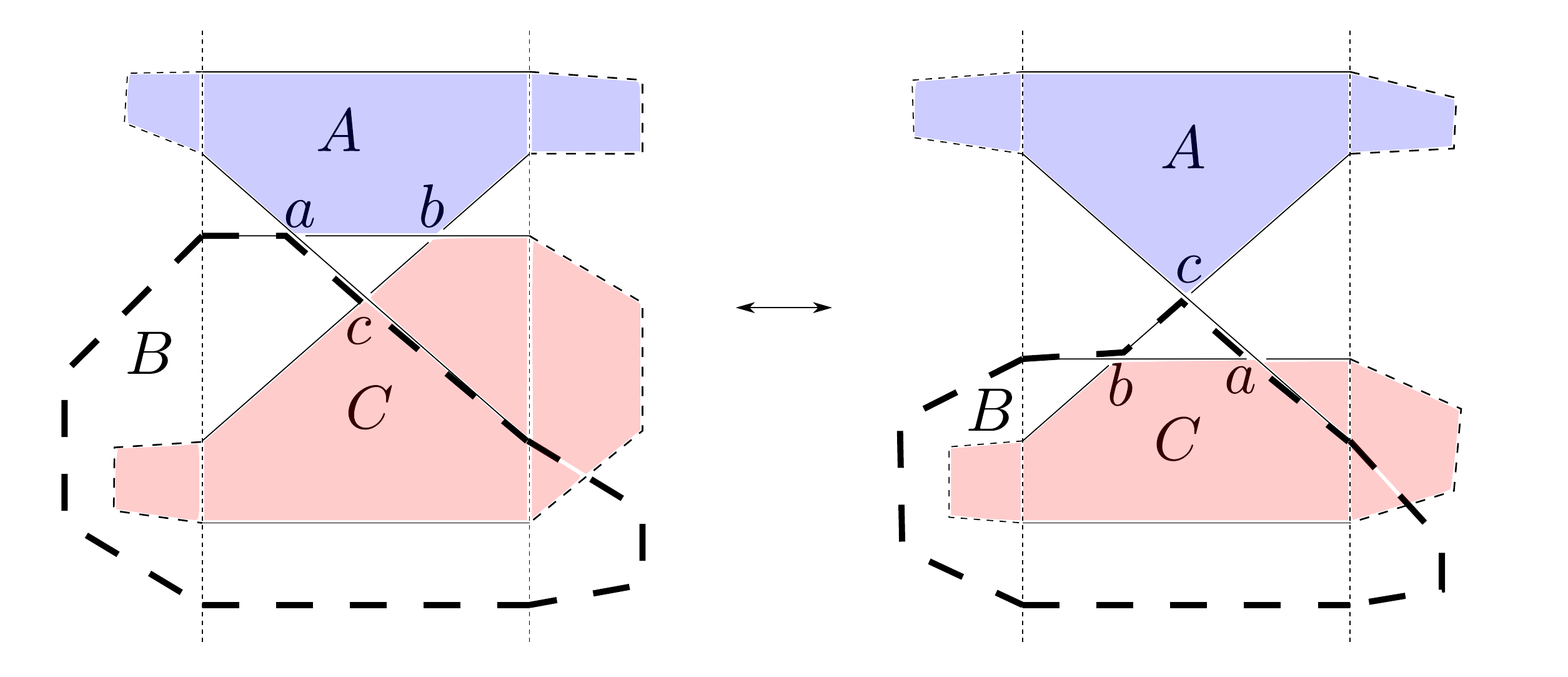}
\endminipage\hfill
\end{center}
\caption{The bijection between $m$-graded normal Rulings under a Legendrian Reidemeister move III: In the left figure, the normal ruling $\rho$ has switches at $a, b$, $A, B, C$ are the 3 eyes containing $a, b, c$, with $A$ the blue disk, $C$ the red disk and $B$ the disk enclosed by thick dashed lines. In the right figure, the corresponding normal ruling $\phi_h(\rho)$ has switches at $b, c$, the other letters are similar. By moving the strips at the switches in the vertical strip, the fillings surfaces $S_{\rho}, S_{\phi_h{(\rho)}}$ are homeomorphic relative to the boundaries at $x=x_L$ and $x=x_R$.}
\label{fig:filling surface_LRIII_1}
\end{figure}

If $\rho$ has 2 switches at $a, b$, no switch at $c$. Then the switches belong to 3 eyes, called $A, B, C$. We look at the relative positions of the 6 paths of the eyes $A, B, C$ in the small vertical strip containing $a, b, c$. We only consider one such a case as in FIGURE \ref{fig:filling surface_LRIII_1} (left), the other cases are similar. Note that $a, b$ are switches imply that all the 3 crossings $a, b, c$ have degree 0 modulo $m$. Moreover, the normal conditions in the definition \ref{def:normal_ruling} of a $m$-graded normal ruling ensures that, there's a unique way to construct a $m$-graded normal ruling $\phi_h(\rho)$ which coincides with $\rho$ outside the vertical strip. The converse is also true for the same reason. The picture of $\phi_h$ is shown in FIGURE \ref{fig:filling surface_LRIII_1} (right), note that now $\phi_h(\rho)$ has 2 switches at $b, c$, no switch at $a$. Moreover, the ruling surfaces $S_{\rho}$ and $S_{\phi_h(\rho)}$ only differ by the gluing of the 3 eyes $A, B, C$ inside the vertical strip. Use the notations as in the proof for smooth isotopies, the left hand side of FIGURE \ref{fig:filling surface_LRIII_1} gives $-B\leftrightarrow A\leftrightarrow -C$, where $-B$ (resp. $-C$) means $B$ with the opposite orientation as that induced from $U\times \mb{R}_z$. On the other hand, the right hand side gives $A\leftrightarrow -(B\leftrightarrow C)$. By moving the gluing strips (or $1$-handles), it's easy to see that the results after gluing can be identified by an (orientation preserving) homeomorphism which is identity outside the vertical strip. This shows that $S_{\rho}$ and $S_{\phi_h(\rho)}$ are homeomorphic relative to the boundaries at $x=x_L$ and $x=x_R$, and the homeomorphism is orientation-preserving if $S_{\rho}$ is orientable.

\begin{figure}[!htbp]
\begin{center}
\minipage{0.8\textwidth}
\includegraphics[width=\linewidth, height=1.2in]{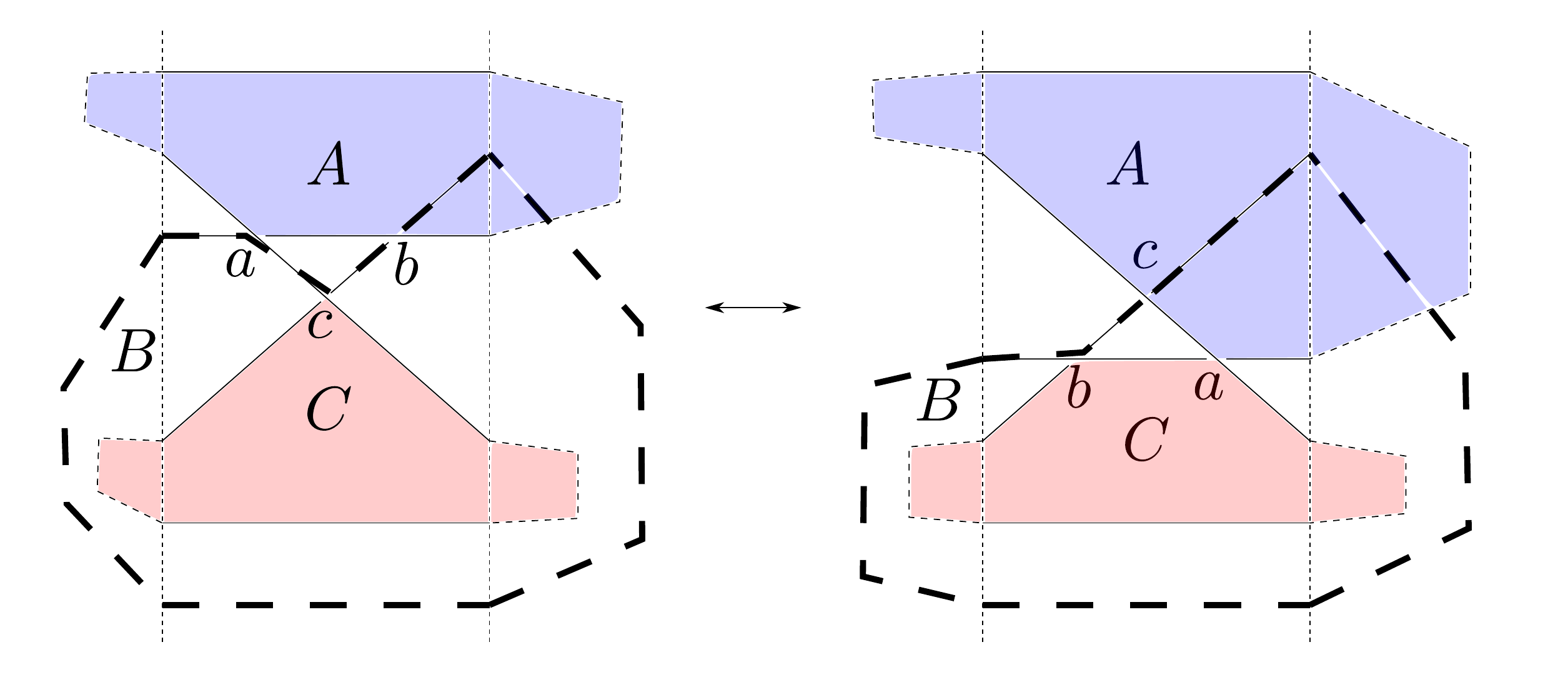}
\endminipage\hfill
\end{center}
\caption{The bijection between $m$-graded normal Rulings under a Legendrian Reidemeister move III: In the left figure, the normal ruling $\rho$ has switches at $a, c$. In the right figure, the corresponding normal ruling $\phi_h(\rho)$ has switches at $a, b$. the other letters are similar as in FIGURE \ref{fig:filling surface_LRIII_1}. By moving the strips at the switches in the vertical strip, the fillings surfaces $S_{\rho}, S_{\phi_h{(\rho)}}$ are homeomorphic relative to the boundaries at $x=x_L$ and $x=x_R$.}
\label{fig:filling surface_LRIII_2}
\end{figure}

If the normal ruling $\rho$ has 2 switches at $a, c$, and no switch at $b$. Similarly, we look at the relative positions of the 3 eyes $A, B, C$ containing $a, b, c$. Again, we only look at one such a case as in FIGURE \ref{fig:filling surface_LRIII_2} (left). The other cases are similar. Now by a similar argument as above, FIGURE \ref{fig:filling surface_LRIII_2} proves the proposition. Note, in this case $\phi_h(\rho)$ has 2 switches at $a, b$, no switch at $c$. The gluing of the 3 eyes on the left figure is $A\leftrightarrow -(B\leftrightarrow C)$, the gluing of the 3 eyes on the right figure is $-(B\leftrightarrow C)\leftrightarrow A$. They can be identified without changing the parts outside the vertical strip.

The case when $\rho$ has switches at $b, c$ is entirely similar to the case above.

\begin{figure}[!htbp]
\begin{center}
\minipage{0.8\textwidth}
\includegraphics[width=\linewidth, height=1.2in]{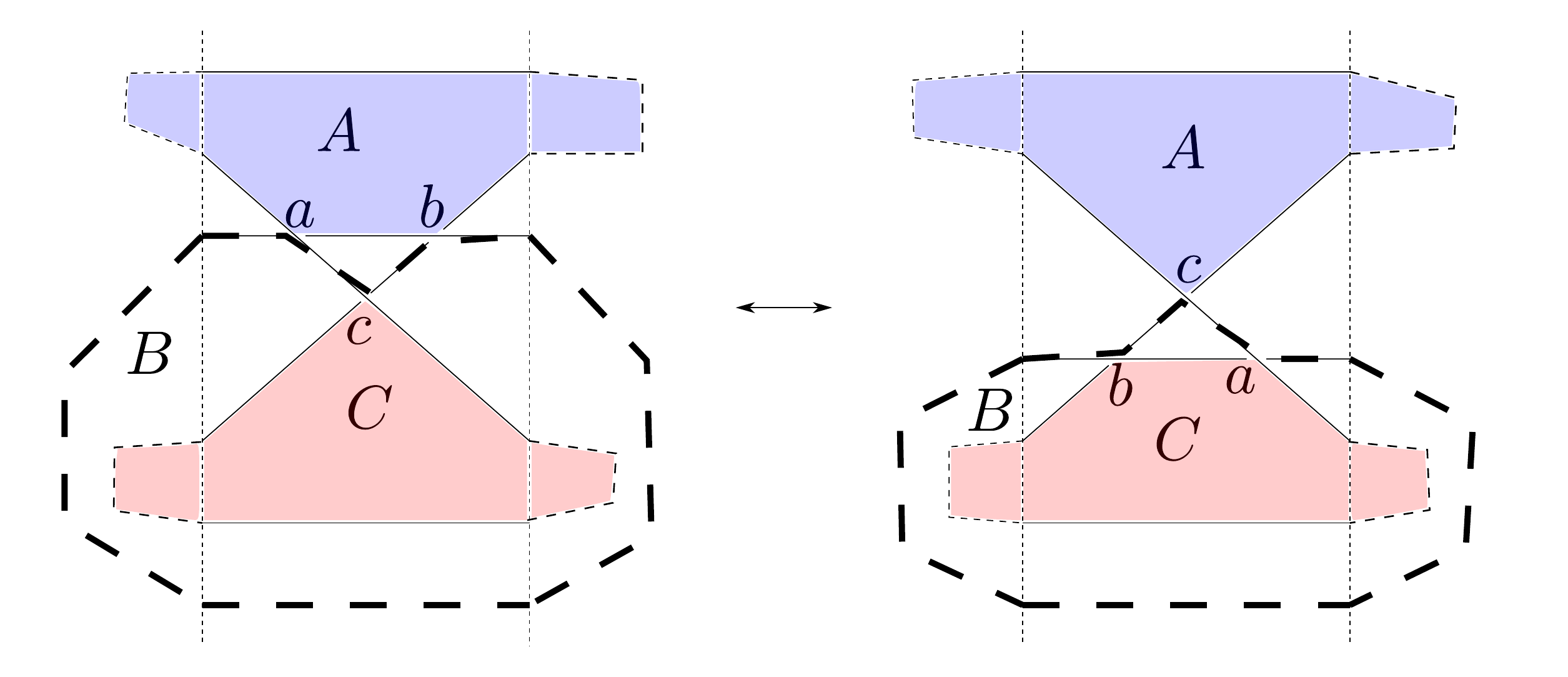}
\endminipage\hfill
\end{center}
\caption{The bijection between $m$-graded normal Rulings under a Legendrian Reidemeister move III: the normal ruling $\rho$ (left) and the corresponding normal ruling $\phi_h(\rho)$ (right) both have switches at the 3 crossings $a, b, c$.}
\label{fig:filling surface_LRIII_3}
\end{figure}

If the normal ruling $\rho$ has switches at all the 3 crossings $a, b, c$ (see FIGURE \ref{fig:filling surface_LRIII_3} for an example), then $\phi_h(\rho)$ is the obvious normal ruling having $a, b, c$ as switches and the same shape as $\rho$ outside the vertical strip. The proposition again follows easily. This finishes the proof.

\end{proof}

\begin{remark}\label{rem:types of crossings under Legendrian isotopy}
Given any $m$-graded normal ruling $\rho$ of a Legendrian tangle $T$, in the proof of the previous lemma, a direct check also shows that $s(\phi_h(\rho))-s(\rho), r(\phi_h(\rho))-r(\rho)$ and $d(\phi_h(\rho))-d(\rho)$ (see Definition \ref{def:switches and returns}) are all independent of $\rho$. We will use this fact in the proof of Corollary \ref{cor:invariance of augmentation numbers}.
\end{remark}

\subsection{Example}

\begin{example}\label{ex:bordered trefoil knot}
Consider the Legendrian tangle (front) $T$ given by FIGURE \ref{fig:filling_surface} (left), obtained by removing the left and right cusps of a Legendrian trefoil knot. $T$ has 3 crossings $a_1, a_2, a_3$, and with the $\mb{Z}$-valued Maslov potential $\mu$ chosen as in the figure, the degrees are $|a_1|=|a_2|=|a_3|=0$. Moreover, $T$ has $4$ left end-points, $4$ right end-points and no left or right cusps. So in Remark \ref{rem:filling_surface_computation_formula}, $n_L=n_R=4, c_L=c_R=0$. The left piece $T_L$ (resp. the right piece $T_R$) consists of 4 parallel lines, labeled from top to bottom, say, by $1, 2, 3, 4$, with Maslov potential values $2, 1, 1, 0$ respectively.

Let's calculate the Ruling polynomials for $(T,\mu)$. Firstly, let $m\neq 1$ be a nonnegative integer, then the set of $m$-graded normal rulings for $T_L$ (resp. $T_R$) is $NR^m_{T_L}=\{(\rho_L)_1=(1 2)(3 4), (\rho_L)_2=(1 3)(2 4)\}$ (resp. $NR^m_{T_R}=\{(\rho_R)_1=(1 2)(3 4), (\rho_R)_2=(1 3)(2 4)\}$). Here, for example in $T_L$ $(1 2)(3 4)$ means the pairing between the strands $1, 2$ (resp. $3, 4$), corresponding to a $m$-graded normal ruling of $T_L$.

Using the notations before Lemma \ref{lem:filling_surface} and in Remark \ref{rem:ruling by switches}, the $m$-graded normal rulings of $(T,\mu)$ are:

\begin{enumerate}[label=(\arabic*).]
\item
$NR_T^m((\rho_L)_1,(\rho_R)_1)=\{\rho_3=\{a_1\};\rho_5=\{a_3\};\rho_8=\{a_1,a_2,a_3\}\}$;
\item
$NR_T^m((\rho_L)_1,(\rho_R)_2)=\{\rho_1=\emptyset;\rho_6=\{a_1,a_2\}\}$;
\item
$NR_T^m((\rho_L)_2,(\rho_R)_1)=\{\rho_2=\emptyset;\rho_7=\{a_2,a_3\}\}$;
\item
$NR_T^m((\rho_L)_2,(\rho_R)_2)=\{\rho_4=\{a_2\}\}$.
\end{enumerate}
Note that $\{a_1, a_3\}$ is not a $m$-graded normal ruling since it violates the normal condition in the definition \ref{def:normal_ruling}. Apply the computation formula in Remark \ref{rem:filling_surface_computation_formula}, the Ruling polynomials of $(T,\mu)$ and the corresponding maximal degrees $d$ in $z$ are:
\begin{enumerate}[label=(\arabic*).]
\item
$<(\rho_L)_1|R_T^m(z)|(\rho_R)_1>=2z+z^3$, and $d=3$;
\item
$<(\rho_L)_1|R_T^m(z)|(\rho_R)_2>=1+z^2$, and $d=2$;
\item
$<(\rho_L)_2|R_T^m(z)|(\rho_R)_1>=1+z^2$ with $d=2$;
\item
$<(\rho_L)_2|R_T^m(z)|(\rho_R)_2>=z$ with $d=1$.
\end{enumerate}

Note that for $m=1$, each of the 2 sets $NR^1_{T_L}$ and $NR^1_{T_R}$ contains an additional $1$-graded normal ruling $(\rho_L)_3$ (resp. $(\rho_R)_3)=\{(1 4)(2 3)\}$. However, no $1$-graded normal ruling of $T$ restricts to $(\rho_L)_3$ or $(\rho_R)_3$, hence the above formula also computes the $1$-graded ruling polynomials for $(T,\mu)$.

In Example \ref{ex:Ruling polynoimals vs aug numbers for bordered trefoil knot}, as an illustration of the main theorem \ref{thm:counting for tangles}, we will see that the Ruling polynomials with boundary conditions calculated here, indeed match with the augmentation numbers with the same boundary conditions (to be defined in Section \ref{subsec:augmentations for tangles}) for the Legendrian tangle $(T,\mu)$ as above.

\end{example}

\section{The LCH differential graded algebras for Legendrian tangles}\label{sec:LCH DGA for tangles}
Generalizing the Chekanov-Eliashberg construction of the LCH DGAs for Legendrian links, the LCH DGAs for Legendrian tangles with simple fronts (See Section \ref{subsubsec:fronts}) were explicitly constructed in \cite{Siv11}. Here we give the basic constructions and properties of the LCH DGAs associated to any Legendrian tangle $T$ (not necessarily with simple front). The key properties of the DGAs are the homotopy invariance and co-sheaf property.
As in the Section \ref{subsec:LCH DGA}, we allow some base points placed on the tangle.

Given an oriented tangle front $T$, provided with a $\mb{Z}/2r$-valued Maslov potential $\mu$. We can orient the tangle so that each strand is right-moving (resp. left-moving) if and only if its Maslov potential value is even (resp. odd). We place some base points $*_1,\ldots,*_B$ on $T$ so that each connected component containing a right cusp has at least one base point. Assume $T$ has $n_L$ left end-points and $n_R$ right end-points, labeled from top to bottom by $1, 2,\ldots, n_L$ (resp. $1, 2,\ldots, n_R$).
We will construct a LCH DGA associated to the resolution (see Section \ref{subsubsec:resolution_construction}) of $T$. The idea is to embed the tangle front $T$ into a Legendrian link front $\Lambda$, and take the resolution of $\Lambda$. Then define the $\mb{Z}/2r$-graded LCH DGA $\mc{A}(T)=\mc{A}(T,\mu,*_1,\ldots,*_B)$ as a sub-DGA of $\mc{A}(\mr{Res}(\Lambda))$.

\subsection{Embed a Legendrian tangle into a Legendrian link}\label{subsec:embed a Legendrian tangle into a link}

Let $(T,\mu,*_1,\ldots,*_B)$ be given as above. In this subsection, we will give the construction to embed $T$ into a Legendrian link front (see FIGURE \ref{fig:LCH DGA_tangle} for an illustration). In the case of FIGURE \ref{fig:LCH DGA_tangle} (left), $T$ is the Legendrian tangle in the vertical strip, with $n_L=4, n_R=2$ and $B=2$.

We firstly glue a $n_L$-copy of a left cusp along the top end-points to the $n_L$ left end-points of $T$ (see the $4$-copy of the left cusp with crossings $\alpha_{ij}$'s in FIGURE \ref{fig:LCH DGA_tangle} (left)), and also glue a $n_R$-copy of a right cusp along the bottom end-points to the $n_R$ right end-points of $T$ (See the $2$-copy of the right cusp with crossing $\beta_1$ in FIGURE \ref{fig:LCH DGA_tangle} (left)). Next, we glue a $n_R$-copy of a left cusp, placed to the left of the diagram, along the top end-points to the top end-points of the $n_R$-copy of the right cusp (See the $2$-copy left cusp to the left of the diagram in FIGURE \ref{fig:LCH DGA_tangle} (left)). Now we are left with a diagram, say $D$, with $n_L+n_R$ right end-points (see the bottom dashed line in FIGURE \ref{fig:LCH DGA_tangle} (left)). We glue these right end-points via $n_L+n_R$ right cusps as follows. We extend the Maslov potential $\mu$ to $D$, this extension is unique. Every connected component with nonempty boundary (i.e. component which is not a loop) of $T$ is connected to exactly 2 such right end-points, and it's easy to see that $\mu(\text{upper end-point})-\mu(\text{lower end-point})=\pm 1 (\mr{mod} 2r)$. We glue a right cusp to the 2 end-points from the right so that $\mu$ defines a $\mb{Z}/2r$-valued Maslov potential on the resulting front diagram. We will place these right cusps so that they have almost the same $x$-coordinates. Note that this procedure may involve some additional crossings (See the bottom-right in FIGURE \ref{fig:LCH DGA_tangle} (left)).

\begin{figure}[!htbp]
\begin{center}
\minipage{0.5\textwidth}
\includegraphics[width=\linewidth, height=2.2in]{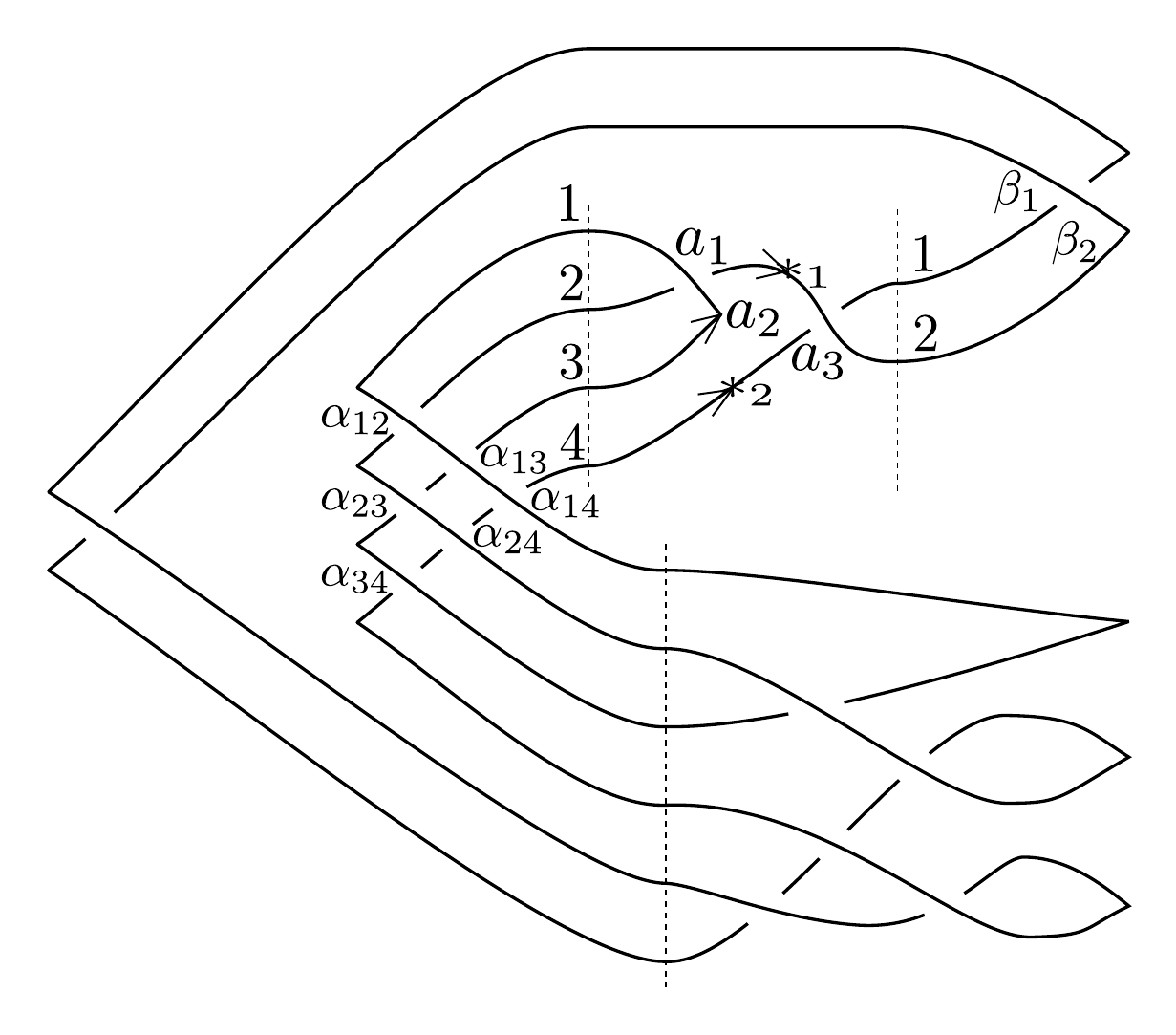}
\endminipage\hfill
\minipage{0.5\textwidth}
\includegraphics[width=\linewidth, height=2.2in]{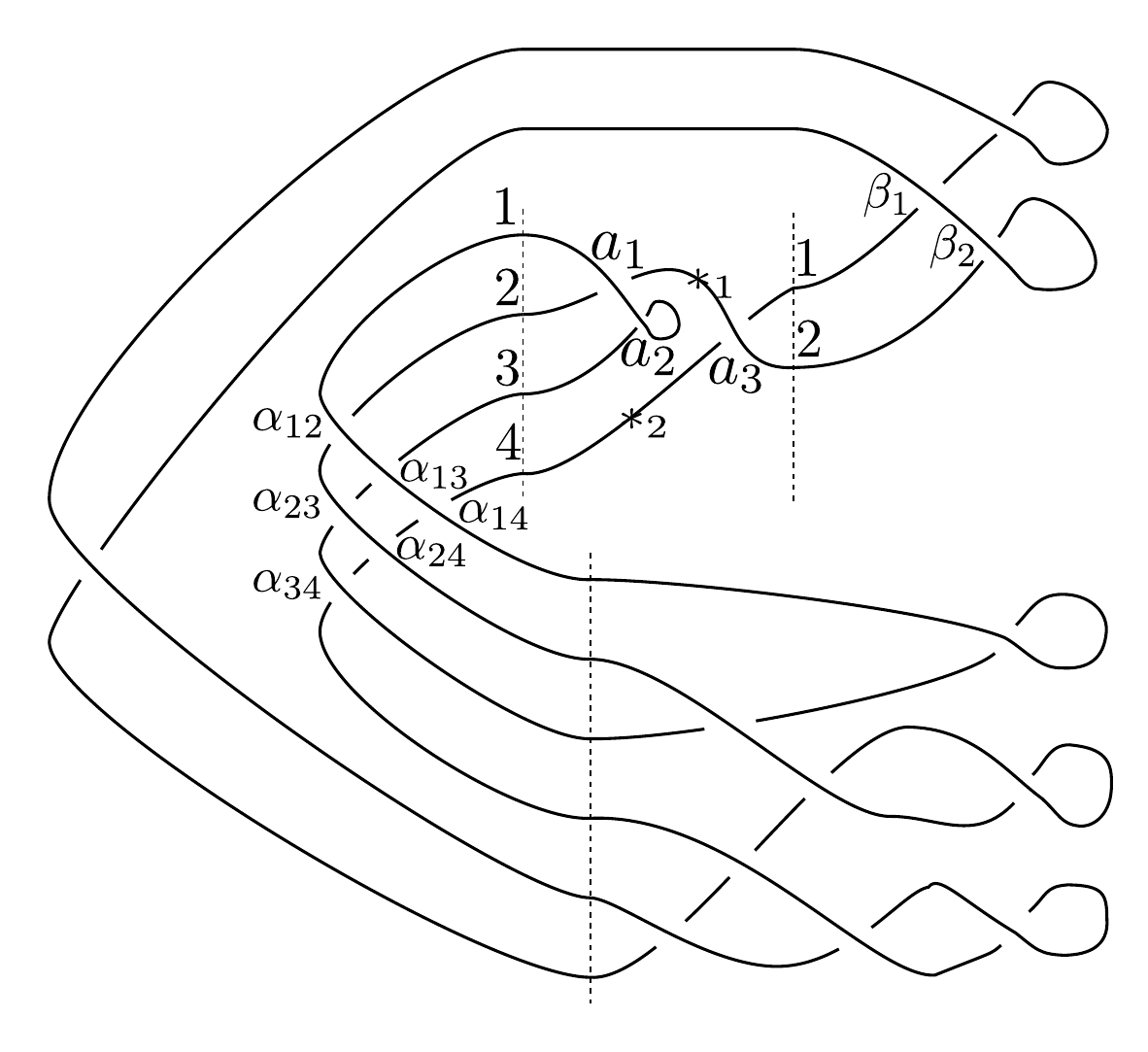}
\endminipage\hfill
\end{center}
\caption{LCH DGA for tangles: In the left picture, embed the Legendrian tangle $T$ (the part in the vertical strip) into a Legendrian link front $\Lambda$; In the right picture, take the resolution of the Legendrian link front $\Lambda$.}
\label{fig:LCH DGA_tangle}
\end{figure}

Let's denote by $\Lambda$ the resulting front diagram, we will also add some additional base points, for example one base point at each of the additional $n_R+n_L$ right cusps in the bottom of $\Lambda$, so that each component of $\Lambda$ contains at least one base point. By construction, $\Lambda$ is equipped with a $\mb{Z}/2r$-valued Maslov potential, still denoted by $\mu$. Moreover, $\Lambda$ is simple (see Section \ref{subsubsec:resolution_construction}) away from $T$. The $n_L$-copy of the left cusp glued to the left end-points of $T$ has $\binom{n_L}{2}$ crossings, denoted by $\alpha_{ij}, 1\leq i<j\leq n_L$, where $\alpha_{ij}$ is the crossing of the 2 strands connected to the 2 left end-points $i, j$ of $T$. Then, we have $|\alpha_{ij}|=\mu(i)-\mu(j)-1$.

Now by the resolution construction, we can define the LCH DGA $\mc{A}(\mr{Res}(\Lambda))$. Let $\{a_1,a_2,\ldots,a_R\}$ be the crossings and right cusps of $T$, $t_1,t_2,\ldots,t_B$ be the generators in $\mc{A}(\mr{Res}(\Lambda))$ corresponding to the base points $*_1,*_2,\ldots,*_B$. By the resolution construction \cite{Ng03}, the differential $\partial a_i$ only involves the generators $\{a_j, 1\leq j\leq R, t_j^{\pm 1}, 1\leq j\leq B\}$ and $\{\alpha_{ij}, 1\leq i<j\leq n_L\}$.

Moreover, the differentials of $\alpha_{ij}$'s are given by
\begin{equation*}
\partial \alpha_{ij}=\sum_{i<k<j}(-1)^{|\alpha_{ik}|+1}\alpha_{ik}\alpha_{kj}.
\end{equation*}

As a consequence, the subalgebra generated by $t_1^{\pm 1},\ldots,t_B^{\pm 1},a_1,\ldots,a_R$ and $\alpha_{ij}, 1\leq i<j\leq n_L$ form a sub-DGA of $\mc{A}(\mr{Res}(\Lambda))$. This leads to the definition of the LCH DGA $\mc{A}(T)$ of the Legendrian tangle front $T$.

\subsection{LCH DGAs via Legendrian tangle fronts}
Now, let's translate the construction of sub-DGAs in the previous subsection into definitions involving only $T$.

\subsubsection{The general definition}
\begin{definition/proposition}
Define the $\mb{Z}/2r$-graded LCH DGA $\mc{A}(T)$ as follows:

\noindent{}\emph{As an algebra:} $\mc{A}(T)=\mb{Z}[t_1^{\pm 1},\ldots,t_B^{\pm 1}]<a_i, 1\leq i\leq R, a_{ij}, 1\leq i<j\leq n_L>$ is a free associative algebra over $\mb{Z}[t_1^{\pm 1},\ldots,t_B^{\pm 1}]$, where $a_{ij}$ corresponds to the pair of left end-points $i, j$ of $T$.

\noindent{}\emph{The grading and differential} is induced from the identification of $\mc{A}(T)$ with the sub-DGA obtained above, via $t_i^{\pm 1}\leftrightarrow t_i^{\pm 1}, a_i \leftrightarrow a_i$ and $a_{ij}\leftrightarrow \alpha_{ij}$.
By the construction above, the DGA $\mc{A}(T)$ is \emph{independent of the choices} involved in the construction of $\Lambda$. In particular, we can translate the DGA $\mc{A}(T)$ purely in terms of the combinatorics of the tangle front $T$.
More precisely, we have:

\noindent{}\emph{The grading:} $|t_i^{\pm 1}|=0, |a_i|=\mu(\text{over-strand})-\mu(\text{under-strand})$ if $a_i$ is a crossing, $|a_i|=1$ if $a_i$ is a right cusp, and $|a_{ij}|=\mu(i)-\mu(j)-1$.

\noindent{}\emph{The differential:} As usual, we impose the graded Leibniz Rule $\partial(x\cdot y)=(\partial x)\cdot y+(-1)^{|x|}x\cdot\partial y$ and the differential of the generators are defined as follows: $\partial(t_i^{\pm 1})$=0; The differential of $a_{ij}$ given by the same formula for $\alpha_{ij}$ with $\alpha_{\bullet,\bullet}$ replaced by $a_{\bullet,\bullet}$, that is,
\begin{equation}\label{eqn:differential for pair of ends}
\partial a_{ij}=\sum_{i<k<j}(-1)^{|a_{ik}|+1}a_{ik}a_{kj}.
\end{equation}

To translate the differential of a crossing or a right cusp, we proceed as in \cite[Def.2.6]{Ng03}. Let $a=a_i$ and $v_1,\ldots, v_n$ be some elements in the generators $\{a_i, 1\leq i\leq R, a_{ij}, 1\leq i<j\leq n_L\}$ of $T$ for $n\geq 0$. Let $D_n^2=D^2-\{p,q_1,\ldots,q_n\}$ be a fixed oriented disk with $n+1$ boundary punctures (or vertices) $p,q_1,\ldots,q_n$, arranged in a counterclockwise order.

\begin{figure}[!htbp]
\begin{center}
\minipage{0.8\textwidth}
\includegraphics[width=\linewidth, height=4in]{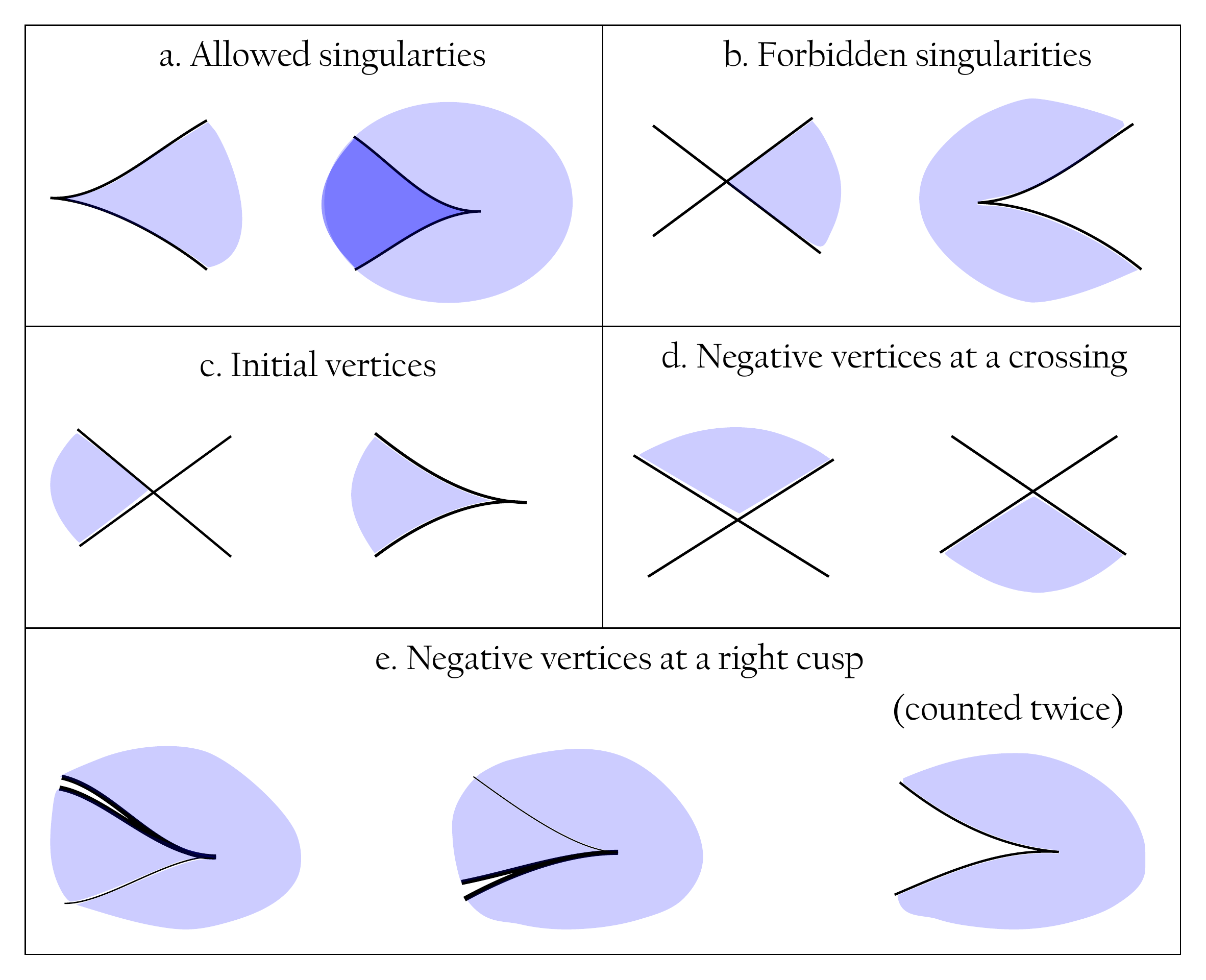}
\endminipage\hfill
\end{center}
\caption{Admissible disks: The image of the disk $D_n^2$ under an admissible map near a singularity or a vertex on the boundary $\partial D_n^2$. The first row indicates the possible singularities, the second and third rows indicate the possible vertices. In the first 2 pictures of part e, 2 copies of the same strand (the heavy lines) are drawn for clarity.}
\label{fig:Admissible_disks}
\end{figure}

\begin{definition}\label{def:admissible_disks_via_fronts}
Define the moduli space $\Delta(a;v_1,\ldots,v_n)$ to be the space of \emph{admissible disks} $u$ of the tangle front $T$ up to re-parametrization, that is,
\begin{enumerate}[label=(\roman*)]
\item\label{item:i}
\noindent{}\emph{(Immersion with singularities)} The map
$u:(D_n^2,\partial D_n^2)\rightarrow (\mathbb{R}_{xz}^2,T)$ is an immersion, orientation-preserving, and smooth away from possible singularities at left and right cusps, near which the image of the map are indicated as in FIGURE \ref{fig:Admissible_disks}.a,b. 
Note that the singularities are not vertices of $D_n^2$;

\item
\noindent{}\emph{(Initial/positive vertex)} $u$ extends continuously to $p$, with
$u(p)=a$, near which the image of the map is indicated as in Figure \ref{fig:Admissible_disks}.c;

\item
\noindent{}\emph{(Negative vertices at a crossing)} If $v_i$ is a crossing, $u$ extends continuously to $q_i$, with $u(q_i)=v_i$, near which the image of the map is indicated as in Figure \ref{fig:Admissible_disks}.d;

\item
\noindent{}\emph{(Negative vertices at a right cusp)} If $v_i$ is a right cusp, $u$ extends continuously to $q_i$, with $u(q_i)=v_i$, near which the image of the map is indicated as in Figure \ref{fig:Admissible_disks}.e;

\item\label{item:v}
\noindent{}\emph{(Negative vertices at a pair of left end-points)}
If $v_i$ is a pair of left end-points $a_{jk}$, we require that, as one approaches $q_i$ in $D_n^2$, $u$ limits to the line segment $[j,k]$ at the left boundary between the left end-points $j, k$ of $T$;

\item
The $x$-coordinate on the image $\overline{u(D_n^2})$ has a unique local maximum at $a$.
\label{item:vi}
\end{enumerate}
Note: the last condition \ref{item:vi} is in fact a consequence of the previous ones \ref{item:i}-\ref{item:v}. All the defining conditions are direct translations from those in Definition \ref{def:admissible_disks} via Lagrangian projection, for $\mc{A}(\mr{Res}(\Lambda))$ in Section \ref{subsec:embed a Legendrian tangle into a link}. Via the resolution construction (Figure \ref{fig:Res}), the only nontrivial part is the translation for a right cusp, near which the defining conditions are illustrated by Figure \ref{fig:Admissible_disks_resolution}.
\end{definition}

\begin{figure}[!htbp]
\begin{center}
\minipage{0.8\textwidth}
\includegraphics[width=\linewidth, height=1in]{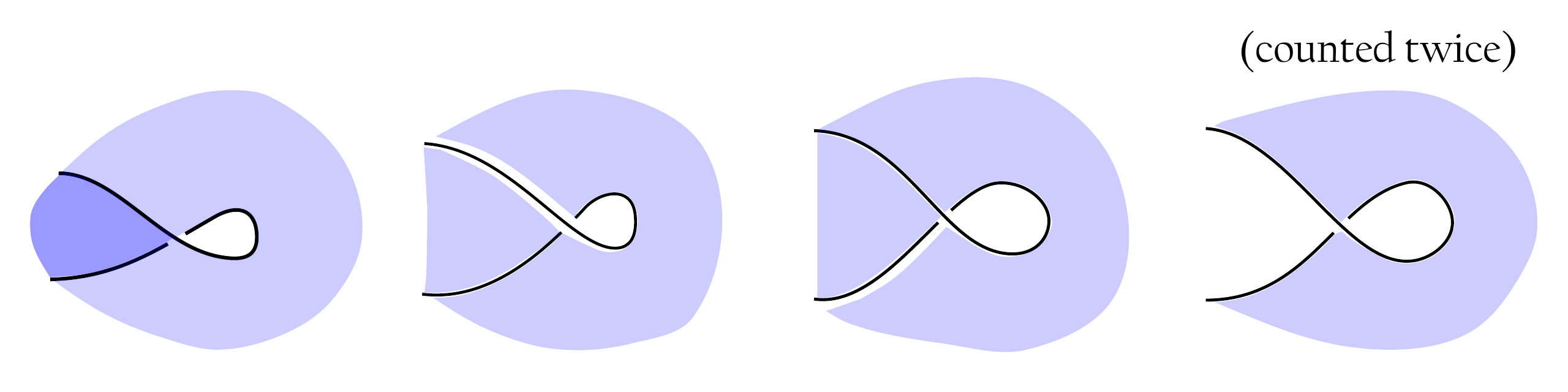}
\endminipage\hfill
\end{center}
\caption{The singularity and negative vertices at a right cusp after resolution: The first figure corresponds to a singularity (Figure \ref{fig:Admissible_disks}.a), the remaining ones correspond to a negative vertex (FIGURE \ref{fig:Admissible_disks}.e, going from left to right).}
\label{fig:Admissible_disks_resolution}
\end{figure}

For each $u\in \Delta(a;v_1,\ldots,v_n)$, walk along $\overline{u(\partial D_n^2)}$ starting from $a$ in counterclockwise direction, we encounter a sequence $s_1,\ldots,s_N(N\geq n)$ of negative vertices of $u$ (crossings, right cusps, or pairs of left end-points as in Definition \ref{def:admissible_disks_via_fronts}) and base points (away from the previous negative vertices). Translate Definition \ref{def:weight for disks}, we obtain

\begin{definition}
The \emph{weight} of $u$ is $w(u):=w(s_1)\ldots w(s_N)$, where
\begin{enumerate}[label=(\roman*)]
\item
$w(s_k)=t_i ($resp. $t_i^{-1})$ if $s_k$ is the base point $*_i$, and the boundary orientation of $u(\partial D^2)$ agrees (resp. disagrees) with the orientation of $T$ near $*_i$. Note that this includes the case when the base point $*_i$ is located at a right cusp, which is also a singularity of $u$ (See Figure \ref{fig:Admissible_disks}.a);
\item
$w(s_k)=v_i$ (resp. $(-1)^{|v_i|+1}v_i$) if $s_k$ is the crossing $v_i$ and the disk $\overline{u(D_n^2)}$ occupies the top (resp. bottom) quadrant of $v_i$ (See Figure \ref{fig:Admissible_disks}.d);
\item
$w(s_k)=a_{ij}$ if $s_k$ is the pair of left end-points $a_{ij}$;
\item
$w(s_k)=w_1(s_k)w_2(s_k)$ if $s_k$ is the right cusp $v_i=u(q_i)$ (see Figure \ref{fig:Admissible_disks}.e), where\\
$w_2(s_k)=v_i$ (resp. $v_i^2$) if the image of $u$ near $q_i$ looks like the first two diagrams (resp. the third diagram) of Figure \ref{fig:Admissible_disks}.e;\\
$w_1(s_k)=1$ if $s_k$ is a unmarked right cusp (equipped with no base point);\\
$w_1(s_k)=t_j$ (resp. $t_j^{-1}$) if $v_i$ is a marked right cusp equipped with the base point $*_j$, and $v_i$ is an up (resp. down) right cusp\footnote{Recall that a cusp is called \emph{up} (resp. \emph{down}) if the orientation of the front $T$ near the cusp goes up (resp. down).}. See Figure \ref{fig:Admissible_disks_resolution} for an illustration.
\end{enumerate}

\noindent Notice that the \emph{convention for the orientation signs} here is as follows: At each crossing of even degree of the tangle front $T$, the two quadrants to the lower right of the under-strand have negative orientation signs. All other quadrants have positive orientation signs.
\end{definition}

\begin{definition}\label{def:differential via tangle fronts}
For $a=a_i$ a crossing or a right cusp, its differential is given by
\begin{equation}
\partial a=\sum_{n,v_1,\dots,v_n}\sum_{u\in\Delta(a;v_1,\ldots,v_n)}w(u)
\end{equation}
where for $a=a_i$ a right cusp, we also include the contribution from an ``invisible" disk $u$ coming from the resolution construction (see Section \ref{subsubsec:resolution_construction}), with $w(u)=1$ (resp. $t_j^{-1}$ or $t_j$), if there's no base point (resp. a base point $*_j$, depending on whether $a_i$ is an up or down right cusp).
\end{definition}

\end{definition/proposition}

\begin{example}
Let $T$ be the Legendrian tangle in Figure \ref{fig:LCH DGA_tangle} (left), with a choice of $\mb{Z}/2r$-valued Maslov potential $\mu$. As an algebra, the LCH DGA is $\mc{A}(T)=\mc{A}(T,\mu,*_1,*_2)=\mb{Z}[t_1^{\pm 1},t_2^{\pm 1}]<a_1,a_2,a_3,a_{ij},1\leq i<j\leq 4>$, where as usual the $a_{ij}$'s are the generators corresponding to the pairs of left endpoints. The differential $\partial$ for $\mc{A}(T)$ is: As usual, $\partial a_{ij}=\sum_{i<k<j}(-1)^{|a_{ik}|+1}a_{ik}a_{kj}$ and $\partial(t_i^{\pm 1})=0$. The differentials for $a_i$'s are as follows
\begin{eqnarray*}
&&\partial a_1=a_{12};\\
&&\partial a_2=1+a_{13}+(-1)^{|a_1|+1}a_1a_{23};\\
&&\partial a_3=t_1^{-1}(a_{24}+a_{23}(a_{14}+(-1)^{|a_1|+1}a_1a_{24}+a_2a_{34}))t_2.
\end{eqnarray*}
Note: there's a strategy to compute $\partial a_3$. We can cut $T$ into elementary Legendrian tangles and apply Definition/Proposition \ref{def:corestriction of DGA}.
\end{example}

By embedding the Legendrian tangle $T$ into a Legendrian link, the proof of Theorem \ref{thm:stable isom} also shows that
\begin{proposition}
The isomorphism class of $\mc{A}(T)$ is independent of the locations of the base points on each connected component of $T$. The stable isomorphism class of $\mc{A}(T)$ is invariant under Legendrian isotopy of $T$.
\end{proposition}

\subsubsection{LCH DGA for simple Legendrian tangles}\label{subsubsec:LCH DGA for simple tangles}
In the case when $T$ is a simple Legendrian tangle (see Section \ref{subsubsec:fronts} for the definition), in particular when $T$ is nearly plat, we have a simple description of $\mr{A}(T)$.

The algebra and grading are the same as before, but the differential counts simpler objects. More precisely, for $a=a_i$ a crossing or a right cusp, the differential $\partial a$ is given by the same formula as in Definition \ref{def:differential via tangle fronts}. However, we have

\begin{lemma}[{\cite[$\S$.2.3]{Ng03}}]\label{lem:disks for simple fronts}
For $T$ a simple Legendrian tangle front, any admissible disk $u$ in $\Delta(a;v_1,\ldots,v_n)$ must satisfy:
\begin{enumerate}
\item
$u$ is an \emph{embedding}, not just an immersion, so no singularity at a right cusp (see Figure \ref{fig:Admissible_disks}.a);
\item
Each negative vertex of $u$ must be a crossing, so there's no negative vertex at a right cusp (see Figure \ref{fig:Admissible_disks}.e);
\item
There's at most one negative vertex at a pair of left end-points.
\end{enumerate}
\end{lemma}

\begin{example}\label{ex:DGA for bordered trefoil knot}
Consider the Legendrian tangle $(T,\mu)$ in Example \ref{ex:bordered trefoil knot} (See Figure \ref{fig:filling_surface} (left)).
As usual, label the left end points of $T$ by $1,2,3,4$ from top to bottom. Let $a_{ij}, 1\leq i<j\leq 4$ be the pairs of left end points of $T$. Then the LCH DGA $\mc{A}(T)$ is generated by $a_{ij}$ and $a_1,a_2,a_3$, with the grading: $|a_{12}|=|a_{13}|=|a_{24}|=|a_{34}|=0, |a_{23}|=-1, |a_{14}|=1$ and $|a_1|=|a_2|=|a_3|=0$. The differential is given by
\begin{eqnarray*}
&&\partial a_{ij}=\sum_{i<l<j}(-1)^{|a_{il}|+1}a_{il}a_{lj};\\
&&\partial a_1=a_{23};\\
&&\partial a_2=\partial a_3=0.
\end{eqnarray*}
\end{example}

\subsection{The co-sheaf property}\label{subsubsec:co-sheaf for DGAs}
Let $T$ be a Legendrian tangle in $J^1U$. Let $V$ be an open subinterval of $U$ such that, the boundary $(\partial \overline{U})\times \mb{R}_z$ is disjoint from the crossings, cusps and base points of $T$. $T|_V$ then gives a Legendrian tangle in $J^1V$ with Maslov potential induced from that of $T$, hence the LCH DGA $\mc{A}(T|_V)$ is defined. There's indeed a co-restriction map of DGAs.
\begin{definition/proposition}[{\cite[Prop.6.12]{NRSSZ15},\cite{Siv11}}]\label{def:corestriction of DGA}
The following defines a morphism of $\mb{Z}/2r$-graded DGAs $\iota_{UV}:\mc{A}(T|_V)\rightarrow \mc{A}(T)$:
\begin{enumerate}[label=(\arabic*)]
\item
$\iota_{UV}$ sends a generator of $\mc{A}(T|_V)$, corresponding to a crossing, a right cusp or a base point of $T$, to the corresponding generator of $\mc{A}(T)$;
\item
For a generator $b_{ij}$ in $\mr{A}(T|_V)$ corresponding to the pair of left end-points $i,j$ of $T|_V$, the image $\iota_{UV}(b_{ij})$ is defined as follows:\\
Use the notations in Section \ref{def:admissible_disks_via_fronts} and consider the moduli space $\Delta(b_{ij};v_1,\ldots,v_n)$ of disks $u:(D_n^2,\partial D_n^2)\rightarrow (\mb{R}_{xz}^2,T)$ satisfying the conditions in definition \ref{def:admissible_disks_via_fronts}, with the condition for $a$ there replaced by ``$u$ limits to the line segments $[i,j]$ between the pair of left end-points $i,j$ of $T|_V$ at the puncture $p\in \partial D^2$ and $u$ attains its local maxima exactly along $[i,j]$". Then define
\begin{equation}
\iota_{UV}(b_{ij})=\sum_{n,v_1,\ldots,v_n}\sum_{u\in\Delta(b_{ij},v_1,\ldots,v_n)}w(u)
\end{equation}
\end{enumerate}
\end{definition/proposition}
\begin{proof}
Apply the proof of Prop. 6.12 in \cite{NRSSZ15}: Though it only deals with Legendrian tangles in nearly plat positions, essentially the same arguments work in the general case, with `embedded disks' replaced by `immersed disks' everywhere.
\end{proof}

\begin{remark}
From the definition, it's easy to see that if the left boundary of $V$ coincides with that of $U$, then the co-restriction
map $\iota_{UV}:\mc{A}(T|V)\hookrightarrow\mc{A}(T)$ is an inclusion.
\end{remark}

\begin{example}[co-restriction $\iota_R$ for a right cusp]
One key example for the co-restriction of DGAs is $\iota_R:\mc{A}(T_R)\rightarrow \mc{A}(T)$, where $T$ be an elementary Legendrian tangle of a single (marked or unmarked) right cusp $a$, and $T_R$ is the right piece of $T$. For simplicity, assume $T$ has $4$ left endpoints and $2$ right endpoints as in Figure \ref{fig:right cusp}. Then $\mc{A}(T_R)=\mb{Z}<b_{12}>$, where $b_{12}$ is the generator corresponding to the pair of left endpoints of $T_R$, and $\mc{A}(T)=\mb{Z}[t,t^{-1}]<a,a_{ij},1\leq i<j\leq 4>$ with $\partial a=t^{\sigma(a)}+a_{23}$ (see Definition \ref{def:sign for a right cusp} below), where $a_{ij}$'s correspond to the pairs of left endpoints of $T$, $t$ is the generator corresponding to the base point if the right cusp is marked and $t=1$ otherwise. Then $\iota_R:\mc{A}(T_R)\rightarrow \mc{A}(T)$ is given by
\begin{eqnarray*}
\iota_R(b_{12})
&=&a_{14}+a_{13}t^{-\sigma(a)}a_{24}+a_{12}at^{-\sigma(a)}a_{24}+a_{13}t^{-\sigma(a)}aa_{34}+a_{12}at^{-\sigma(a)}aa_{34}\\
&=&a_{14}+t^{-\sigma(a)}(a_{13}+a_{12}a)(a_{24}+aa_{34}).
\end{eqnarray*}
\end{example}

\begin{figure}[!htbp]
\begin{center}
\minipage{0.6\textwidth}
\includegraphics[width=\linewidth,height=0.8in]{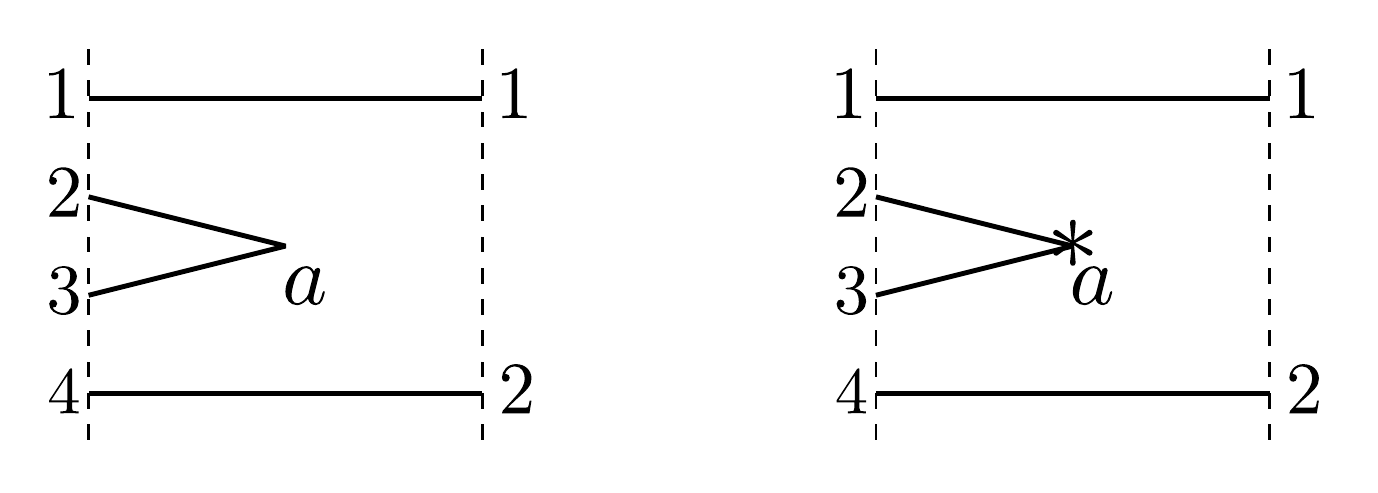}
\endminipage\hfill
\end{center}
\caption{Left: An elementary Legendrian tangle of an unmarked right cusp. Right: An elementary Legendrian tangle of a marked right cusp.}
\label{fig:right cusp}
\end{figure}

We introduce a sign at a right cusp, which will also be used later (see Lemma \ref{lem:aug and A-form MCS for elementary tangles}).
\begin{definition}\label{def:sign for a right cusp}
Given a right cusp $a$ of the oriented tangle front $T$, we define the \emph{sign} $\sigma=\sigma(a)$ of $a$ to be $1$ (resp. $-1$) if $a$ is a down (resp. up) cusp. See Figure \ref{fig:sign for a right cusp}.
\end{definition}

\begin{figure}[!htbp]
\begin{center}
\minipage{0.5\textwidth}
\includegraphics[width=\linewidth,height=0.6in]{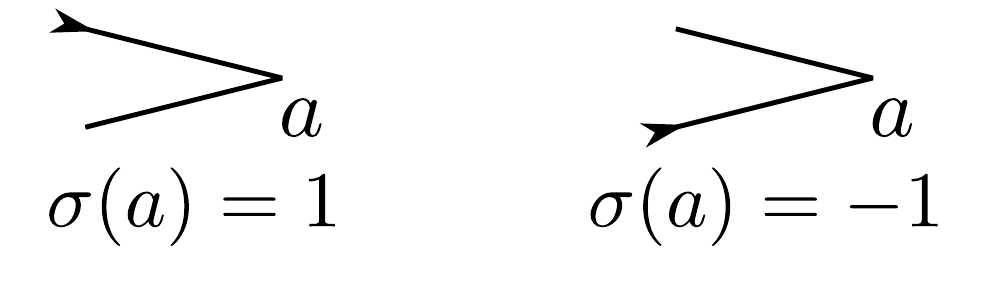}
\endminipage\hfill
\end{center}
\caption{Left: a down right cusp. Right: an up right cusp.}
\label{fig:sign for a right cusp}
\end{figure}

One key property of LCH DGAs for Legendrian tangles is the co-sheaf property:
\begin{proposition}[{\cite[Thm.6.13]{NRSSZ15}}]\label{prop:co-sheaf of DGAs}
If $U=L\cup_V R$ is the union of 2 open intervals $L, R$ with non-empty intersection $V$, then the diagram of co-restriction maps
\begin{equation}\label{eq:co-sheaf property}
\xymatrix{
\mc{A}(T|_V) \ar[r]^{\iota_{RV}} \ar[d]_{\iota_{LV}} & \mc{A}(T|_R) \ar[d]^{\iota_{UR}} \\
\mc{A}(T|_L) \ar[r]^{\iota_{UL}} & \mc{A}(T)
}
\end{equation}
gives a pushout square of $\mb{Z}/2r$-graded DGAs.
\end{proposition}
\begin{proof}
Again the same argument in the proof of Theorem 6.13 in \cite{NRSSZ15} (The case for Legendrian tangles in nearly plat positions) applies to the general case.
\end{proof}

\section{Augmentations for Legendrian tangles}\label{sec:Augmentations for tangles}

\subsection{Augmentation varieties and augmentation numbers}\label{subsec:augmentations for tangles}
Fix a Legendrian tangle $T$, with $\mb{Z}/2r$-valued Maslov potential $\mu$, base points $*_1,\ldots,*_B$ so that each connected component containing a right cusp has at least one base point. Denote the crossings, right cusps and pairs of left end-points by $\mathcal{R}=\{a_1,\ldots,a_N\}$. As always, the base points are assumed to be away from the crossings and left cusps of $T$. Let $n_L, n_R$ be the numbers of left and right end-points in $T$ respectively.

We define the LCH DGA $(\mathcal{A},\partial)$ as in the previous Section. So as an associative algebra we have $\mc{A}(T)=\mb{Z}[t_1^{\pm1},\dots,t_B^{\pm1}]<a_1,\ldots,a_N>$. Fix a nonnegative integer $m$ dividing $r$ and a base field $k$.

\begin{definition}
A \emph{$m$-graded (or $\mathbb{Z}/m$-graded) $k$-augmentation} of $\mathcal{A}$ is unital algebraic map $\epsilon:(\mathcal{A},\partial)\rightarrow (k,0)$ such that $\epsilon\circ \partial =0$, and for all $a$ in $\mc{A}$ we have $\epsilon(a)=0$ if $|a|\neq 0 (\mathrm{mod} m)$. Here $(k,0)$ is viewed as a DGA concentrated on degree 0 with zero differential. Morally, ``$\epsilon$ is a $\mb{Z}/m\mb{Z}$-graded DGA map".
\end{definition}

\begin{definition}
Define $\mr{Aug}_m(T,k)$ to be the set of $m$-graded $k$-augmentations of $\mc{A}(T)$. This defines an affine subvariety of $(k^{\times})^B\times k^N$, via the map
$$\mr{Aug}_m(T,k)\ni\epsilon\rightarrow (\epsilon(t_1,\ldots,\epsilon(t_B),\epsilon(a_1),\ldots,\epsilon(a_N)))\in(k^{\times})^B\times k^N$$
with the defining polynomial equations $\epsilon\circ\partial(a_i)=0, 1\leq i\leq N$ and $\epsilon(a_i)=0$ for $|a_i|\neq 0 (\mr{mod} m)$. This affine variety $\mr{Aug}_m(T,k)$ will be called the (full) \emph{$m$-graded augmentation variety} of $(T,\mu,*_1,\ldots,*_B)$.
\end{definition}

\begin{example}[The augmentation variety for trivial Legendrian tangles]\label{ex:aug. variety for lines}
Let $T$ be the trivial Legendrian tangle of $n$ parallel strands, labeled from top to bottom by $1,2,\ldots,n$, equipped a $\mb{Z}/2r$-valued Maslov potential $\mu$. The LCH DGA is $\mc{A}(T)=\mb{Z}<a_{ij},1\leq i<j\leq n>$, with the grading $|a_{ij}|=\mu(i)-\mu(j)-1$ and the differential given by formula (\ref{eqn:differential for pair of ends}). The $m$-graded augmentation variety $\mr{Aug}_m(T;k)$ is
\begin{eqnarray*}
\mr{Aug}_m(T;k)=\{(\epsilon(a_{ij}))_{1\leq i<j\leq n}|\epsilon\circ\partial a_{ij}=0, \text{ and $\epsilon(a_{ij})=0$ if $|a_{ij}|\neq 0 (\mr{mod} m)$}.\}
\end{eqnarray*}

On the other hand,
\begin{definition}\label{def:filtered module for trivial Legendrian tangles}
Associate to the trivial Legendrian tangle $(T,\mu)$, \emph{define} a canonical $\mb{Z}/m$-graded filtered $k$-module $C=C(T)$:
$C$ is the free $k$-module generated by $e_1,\ldots,e_{n}$ corresponding to the $n$ strands of $T$ with grading $|e_i|=\mu(i) (\mr{mod} m)$. Moreover, $C$ is equipped with a decreasing filtration $F^0\supset F^1\supset\ldots\supset F^{n}: F^iC=\mr{Span}\{e_{i+1},\ldots,e_{n}\}$.

\noindent{}\emph{Define} $B_m(T):=\mr{Aut}(C)$ to be the automorphism group of the $\mb{Z}/m$-graded filtered $k$-module $C$. \emph{Denote} $I=I(T):=\{1,2,\ldots,n\}$.
\end{definition}

Now, in the example, given any $m$-graded augmentation $\epsilon$ for $\mc{A}(T)$, we construct a $\mb{Z}/m$-graded chain complex $C(\epsilon)=(C,d(\epsilon))$: The differential $d=d(\epsilon)$ is filtration preserving, of degree $-1$ given by
\begin{eqnarray*}
<de_i,e_j>=0 \text{ for $i\geq j$ and $<de_i,e_j>=(-1)^{\mu(i)}\epsilon(a_{ij})$ for $i<j$}.
\end{eqnarray*}
Here $<de_i,e_j>$ denotes the coefficient of $e_j$ in $de_i$. The condition that $d$ is of degree $-1$ is equivalent to: $<de_i,e_j>=(-1)^{\mu(i)}\epsilon(a_{ij})=0$ if $\mu(i)-\mu(j)-1=|a_{ij}|\neq 0 (\mr{mod} m)$ for all $i<j$. The condition of the differential $d^2=0$ is equivalent to: for all $i<j$ have $<d^2e_i,e_j>=\sum_{i<k<j}<de_i,e_k><de_k,e_j>=0$, i.e. $\sum_{i<k<j}(-1)^{\mu(i)-\mu(k)}\epsilon(a_{ik})\epsilon(a_{kj})=\epsilon\circ\partial a_{ij}=0$.

Thus, we see that the map $\epsilon\rightarrow C(\epsilon)$ gives an \emph{isomorphism} between the augmentation variety $\mr{Aug}_m(T;k)$ and the set $MCS_m^A(T;k)$ of $\mb{Z}/m$-graded filtered chain complexes $(C,d)$, or equivalently, the set of filtration preserving degree $-1$ differentials $d$ of $C$. From now on, we will always use this identification (see also Section \ref{subsec:aug vs A-form MCS}).
\end{example}

Given the Legendrian tangle $(T,\mu)$ of $n$ parallel strands, $B_m(T)$ acts naturally on $\mr{Aug}_m(T;k)=MCS_m^A(T)$ via conjugation: given $\varphi\in B_m(T)$ and $(C,d)$ in $MCS_m^A(T;k)$, have $\varphi\cdot (C,d):=(C,\varphi\circ d\circ \varphi^{-1})$. In particular, the $B_m(T)$-orbit $B_m(T)\cdot (C,d)$ (or $B_m(T)\cdot d$) is simply the isomorphism classes of $d$.

\begin{lemma}[Barannikov normal form, See also \cite{Bar94, Lau15}]\label{lem:Barannikov normal form}
Let $(C,d)$ be any $\mb{Z}/m$-graded filtered chain complex over $k$, where $C=\mr{Span}_k\{e_1,\ldots,e_n\}$ is fixed with the decreasing filtration $F^0\supset F^1\supset\ldots\supset F^n$: $F^iC=\mr{Span}_k\{e_{i+1},\ldots,e_n\}$, then the isomorphism class of $(C,d)$ has a unique representative, say $(C,d_0)$, such that the matrix $(<d_0e_i,e_j>)_{i,j}$ has at most one nonzero entry in each row and column and moreover these are all $1$'s.
Equivalently, there're $2k$ distinct indices $i_1<j_1,\ldots,i_k<j_k$ in $\{1,\ldots,n\}$ for some $k$, such that $d_0e_{i_l}=e_{j_l}$ for $1\leq l\leq k$ and $d_0e_i=0$ otherwise.

The unique representative $(C,d_0)$ is called the \emph{Barannikov normal form} of $(C,d)$.
\end{lemma}

\begin{proof}
We divide the index set $I:=\{1,\ldots,n\}$ into 3 types: \emph{upper}, \emph{lower} and \emph{homological}.

For each $1\leq i\leq n$, an element of the form $c_ie_i+\sum_{k>i}c_ke_k$ is called \emph{$i$-admissible} if $c_i\neq 0$ and $c_k=0$ if $|e_k|\neq |e_i| (\mr{mod} m)$ for all $k>i$. In other words, the set of $i$-admissible elements is the same as $\mr{Aut}(C)\cdot e_i$, the image of $e_i$ under the automorphism group of the $\mb{Z}/m$-graded filtered $k$-module $C$. In particular, any automorphism of $C$ preserves the set of $i$-admissible elements.

\begin{itemize}
\item
$i$ is called \emph{$d$-closed} (or \emph{closed}) if there's an $i$-admissible element $x$ such that $dx=0$. Otherwise, $i$ is called \emph{$d$-upper} (or \emph{upper}).
\end{itemize}
To check the definition only depends on the isomorphism class of $d$: If $d'$ is another representative in the isomorphism class of $d$, so $d'=\varphi\cdot d=\varphi\circ d\circ \varphi^{-1}$ for some $\varphi\in\mr{Aut}(C)$. If $i$ is $d$-closed, say $dx=0$ with $x$ $i$-admissible, then $d'\varphi(x)=\varphi(dx)=0$ with $\varphi(x)$ $i$-admissible, hence $i$ is also $d'$-closed.

For each index $i$, and any $i$-admissible element $x$, we can write $dx=*_le_l+\sum_{k>l}*_ke_k$ for some $l>i$ with $*_l\neq 0$, i.e. $dx$ is $l$-admissible. If $dx=0$ (that is, $i$ is closed), then $l:=\infty$ and ``$dx$ is $\infty$-admissible" means $dx=0$. Now, \emph{define} $\rho_d(x):=l$.

If $d'=\varphi\cdot d$ is another representative, then $d'\varphi(x)=\varphi(dx)$ is also $l$-admissible, hence $\rho_{d'}(\varphi(x))=\rho_d(x)$.
For each index $i$, \emph{define}
\begin{eqnarray*}
\rho_d(i):=\mr{max}\{\rho_d(x)|x \text{ is $i$-admissible}\}
\end{eqnarray*}
By definition, $\rho_d(i)>i$. And, the previous identity shows that $\rho_d(i)$ only depends on the isomorphism class of $d$. So we can write $\rho(i)=\rho_d(x)$. Also, by definition, $i$ is upper if and only if $\rho(i)<\infty$.

\begin{itemize}
\item
$j$ is called \emph{lower}, if $j=\rho(i)$ for some upper index $i$.
\end{itemize}
If $j=\rho(i)$ is lower, then $j=\rho(x)$ for some $i$-admissible element $x$, hence
$dx=*_je_j+\sum_{k>j}*_ke_k$ is $j$-admissible. It follows that $d(*_je_j+\sum_{k>j}*_ke_k)=d^2x=0$. Therefore, $j$ is \emph{closed}.

\begin{itemize}
\item
$j$ is called \emph{homological}, if $j$ is closed but not lower.
\end{itemize}
As a consequence, we obtain a partition and a map associated to the isomorphism class of $d$
\begin{eqnarray}\label{eqn:partition associated to augmentations}
&&I=L\sqcup H\sqcup U\\
&&\rho: U\rightarrow L\nonumber
\end{eqnarray}
where $L, H$ and $U$ are the sets of lower, homological and upper indices respectively. We emphasize that the partition and the map depend only on the isomorphism class of $d$.
Moreover, $\rho:U\rightarrow L$ is a \emph{bijection}. By definition, it's clearly surjective. To show it's also injective: Otherwise, assume $i<i'$ are 2 upper indices such that $\rho(i)=\rho(i')=k$. In particular, $|e_i|=|e_{i'}|=|e_k|+1$. Then for some $i$-admissible $x$ and $i'$-admissible element $x'$ we have $k=\rho(i)=\rho(x)$ and $k=\rho(i')=\rho(x')$, that is, $dx=c_ke_k+\sum_{j>k}c_je_j$ and $dx'=c_k'e_k+\sum_{j>k}c_j'e_j$ are both $k$-admissible, i.e. $c_k\neq 0, c_k'\neq 0$. If follows that $d(c_k'x-c_kx')=\sum_{j>k}*_je_j$ and $c_k'x-c_kx'$ is still $i$-admissible. Hence, $\rho_d(c_k'x-c_kx')>k=\rho(i)=\mr{max}\{\rho_d(y)| y \text{ is $i$-admissible}\}$, contradiction.

Suppose $U=\{i_l,1\leq l\leq k|i_1<i_2<\ldots<i_k\}$ and $j_l:=\rho(i_l), 1\leq l\leq k$, then $L=\{j_l,1\leq l\leq k\}$.
By definition of $\rho$, for each $l$ there exists an $i_l$-admissible element, say $e_{i_l}'$, such that $e_{j_l}':=de_{i_l}'$ is $j_l$-admissible. We may even assume that $e_{i_l}'=e_{i_l}+\sum_{j>i_l}*_je_j$. For each $i$ in $H$, by definition, there exists an $i$-admissible element $e_i'=e_i+\sum_{j>i}*_je_j$ such that $de_i'=0$. We thus have constructed a set of elements $\{e_1',e_2',\ldots,e_n'\}$ in $C$ with $e_i'$ $i$-admissible, it follows that they form a basis of $C$. Define an automorphism $\varphi$ of $C$ by $\varphi(e_i')=e_i$, and take $d_0=\varphi\cdot d$. Then, $d_0e_{i}=\varphi(de_{i}')$. As a consequence, $d_0e_{i_l}=e_{j_l}$ for $1\leq l\leq k$ and $d_0e_i=0$ for $i\in H$. That is, $d_0$ is a Barannikov normal form of $d$.

Conversely, given a Barannikov normal form $d_0$ of $d$, there exist $2k$ distinct indices $i_1<j_1,\ldots, i_k<j_k$ such that $d_0e_{i_l}=e_{j_l}$ for $1\leq l\leq k$ and $d_0e_i=0$ otherwise. Apply the definition of the 3 types of indices with respect to $d_0$, we must have $U=\{i_1,\ldots, i_k\}$ and $\rho(i_l)=\mr{max}\{\rho_{d_0}(x)|x \text{ is $i_l$-admissible}\}=j_l$, so $L=\{j_1,\ldots,j_k\}$. Hence, $d_0$ is uniquely determined by the partition $I=L\sqcup H\sqcup U$ and the bijection $\rho: U\xrightarrow[]{\sim}L$, which are determined by the isomorphism class of $d$.
\end{proof}

\begin{definition}\label{def:isomorphism type for trivial Legendrian tangles}
Given a trivial Legendrian tangle $(T,\mu)$, a partition $I(T)=U\sqcup H\sqcup L$ together with a bijection $\rho:U\xrightarrow[]{\sim} L$ as in the proof of the previous lemma (see Equation (\ref{eqn:partition associated to augmentations})), will be called an \emph{$m$-graded isomorphism type} of $T$, denoted by $\rho$ for simplicity. Note: $\rho(i)>i$ and $|e_{\rho(i)}|=|e_i|-1 (\mr{mod} m)$ for all $i\in U$.
\end{definition}

\begin{remark}\label{rem:normal rulings via non-degenerate augmentations}
By Lemma \ref{lem:Barannikov normal form}, each $m$-graded isomorphism type $\rho$ of $T$ determines an unique isomorphism class $\mc{O}_m(\rho;k)$ of $\mb{Z}/m$-graded filtered $k$-complexes $(C(T),d)$. In other words, $\mc{O}_m(\rho;k)$ is the $B_m(T)$-orbit of the \emph{canonical augmentation} $\epsilon_{\rho}$ (equivalently, the \emph{Barannikov normal form} $d_{\rho}$ determined by $\rho$), using the identification in Example \ref{ex:aug. variety for lines}. We thus obtain a decomposition of the augmentation variety for the trivial Legendrian tangle $(T,\mu)$:
\begin{eqnarray}
\mr{Aug}_m(T;k)=\sqcup_{\rho}\mc{O}_m(\rho;k)
\end{eqnarray}
where $\rho$ runs over all $m$-graded isomorphism types of $T$.

In addition, take a $m$-graded augmentation $\epsilon$ of $\mc{A}(T)$, or equivalently the $m$-graded filtered chain complex $C(\epsilon)=(C,d(\epsilon))$. Suppose $\epsilon$ is \emph{acyclic}, meaning that $(C,d(\epsilon))$ is acyclic or $H=\emptyset$ in the partition $I=L\sqcup H\sqcup U$ associated to $d(\epsilon)$. Then, the associated $m$-graded isomorphism type $\rho:U\xrightarrow[]{\sim}L$ can be identified with an \emph{$m$-graded normal ruling} (denoted by the same $\rho$) of $T$.
\end{remark}

\begin{remark}\label{rem:standard Morse complex via augmentations}
In Lemma \ref{lem:Barannikov normal form}, given any complex $(C,d)$ (or the corresponding augmentation $\epsilon$), which determines a partition $I=U\sqcup L\sqcup H$ and a bijection $\rho:U\xrightarrow[]{\sim}L$, say $U=\{i_1<i_2<\ldots<i_k\}$ and $\rho(i_l)=j_l$. Then $\varphi^{-1}\cdot d=d_0$ is the Barannikov normal form for some $\varphi\in Aut(C)$. Can take the decomposition $\varphi=D\circ \varphi_0$, where $D$ is diagonal and $\varphi_0$ is unipotent, i.e. $\varphi_0(e_i)=e_i+\sum_{j>i}*_je_j$. Then $(\varphi_0^{-1}\cdot d) (e_{i_l})=c_le_{j_l}$ for $c_l\in k^*$ and $1\leq l\leq k$, and $(\varphi_0^{-1}\cdot d)(e_j)=0$ for the remaining cases. Such a complex $(C,\varphi_0^{-1}\cdot d)$ (or the corresponding augmentation $\varphi_0^{-1}\cdot \epsilon$) is called \emph{standard}, and we say $(C,\varphi_0^{-1}\cdot d)$ is \emph{standard with respect to $\rho$}.

In fact, the unipotent automorphism $\varphi_0$ can be taken to be canonical. See Lemma \ref{lem:canonical automorphism via non-degenerate augmentation}.
\end{remark}

Augmentation varieties for Legendrian tangles also satisfy a sheaf property, induced by the co-sheaf property of LCH DGAs in Section \ref{subsubsec:co-sheaf for DGAs}. More precisely, we have

\begin{definition/proposition}
Let $T$ a Legendrian tangle in $J^1U$.
\begin{enumerate}[label=(\arabic*)]
\item
Let $V$ be an open subinterval of $U$, then the co-restriction of DGAs $\iota_{UV}:\mc{A}(T|_V)\rightarrow \mc{A}(T)$ induces a restriction $r_{VU}=\iota_{VU}^*:\mr{Aug}_m(T;k)\rightarrow \mr{Aug}_m(T|_V;k)$.
\item
If $U=L\cup_V R$ is the union of 2 open intervals $L, R$ with non-empty intersection $V$, then the diagram of restriction maps
\begin{equation}\label{eq:sheaf property}
\xymatrix{
\mr{Aug}_m(T;k) \ar[r]^{r_{RU}} \ar[d]_{r_{LU}} & \mr{Aug}_m(T|_R;k) \ar[d]^{r_{VR}} \\
\mr{Aug}_m(T|_L;k) \ar[r]^{r_{VL}} & \mr{Aug}_m(T|_V;k)
}
\end{equation}
gives a fiber product of augmentation varieties.
\end{enumerate}
\end{definition/proposition}

Take the left and right pieces of $T$, called $T_L, T_R$ respectively. We get 2 restrictions of augmentation varieties $r_L=\iota_L^*:\mr{Aug}_m(T)\rightarrow \mr{Aug}_m(T_L)$ and $r_R=\iota_R^*:\mr{Aug}_m(T)\rightarrow \mr{Aug}_m(T_R)$. We can then define some subvarieties:
\begin{definition}\label{def:aug varieties with boundary conditions}
Given $m$-graded isomorphism types $\rho_L, \rho_R$ for $T_L, T_R$ respectively, and $\epsilon_L\in\mc{O}_m(\rho_L;k)$. Define the varieties
\begin{eqnarray*}
&&\mr{Aug}_m(T,\epsilon_L,\rho_R;k):=\{\epsilon_L\}\times_{\mr{Aug}_m(T_L;k)}\times \mr{Aug}_m(T;k)\times_{\mr{Aug}_m(T_R;k)}\times\mc{O}_m(\rho_R;k)\\
&&\mr{Aug}_m(T,\rho_L,\rho_R;k):=\mc{O}_m(\rho_L;k)\times_{\mr{Aug}_m(T_L;k)}\times \mr{Aug}_m(T;k)\times_{\mr{Aug}_m(T_R;k)}\times\mc{O}_m(\rho_R;k)
\end{eqnarray*}
$\mr{Aug}_m(T,\epsilon_L,\rho_R;k)$ will be called the \emph{$m$-graded augmentation variety with boundary conditions $(\epsilon_L,\rho_R)$} for $T$. When $\epsilon_L=\epsilon_{\rho_L}$ is the \emph{canonical augmentation} of $T_L$ corresponding to the Barannikov normal form determined by $\rho_L$, we will call $\mr{Aug}_m(T,\epsilon_{\rho_L},\rho_R;k)$ the \emph{$m$-graded augmentation variety (with boundary conditions $(\rho_L,\rho_R)$}) of $T$.
\end{definition}

By definition, we immediately obtain a decomposition of the full augmentation variety
\begin{eqnarray}\label{eqn:decomp for the full aug var}
\mr{Aug}_m(T;k)=\sqcup_{\rho_L,\rho_R}\mr{Aug}_m(T,\rho_L,\rho_R;k)
\end{eqnarray}
where $\rho_L,\rho_R$ run over all $m$-graded isomorphism types of $T_L,T_R$ respectively.

Note that the augmentation variety $\mr{Aug}_m(T,k)$ itself is not a Legendrian isotopy invariant.
However, \emph{assume} the numbers $n_L,n_R$ of left endpoints and right endpoints of $T$ are both \emph{even}. We can define
\begin{definition}\label{def:augmentation number}
Let $\mb{F}_q$ be any finite field, and $\rho_L,\rho_R$ be \emph{$m$-graded isomorphism types} of $T_L,T_R$ respectively. The \emph{$m$-graded augmentation number (with boundary conditions $(\rho_L,\rho_R)$)} of $T$ over $\mb{F}_q$ is
\begin{eqnarray}
\mr{aug}_m(T,\rho_L,\rho_R;q):=
q^{-\mr{dim}_{\mb{C}}\mr{Aug}_m(T,\epsilon_{\rho_L},\rho_R;\mb{C})}|\mr{Aug}_m(T,\epsilon_{\rho_L},\rho_R;\mb{F}_q)|
\end{eqnarray}
where $|\mr{Aug}_m(T,\epsilon_{\rho_L},\rho_R;\mb{F}_q)|$ is simply the counting of $\mb{F}_q$-points.
\end{definition}

\begin{remark}
Alternatively, we can use $\mr{Aug}_m(T,\rho_L,\rho_R;k)$ instead of $\mr{Aug}_m(T,\epsilon_{\rho_L},\rho_R;k)$ to define the augmentation number. However, this alternative definition only differs from the previous one by a normalized factor $q^{-\mr{dim}\mc{O}_m(\rho_L;k)}|\mc{O}_m(\rho_L;\mb{F}_q)|=(\frac{q-1}{q})^{|L|}$, where $L$ comes from the partition $I(T_L)=U\sqcup H\sqcup L$ determined by $\rho_L$. See Corollary \ref{cor:structures for oribits determined by normal rulings}.
\end{remark}

In the next subsection, we will see that the augmentation numbers defined above are Legendrian isotopy invariants. However, for the purpose of clarity, we will now restrict ourselves to the case when $\rho_L,\rho_R$ are $m$-graded normal rulings. In particular, this ensures that $T$ has even left and even right endpoints. In Section \ref{sec:Ruling vs aug:general case}, we will come back to the general case (the part related to $\mr{aug}_m(T,\rho_L,\rho_R;q)$ defined here is Section \ref{subsec:alternative generalization}).

\subsection{Computation for augmentation numbers}\label{subsec:computation of aug number}

Given a Legendrian tangle $(T,\mu)$. For the moment, we will assume $T$ is placed with $B$ base points so that \emph{each right cusp is marked}. Label the crossings, cusps and base points away from the right cusps of $T$ by $q_1,\ldots,q_n$ with $x$-coordinates, from left to right. Let $x_0<x_1<\ldots<x_{n}$ be the $x$-coordinates which cut $T$ into elementary tangles. That is, $x_0$ and $x_n$ are the the $x$-coordinates of the left and right end-points of $T$, and $x_{i-1}<x_{q_i}<x_i$ for all $1\leq i\leq n$. Let $T_i=T|_{\{x_0<x<x_i\}}$ and $E_i:=T|_{\{x_{i-1}<x<x_i\}}$ be the $i$-th elementary tangle around $q_i$, then $T=T_n=E_1\circ E_2\circ\ldots\circ E_n$ is the composition of $n$ elementary tangles.

Fix $m$-graded normal rulings $\rho_L,\rho_R$ of $T_L, T_R$ respectively. Fix $\epsilon_L\in\mc{O}_m(\rho_L;k)$.

\begin{definition}\label{def:augmentaion varieties via normal rulings}
For any $m$-graded normal ruling $\rho$ of $T$ such that $\rho|_{T_L}=\rho_L$ and $\rho|_{T_R}=\rho_R$, denote $\rho_i:=\rho|_{(T_i)_R=(T_{i+1})_L}$ for $0\leq i\leq n$. In particular, $\rho_0=\rho_L, \rho_n=\rho_R$. Define the variety
\begin{eqnarray*}
\mr{Aug}_m^{\rho}(T,\epsilon_L)&:=&
\mr{Aug}_m(E_1,\epsilon_L,\rho_1)\times_{\mc{O}_m(\rho_1)}
\ldots\times_{\mc{O}_m(\rho_{n-1})}\mr{Aug}_m(E_n,\rho_{n-1},\rho_n)\\
\mr{Aug}_m^{\rho}(T,\rho_L)&:=&
\mr{Aug}_m(E_1,\rho_0,\rho_1)\times_{\mc{O}_m(\rho_1)}
\ldots\times_{\mc{O}_m(\rho_{n-1})}\mr{Aug}_m(E_n,\rho_{n-1},\rho_n)
\end{eqnarray*}
while for simplicity we have ignored the coefficient field $k$.
\end{definition}

\begin{remark}\label{rem:nondegeneracy of augmentations}
Given any elementary Legendrian tangle $E$: a single crossing, a left cusp, a (marked or unmarked) right cusp, or $2n$ parallel strands with a single base point, let $\epsilon$ be any $m$-graded augmentation of $\mc{A}(E)$ and denote $\epsilon_L:=\epsilon|_{E_L},\epsilon_R:=\epsilon|_{E_R}$. If $\epsilon_L$ is acyclic (see Remark \ref{rem:normal rulings via non-degenerate augmentations}), then so is $\epsilon_R$.
By induction, this result then generalizes to all Legendrian tangles. For a justification, see Corollary \ref{cor:nondegeneracy of augmentations}.
\end{remark}

We then obtain a partition into subvarieties
\begin{eqnarray}\label{eqn:partition of aug var}
\mr{Aug}_m(T,\epsilon_L,\rho_R;k)=\sqcup_{\rho}\mr{Aug}_m^{\rho}(T,\epsilon_L;k)
\end{eqnarray}
where $\rho$ runs over all $m$-graded normal rulings of $T$ such that $\rho|_{T_L}=\rho_L$ and $\rho|_{T_R}=\rho_R$.

Consider the natural map
\begin{eqnarray}\label{eqn:iterated fibration for augmentation varieties}
P_n:\mr{Aug}_m^{\rho}(T_n,\epsilon_L;k)\rightarrow \mr{Aug}_m^{\rho|_{T_{n-1}}}(T_{n-1},\epsilon_L;k)
\end{eqnarray}
Clearly the fibers are $\mr{Aug}_m(E_n,\epsilon_{n-1},\rho_n;k)$,
where $\epsilon_{n-1}\in\mc{O}_m(\rho_{n-1};k)$.

\begin{lemma}\label{lem:augmentation varieties for elementary tangles}
Let $(E,\mu)$ be an elementary Legendrian tangle: a single crossing $q$, a left cusp $q$, a marked right cusp $q$ or $2n$ parallel strands with a single base point $*$. Let $\rho$ be a $m$-graded normal ruling of $E$, denote $\rho_L:=\rho|_{E_L},\rho_R:=\rho|_{E_R}$. Take any $\epsilon_L$ in $\mc{O}_m(\rho_L;k)$, have
\begin{eqnarray*}
\mr{Aug}_m(E,\epsilon_L,\rho_R;k)\cong (k^*)^{-\chi(\rho)+B}\times k^{r(\rho)}
\end{eqnarray*}
where $B$ is the number of base points in $E$, $\chi(\rho)=s(\rho)-c_R$, $c_R$ the number of right cusps in $E$. And, $s(\rho)$ and $r(\rho)$ are defined as in Definition \ref{def:switches and returns}.
\end{lemma}

We will not show the lemma until the Section \ref{subsec:aug vars for elementary Leg tangles}.
\begin{remark}
In fact, for any Legendrian tangle $T=T_n$ as above, one can show that
\begin{eqnarray*}
\mr{Aug}_m^{\rho}(T,\epsilon_L;k)\cong (k^*)^{-\chi(\rho)+B}\times k^{r(\rho)}
\end{eqnarray*}
See Theorem \ref{thm:augmentation varieties for Legendrian tangles}.
However, for our purpose of points-counting, the previous lemma will suffice.
\end{remark}

Assuming the lemma, we see that the map $P_n:\mr{Aug}_m^{\rho}(T_n,\epsilon_L;k)\rightarrow \mr{Aug}_m^{\rho|_{T_{n-1}}}(T_{n-1},\epsilon_L;k)$ is surjective with smooth isomorphic fibers $(k^*)^{-\chi(\rho|_{E_n})+B(E_n)}\times k^{r(\rho|_{E_n})}$, where $B(E_n)$ denotes the number of base points in $E_n$. It follows that
\begin{eqnarray*}
&&\mr{dim}\mr{Aug}_m^{\rho}(T_n,\epsilon_L)=
\mr{dim}\mr{Aug}_m^{\rho|_{T_{n-1}}}(T_{n-1},\epsilon_L)-\chi(\rho|_{E_n})+B(E_n)+r(\rho|_{E_n})\\
&&|\mr{Aug}_m^{\rho}(T_n,\epsilon_L;\mb{F}_q)|
=|\mr{Aug}_m^{\rho|_{T_{n-1}}}(T_{n-1},\epsilon_L;\mb{F}_q)|(q-1)^{-\chi(\rho|_{E_n})+B(E_n)}q^{r(\rho|_{E_n})}
\end{eqnarray*}

So by induction, we obtain
\begin{eqnarray}
&&\mr{dim}Aug_m^{\rho}(T,\epsilon_L;k)=-\chi(\rho)+B-r(\rho)\\
&&|\mr{Aug}_m^{\rho}(T,\epsilon_L;\mb{F}_q)|=(q-1)^{-\chi(\rho)+B}q^{r(\rho)}\nonumber
\end{eqnarray}

As a consequence of Equation (\ref{eqn:partition of aug var}), we then have
\begin{lemma}\label{lem:augmentation number formula}
Given a Legendrian tangle $(T,\mu)$ with $B$ base points so that each right cusp is marked, let $\rho_L, \rho_R$ be $m$-graded normal rulings of $T_L,T_R$ respectively, then for any $\epsilon_L\in\mc{O}_m(\rho_L;k)$, have
\begin{eqnarray}\label{eqn:dimension formula for augmentation varieties}
\mr{dim}\mr{Aug}_m(T,\epsilon_L,\rho_R;k)=\mr{max}_{\rho}\{-\chi(\rho)+B+r(\rho)\}
\end{eqnarray}
and the augmentation number is given by
\begin{eqnarray}\label{eqn:augmentation number formula}
\mr{aug}_m(T,\rho_L,\rho_R;q)
=q^{-\mr{max}_{\rho}\{-\chi(\rho)+B+r(\rho)\}}\sum_{\rho}(q-1)^{-\chi(\rho)+B}q^{r(\rho)}
\end{eqnarray}
where $\rho$ runs over all $m$-graded normal rulings such that $\rho|_{T_L}=\rho_L, \rho|_{T_R}=\rho_R$.
\end{lemma}

\begin{corollary}[Invariance of augmentation numbers]\label{cor:invariance of augmentation numbers}
In the setting of the previous lemma with $B$ fixed, then the augmentation numbers $\mr{aug}_m(T,\rho_L,\rho_R;q)$ are Legendrian isotopy invariants.
\end{corollary}
\begin{proof}
Given a Legendrian isotopy $h:T\rightarrow T'$, by Lemma \ref{lem:filling_surface} there's a canonical bijection $\phi_h:\mr{NR}_T^m\xrightarrow[]{\sim}\mr{NR}_{T'}^m$ between the sets of $m$-graded normal rulings of $T,T'$, which commutes with the restriction to left and right pieces, and $\chi(\phi_h(\rho))=\chi(\rho)$ for any $m$-graded normal ruling of $T$. Moreover, by Remark \ref{rem:types of crossings under Legendrian isotopy}, there's a constant $C_r$ which only depend on $T$ and $h$, such that $r(\phi_h(\rho))=r(\rho)+C_r$. Apply the previous lemma, we get
\begin{eqnarray*}
\mr{dim}\mr{Aug}_m(T',\rho_L,\rho_R;k)&=&\mr{max}_{\rho}\{-\chi(\phi_h(\rho))+B+r(\phi_h(\rho))\}\\
&=&\mr{max}_{\rho}\{-\chi(\rho)+B+r(\rho)\}+C_r\\
&=&\mr{dim}\mr{Aug}_m(T,\rho_L,\rho_R;k)+C_r
\end{eqnarray*}
where $\rho$ runs over all $m$-graded normal rulings of $T$ such that $\rho|_{T_L}=\rho_L,\rho|_{T_R}=\rho_R$, and
\begin{eqnarray*}
\mr{aug}_m(T',\rho_L,\rho_R;q)
&=&q^{-\mr{dim}\mr{Aug}_m(T,\rho_L,\rho_R;k)-C_r}\sum_{\rho}(q-1)^{-\chi(\phi_h(\rho))+B}q^{r(\phi_h(\rho))}\\
&=&q^{-\mr{dim}\mr{Aug}_m(T,\rho_L,\rho_R;k)-C_r}\sum_{\rho}(q-1)^{-\chi(\rho)+B}q^{r(\rho)+C_r}\\
&=&\mr{aug}_m(T,\rho_L,\rho_R;q)
\end{eqnarray*}
where $\rho$ runs as above.
\end{proof}

\subsection{Ruling polynomials compute augmentation numbers}\label{subsec:Ruling vs aug: acyclic case}
\begin{theorem}\label{thm:counting for tangles}
Let $T$ be a Legendrian tangle equipped with a $\mb{Z}/2r$-valued Maslov potential $\mu$ and $B$ base points so that each connected component containing a right cusp has at least one base point. Fix a nonnegative integer $m$ dividing $2r$ and $m$-graded normal rulings $\rho_L,\rho_R$ of $T_L,T_R$ respectively, then the augmentation numbers and Ruling polynomials of $(T,\mu)$ are related by
\begin{equation}\label{eqn:aug vs Ruling poly}
\mr{aug}_m(T,\rho_L,\rho_R;q)=q^{-\frac{d+B}{2}}z^B<\rho_L|R_T^m(z)|\rho_R>
\end{equation}
where $q$ is the order of a finite field $\mb{F}_q$, $z=q^{\frac{1}{2}}-q^{-\frac{1}{2}}$, $d$ is the maximal degree in $z$ of $<\rho_L|R_T^m(z)|\rho_R>$.
\end{theorem}

\begin{proof}
Firstly, we prove the theorem when \emph{each right cusp is marked} in $T$.
We need the following direct generalization of \cite[lem.3.5]{HR15}

\begin{lemma}\label{lem:a constant function of rulings}
Let $(T,\mu)$ be any Legendrian tangle and fix $m$-graded normal rulings $\rho_L,\rho_R$ of $T_L,T_R$ respectively. Let $\rho$ and $\rho'$ be any two $m$-graded normal rulings of $T$ which restricts to $\rho_L$ (resp. $\rho_R$) on $T_L$ (resp. $T_R$), then
\begin{equation*}
-\chi(\rho)+2r(\rho)=-\chi(\rho')+2r(\rho')
\end{equation*}
\end{lemma}
Note: Unlike \cite[lem.3.5]{HR15}, we do not assume $T$ to have nearly plat front diagram. We will postpone the proof of the lemma until the end of this subsection.

Assuming Lemma \ref{lem:a constant function of rulings}, we prove Theorem \ref{thm:counting for tangles}. Fix $\rho_0$ such that $\mr{dim} \mr{Aug}_m(T,\rho_L,\rho_R)=-\chi(\rho_0)+B+r(\rho_0)$. It follows from lemma \ref{lem:a constant function of rulings} that $-\chi(\rho_0)$ is also maximal, hence $d=-\chi(\rho_0)=\mr{max.deg}_z <\rho_L|R_T^m(z)|\rho_R>$. For any $m$-graded normal ruling $\rho$, Lemma \ref{lem:a constant function of rulings} implies that $r(\rho)-r(\rho_0)=\frac{1}{2}(d+\chi(\rho))$. Plug this into equation (\ref{eqn:augmentation number formula}), we obtain
\begin{eqnarray*}
aug_m(T,\rho_L,\rho_R;q)&=q^{-d-B-r(\rho_0)}\sum_{\rho}(q-1)^{-\chi(\rho)+B}q^{r(\rho)}\\
&=q^{-\frac{d+B}{2}}\sum_{\rho}(q^{\frac{1}{2}}-q^{-\frac{1}{2}})^{-\chi(\rho)+B}\\
&=q^{-\frac{d+B}{2}}z^B<\rho_L|R_T^m(z)|\rho_R>
\end{eqnarray*}
where $z=q^{\frac{1}{2}}-q^{-\frac{1}{2}}$ and $\rho$ runs over all $m$-graded normal rulings of $T$ such that $\rho|_{T_L}=\rho_L,\rho|_{T_R}=\rho_R$.

In general, the theorem reduces to the previous case via Lemma \ref{lem:dependence of aug numbers on base points} below.
\end{proof}

\begin{lemma}[Dependence on the base points of augmentation numbers]\label{lem:dependence of aug numbers on base points}
As in the previous theorem, let $(T,\mu)$ be a Legendrian tangle with $B$ base points $*_1,\ldots,*_B$ so that each connected component containing a right cusp has at least one base point. Fix $m$-graded normal rulings $\rho_L,\rho_R$ of $T_L,T_R$ respectively, then the \emph{normalized augmentation number}
\begin{eqnarray*}
\mr{N.aug}_m(T,\rho_L,\rho_R;q):=q^{\frac{d+B}{2}}z^{-B}\mr{aug}_m(T,\rho_L,\rho_R;q)
\end{eqnarray*}
is independent of the number and positions of the base points on $T$.
\end{lemma}

\begin{proof}
To express the explicit dependence on the base points, we write $\mr{aug}_m(T,\rho_L,\rho_R;q)=\mr{aug}_m(T,*_1,\ldots,*_B,\rho_L,\rho_R;q)$ and  $\mr{N.aug}_m(T,\rho_L,\rho_R;q)=\mr{N.aug}_m(T,*_1,\ldots,*_B,\rho_L,\rho_R;q)$.

Firstly, we show that the (normalized) augmentation number is \emph{independent of the positions of the base points} in each connected component of $T$. It suffices to show that: Let $(*_1,\ldots,*_B)$ and $(*_1',\ldots,*_B')$ be 2 collections of base points on $T$, which are identical except that for some $i$, when $*_i'$ is the result of sliding $*_i$ across a crossing of $Res(T)$. Then $\mr{aug}_m(T,*_1,\ldots,*_B,\rho_L,\rho_R;q)=\mr{aug}_m(T,*_1',\ldots,*_B',\rho_L,\rho_R;q)$. Notice that a base point on a right cusp corresponds to a base point on the boundary of the ``invisible" disk after resolution.

Suppose $*_i,*_i'$ lie on the opposite sides of $a$ in $Res(T)$ and the orientation of $T$ goes from $*_i$ to $*_i'$, where $a$ is a crossing or right cusp of $T$. We firstly assume the strand containing $*_i,*_i'$ is the over-strand at $a$ of $Res(T)$. If $u$ is an \emph{admissible disk} as in Definition \ref{def:admissible_disks_via_fronts} with an initial vertex at $a$, and $w(u), w'(u)$ are the weights of $u$ in the DGAs $(\mc{A}(T,*_1,\ldots,*_B),\partial), (\mc{A}(T,*_1',\ldots,*_B'),\partial')$ respectively. Then $w(u)=t_i^{-1}w'(u)$, i.e. $\partial' (t_i^{-1}a)=\partial a$.
If $u$ is an admissible disk with at least one negative vertex at $a$, then $w'(u)$ is the result of replacing each $a$ by $t_i^{-1}a$ in $w(u)$. In other words, we have an isomorphism of $\mb{Z}/2r$-graded DGAs $\phi:\mc{A}(T,*_1,\ldots,*_B)\xrightarrow[]{\sim}\mc{A}(T,*_1',\ldots,*_B')$ given by $\phi(a)=t_i^{-1}a, \phi(a')=a'$ for all generators $a'\neq a$, and $\phi(t_j)=t_j$. It follows that $\phi$ induces an isomorphism $\phi^*:\mr{Aug}_m(T,*_1',\ldots,*_B',\epsilon_{\rho_L},\rho_R;k)
\xrightarrow[]{\sim}\mr{Aug}_m(T,*_1,\ldots,*_B,\epsilon_{\rho_L},\rho_R;k)$ defined by $\phi^*\epsilon'=\epsilon'\circ\phi$. Notice that $\phi'$ only changes the values of augmentations at $a$, the boundary condition $(\epsilon_{\rho_L},\rho_R)$ is indeed preserved by $\phi^*$. Now, by definition $\mr{aug}_m(T,*_1,\ldots,*_B,\rho_L,\rho_R;q)=\mr{aug}_m(T,*_1',\ldots,*_B',\rho_L,\rho_R;q)$.

If the strand containing $*_i,*_i'$ is the under-strand at $a$ of $Res(T)$. A similar argument shows that $\phi:\mc{A}(T,*_1,\ldots,*_B)\xrightarrow[]{\sim}\mc{A}(T,*_1',\ldots,*_B')$, given by $\phi(a)=at_i, \phi(a')=a'$ for $a'\neq a$ and $\phi(t_j)=t_j$, defines an isomorphism of $\mb{Z}/2r$-graded DGAs. Again, the desired equality follows as in the previous case.

Secondly, we show that the normalized augmentation number is \emph{independent of the number of base points} on $T$. By the first half of the result proved above, it suffices to show that: Let $*_1,\ldots,*_B,*_{B+1}$ a collection of base points on $T$ such that $*_B,*_{B+1}$ lie in a small neighborhood of $T$ avoiding the crossings, cusps and other base points, then $\mr{N.aug}_m(T,*_1,\ldots,*_B,\rho_L,\rho_R;q)=\mr{N.aug}_m(T,*_1,\ldots,*_B,*_{B+1},\rho_L,\rho_R;q)$, or equivalently,
\begin{eqnarray}\label{eqn:aug number and number of base points}
\mr{aug}_m(T,*_1,\ldots,*_{B+1},\rho_L,\rho_R;q)
=\frac{(q-1)}{q}\mr{aug}_m(T,*_1,\ldots,*_B,\rho_L,\rho_R;q)
\end{eqnarray}

In this case, there's a natural morphism of $\mb{Z}/2r$-graded DGAs $\phi:\mc{A}(T,*_1,\ldots,*_B)\rightarrow\mc{A}(T,*_1,\ldots,*_{B+1})$ given by $\phi(a)=a$ for all generators $a$, $\phi(t_i)=t_i$ for $i<B$ and $\phi(t_B)=t_{B}t_{B+1}$. Indeed, we obtain an isomorphism of DGAs $\Phi:\mc{A}(T,*_1,\ldots,*_B)[t,t^{-1}]\xrightarrow[]{\sim}\mc{A}(T,*_1,\ldots,*_{B+1})$ given by $\Phi(a)=a, \phi(t_i)=t_i,i<B, \Phi(t_B)=t_{B}t_{B+1}$ and $\Phi(t)=t_{B+1}$.
Hence, we obtain an induced isomorphism
\begin{eqnarray*}
\Phi^*:\mr{Aug}_m(T,*_1,\ldots,*_{B+1},\epsilon_{\rho_L},\rho_R;k)\xrightarrow[]{\sim} \mr{Aug}_m(T,*_1,\ldots,*_B,\epsilon_{\rho_L},\rho_R;k)\times k^*
\end{eqnarray*}
given by $\Phi^*:=\phi^*\times e_{B+1}$, with $\phi^*\epsilon(a)=\epsilon(a)$ for all generators $a$, $\phi^*\epsilon(t_i)=\epsilon(t_i)$ for $i<B$ and $\phi^*\epsilon(t_B)=\epsilon(t_B)\epsilon(t_{B+1})$, and $e_{B+1}(\epsilon)=\epsilon(t_{B+1})$. By definition of augmentation numbers,  it then follows that the equality (\ref{eqn:aug number and number of base points}) holds.
\end{proof}

Now, let's prove Lemma \ref{lem:a constant function of rulings}. For any $m$-graded normal ruling $\rho$ of $(T,\mu)$, \emph{define} $r'(\rho)$ to be the number of $m$-graded returns of $\rho$. It suffices to show $-\chi(\rho)+2r'(\rho)=-\chi(\rho')+2r'(\rho')$ for any $\rho, \rho'$ as in Lemma \ref{lem:a constant function of rulings}. However, $-\chi(\rho)=s(\rho)-c_R$ implies
\begin{eqnarray*}
-\chi(\rho)+2r'(\rho)&=&(s(\rho)+r'(\rho)+d(\rho))-c_R+r'(\rho)-d(\rho)\\
&=&r_m-c_R+r'(\rho)-d(\rho)
\end{eqnarray*}
where $r_m$ is the number of crossings of the front $T$ of degree $0$ modulo $m$. Hence, Lemma \ref{lem:a constant function of rulings} is a consequence of the following

\begin{proposition}\label{prop:r-d is independent of normal rulings}
Let $(T,\mu)$ be any Legendrian tangle and fix $m$-graded normal rulings $\rho_L,\rho_R$ of $T_L,T_R$ respectively. Then for any $m$-graded normal ruling $\rho$ such that $\rho|_{T_L}=\rho_L,\rho|_{T_R}=\rho_R$, $r'(\rho)-d(\rho)$ is independent of $\rho$.
\end{proposition}

Before the proof, let's firstly make some definitions. For any $m$-graded isomorphism type (Definition \ref{def:isomorphism type for trivial Legendrian tangles}) $\rho$ of a trivial Legendrian tangle $E$ of $n$ parallel strands. So $\rho$ determines a partition $I=U\sqcup L\sqcup H$ and a bijection $\rho:U\xrightarrow[]{\sim}L$, where $I=I(E)=\{1,2,\ldots,n\}$ is the set of left endpoints of $E$. Notice that $H=\emptyset$ when $\rho$ is a $m$-graded normal ruling (Remark \ref{rem:normal rulings via non-degenerate augmentations}). For each $i$ in $H$, we take $\rho(i):=\infty$.
Now, we define the subsets $I(i),i\in I$, $A(i)=A_{\rho}(i),i\in U\sqcup H$ of $I$, and an index $A(\rho)$, depending on $\rho$ as follows:

\begin{definition}\label{def:subsets of I determined by normal rulings}
For any $i\in I$, define
\begin{eqnarray*}
I(i):=\{j\in I|j>i, \mu(j)=\mu(i) (\mr{mod} m)\}.
\end{eqnarray*}
Note: $I(i)$ is independent of $\rho$. Now for any $i\in U\sqcup H$, define
\begin{eqnarray*}
A(i)=A_{\rho}(i):=\{j\in U\sqcup H| j\in I(i) \text{ and $\rho(j)<\rho(i)$}\}.
\end{eqnarray*}
Note: for any $j\in A(i)$, have $\rho(j)<\rho(i)\leq\infty$, hence we necessarily have $j\in U$.
Now, define $A(\rho)\in\mb{N}$ by
\begin{eqnarray*}
A(\rho):=\sum_{i\in U\sqcup H}|A(i)|+\sum_{i\in L}|I(i)|.
\end{eqnarray*}
See Corollary \ref{cor:structures for oribits determined by normal rulings} for an interpretation of $A(\rho)$.
\end{definition}

With the definition above, we can now prove the proposition.

\begin{proof}[Proof of Proposition \ref{prop:r-d is independent of normal rulings}]
Assume $T$ lives over the interval $[x_0,x_1]$. Let $\rho$ be any $m$-graded normal ruling of $T$ such that $\rho|_{T_L}=\rho_L,\rho|_{T_R}=\rho_R$. For each $x$ in $[x_0,x_1]$ avoiding the crossings and cusps of $T$, define $A(x):=A(\rho|_{\{x\}})$. In particular, $A(x_0)=A(\rho_L)$ and $A(x_1)=A(\rho_R)$.

Observe that, as $x$ increases, $A(x)$ increases (resp. decreases) by $1$ when passing an $m$-graded return (resp. $m$-graded departure) of $\rho$ and is unchanged when passing a crossing of all other types. When passing a right cusp $q$, let $x_c,x'_c$ be the $x$-coordinate immediately before and after $q$. Suppose $\rho_c:=\rho|_{x=x_c}$ and $\rho_c':=\rho|_{x=x_c'}$ determine the partitions $I_c=I(T|_{x=x_c})=U_c\sqcup L_c$ and $I_c'=I(T|_{x=x_c'})=U_c'\sqcup L_c'$ respectively. Suppose $q$ connects strands $k,k+1$ of $T|_{x=x_c}$, then $k\in U_c, k+1\in L_c$, $\rho_c(k)=k+1$ and $A_{\rho_c}(k)=\emptyset$. Denote $I_c^a:=\{i\in I(T|_{x=x_c})|\mu(i)=a (\mr{mod} m)\}$, $U_c^a:=\{i\in U_c|\mu(i)=a (\mr{mod} m)\}$ and $L_c^a:=\{i\in L_c|\mu(i)=a (\mr{mod} m)\}$ for all congruence classes $a (\mr{mod} m)$. Say, $\mu(k)=a (\mr{mod} m)$. It follows that
\begin{eqnarray*}
&&A(x_c)-A(x_c')-|I_c(k+1)|\\
&=&\sum_{i\in U_c,k\in A_{\rho_c}(i)}1+
\sum_{i\in L_c,k\in I_c(i)}1+\sum_{i\in L_c, k+1\in I_c(i)}1\\
&=&\sum_{i\in U_c^a, i<k,\rho_c(i)>k+1}1+\sum_{i\in L_c^a,i<k}1+\sum_{j=\rho_c^{-1}(i)\in U_c^a,j<\rho_c(j)=i<k+1}1\\
&=&\sum_{i\in U_c^a, i<k}1+\sum_{i\in L_c^a,i<k}1\\
&=&|\{i\in I_c^a|i<k\}|.
\end{eqnarray*}
Hence,
\begin{eqnarray*}
A(x_c)-A(x_c')&=&|\{i\in I_c|i>k+1,\mu(i)=\mu(k+1) (\mr{mod} m)\}|\\
&+&|\{i\in I_c|i<k,\mu(i)=\mu(k) (\mr{mod} m)\}|
\end{eqnarray*}
is independent of $\rho$. Similarly, when passing a left cusp, $A(x)$ only changes by a constant, which only depends on $(T,\mu)$ near the cusp, not on $\rho$.

As a consequence, by moving $x$ from $x_0$ to $x_1$, we obtain that $A(\rho_R)-A(\rho_L)=A(x_1)-A(x_0)=r'(\rho)-d(\rho)+C$ for some constant $C$, which depends only on $(T,\mu)$, not on $\rho$.
It follows that $r'(\rho)-d(\rho)$ is independent of $\rho$.
\end{proof}

\begin{remark}
In Section \ref{subsec:alternative generalization}, we will introduce the concepts of \emph{$m$-graded generalized normal rulings}. It turns out, by the same proof, the previous proposition still holds, when we replace ``$m$-graded normal ruling" by ``$m$-graded generalized normal ruling" everywhere. It follows that Lemma \ref{lem:a constant function of rulings} holds for $m$-graded generalized normal rulings as well, if we use Definition \ref{def:generalized Ruling polynomials_2} to define $\chi(\rho)$.
\end{remark}

\subsection{Example}
\begin{example}\label{ex:Ruling polynoimals vs aug numbers for bordered trefoil knot}
Consider the Legendrian tangle $(T,\mu)$ in Example \ref{ex:bordered trefoil knot} (See Figure \ref{fig:filling_surface} (left)), with no base point. Hence, $B=0$. Let's check Theorem \ref{thm:counting for tangles} with our example by a direct calculation.

Let $b_{ij}, 1\leq i<j\leq 4$ be the pairs of right end points of $T$, so $\mc{A}(T_R)$ is generated by $b_{ij}$'s with the grading: $|b_{12}|=|b_{13}|=|b_{24}|=|b_{34}|=0, |b_{23}|=-1, |b_{14}|=1$. By Example \ref{ex:bordered trefoil knot}, $T_R$ has 2 $m$-graded normal rulings $(\rho_R)_1,(\rho_R)_2$.
Let's firstly determine the orbits $\mc{O}_m((\rho_R)_1;k),\mc{O}_m((\rho_R)_2;k)$. Use the identification in Example \ref{ex:aug. variety for lines}, given any $m$-graded augmentation $\epsilon_R$ for $T_R$, denote by $d_R$ the corresponding differential for $C(T_R)$. Let $I:=\{1,2,3,4\}$ be the set of right endpoints of $T$.

By the proof of Lemma \ref{lem:Barannikov normal form}, $d_R\in\mc{O}_m((\rho_R)_1;k)$ if and only if the partition and bijection determined by $d_R$ is $I=U\sqcup L$ and $(\rho_R)_1:U\xrightarrow[]{\sim}L$, where $U=\{1,3\}, L=\{2,4\}$ and $(\rho_R)_1(1)=2, (\rho_R)_1(3)=4$. That is, the condition says $1,3$ are $d_R$-upper, and
$(\rho_R)_1(i)=\rho_{d_R}(i):=\mr{max}\{\rho_{d_R}(x)|x \text{ is $i$-admissible}\}$
for $i=1,3$, equivalently, $<d_Re_1,e_2>\neq 0$ and $<d_Re_3,e_4>\neq 0$. Hence, we have
\begin{eqnarray*}
\mc{O}_m((\rho_R)_1;k)=\{\epsilon_R\in\mr{Aug}_m(T_R;k)|\epsilon(b_{12})\neq 0, \epsilon(b_{34})\neq 0\}.
\end{eqnarray*}

Similarly, $d_R\in\mc{O}_m((\rho_R)_2;k)$ if and only if $1,2$ are $d_R$-upper and
$(\rho_R)_2(i)=\rho_{d_R}(i)=\mr{max}\{\rho_{d_R}(x)|x \text{ is $i$-admissible}\}$
for $i=1,2$. For $i=1$, the previous condition says $<d_Re_1,e_2>=0$ (otherwise, $\rho_{d_R}(1)=2$, contradiction), $<d_Re_2,e_3>=0$ (Otherwise, $2$ is $d_R$-upper and $\rho_{d_R}(2)=3$, contradiction), and $<d_Re_1,e_3>\neq 0$; For $i=2$, the previous condition says $<d_Re_2,e_4>\neq 0$ and $<d_Re_3,e_4>=0$. As a consequence, we have
\begin{eqnarray*}
\mc{O}_m((\rho_R)_2;k)=\{\epsilon_R\in\mr{Aug}_m(T_R;k)|\epsilon_R(b_{12}),\epsilon_R(b_{23}),\epsilon_R(b_{34})=0, \epsilon_R(b_{13}),\epsilon_R(b_{24})\neq 0\}.
\end{eqnarray*}

Now, let $\epsilon$ be any $m$-graded $k$-augmentation of $T$, denote by $\epsilon_L,\epsilon_R$ the restriction of $\epsilon$ to $T_L, T_R$ respectively. Let $x_i=\epsilon(a_i)$, and $x_{ij}=\epsilon(a_{ij})$ for $i<j$. Notice that $\epsilon(a_{23})=0=\epsilon(a_{14})$. By Example \ref{ex:DGA for bordered trefoil knot}, the full augmentation variety for $T$ is:
\begin{eqnarray*}
\mr{Aug}_m(T;k)=\{(x_i, x_{ij})|x_{23},x_{14}=0,\sum_{i<k<j}(-1)^{|a_{ik}|+1}x_{ik}x_{kj}=0\}
\end{eqnarray*}
for $m\neq 1$ and
\begin{eqnarray*}
\mr{Aug}_m(T;k)=\{(x_i, x_{ij})|x_{23}=0,\sum_{i<k<j}(-1)^{|a_{ik}|+1}x_{ik}x_{kj}=0\}
\end{eqnarray*}
for $m=1$. Moreover, the co-restriction $\iota_R:\mc{A}(T_R)\rightarrow \mc{A}(T)$ is given by
\begin{eqnarray*}
&&\iota_R(b_{12})=a_{13}(1+a_2a_3)+a_{12}(a_1+a_3+a_1a_2a_3);\\
&&\iota_R(b_{13})=a_{12}(1+a_1a_2)+a_{13}a_2;\\
&&\iota_R(b_{14})=a_{14};\iota_R(b_{23})=0;\\
&&\iota_R(b_{24})=(1+a_2a_1)a_{34}-a_2a_{24};\\
&&\iota_R(b_{34})=(1+a_3a_2)a_{24}-(a_1+a_3+a_3a_2a_1)a_{34}.
\end{eqnarray*}

It follows that
\begin{eqnarray*}
&&\epsilon_R(b_{12})=x_{13}(1+x_2x_3)+x_{12}(x_1+x_3+x_1x_2x_3);\\
&&\epsilon_R(b_{13})=x_{12}(1+x_1x_2)+x_{13}x_2;\\
&&\epsilon_R(b_{14})=x_{14};\epsilon_R(b_{23})=0;\\
&&\epsilon_R(b_{24})=(1+x_1x_2)x_{34}-x_2x_{24};\\
&&\epsilon_R(b_{34})=(1+x_2x_3)x_{24}-(x_1+x_3+x_1x_2x_3)x_{34}.
\end{eqnarray*}

With the preparation above, we have the following augmentation variety and augmentation number (with fixed boundary conditions) associated to $(T,\mu)$, corresponding to each case in Example \ref{ex:bordered trefoil knot}:

\noindent{}\textbf{(1).}
Notice that for $\epsilon_L=\epsilon_{(\rho_L)_1}$, have $x_{12}=1=x_{34}$ and $x_{ij}=0$ otherwise. Hence, for the boundary conditions $((\rho_L)_1,(\rho_R)_1)$ (see Definition \ref{def:aug varieties with boundary conditions}), have
\begin{eqnarray*}
&&\mr{Aug}_m(T,\epsilon_{(\rho_L)_1)},(\rho_R)_1;k)
=\{\epsilon\in\mr{Aug}_m(T;k)|\epsilon_L=\epsilon_{(\rho_L)_1},\epsilon_R\in\mc{O}_m((\rho_R)_1;k)\}\\
&=&\{(x_i)_{1\leq i\leq 3}\in k^3|x_1+x_3+x_1x_2x_3\neq 0\}\\
&=&k^*\times k\sqcup k^*\times k\sqcup (k^*)^3
\end{eqnarray*}
where in the decomposition of the last equality, the subvarieties are $\{x_1=0,x_3\neq 0\}$, $\{x_3= 0,x_1\neq 0\}$ and $\{x_1\neq 0,x_3\neq 0,x_1+x_3+x_1x_2x_3\neq 0\}$ respectively. Hence, by Definition \ref{def:augmentation number} and Example \ref{ex:bordered trefoil knot}, the augmentation number is
\begin{eqnarray*}
&&\mr{aug}_m(T,(\rho_L)_1,(\rho_R)_1;q)=q^{-3}(2(q-1)q+(q-1)^3)\\
&=&q^{-\frac{3}{2}}(2z+z^3)=q^{-\frac{d}{2}}<(\rho_L)_1|R_T^m(z)|(\rho_R)_1>
\end{eqnarray*}
where $z=q^{\frac{1}{2}}-q^{-\frac{1}{2}}$ and $3=d=\mr{max.deg}_z<(\rho_L)_1|R_T^m(z)|(\rho_R)_1>$.

\noindent{}\textbf{(2).} For the boundary conditions $((\rho_L)_1,(\rho_R)_2)$, have
\begin{eqnarray*}
&&\mr{Aug}_m(T,\epsilon_{(\rho_L)_1)},(\rho_R)_2;k)
=\{\epsilon\in\mr{Aug}_m(T;k)|\epsilon_L=\epsilon_{(\rho_L)_1},\epsilon_R\in\mc{O}_m((\rho_R)_2;k)\}\\
&=&\{\{(x_i)_{1\leq i\leq 3}\in k^3|x_1+x_3+x_1x_2x_3=0,1+x_1x_2\neq 0\}\\
&=&\{(x_1,x_2)\in k^2|1+x_1x_2\neq 0\}=k\sqcup (k^*)^2
\end{eqnarray*}
where in the decomposition of the last equality, the subvarieties are $\{x_1=0,x_2\in k\}$ and $\{x_1\neq 0, 1+x_1x_2\neq 0\}$ respectively. Hence, by Definition \ref{def:augmentation number} and Example \ref{ex:bordered trefoil knot}, the augmentation number is:
\begin{eqnarray*}
&&\mr{aug}_m(T,(\rho_L)_1,(\rho_R)_2;q)=q^{-2}(q+(q-1)^2)\\
&=&q^{-1}(1+z^2)=q^{-\frac{d}{2}}<(\rho_L)_1|R_T^m(z)|(\rho_R)_2>
\end{eqnarray*}
where $z=q^{\frac{1}{2}}-q^{-\frac{1}{2}}$ and $2=d=\mr{max.deg}_z<(\rho_L)_1|R_T^m(z)|(\rho_R)_2>$.

\noindent{}\textbf{(3).} Notice that for $\epsilon_L=\epsilon_{(\rho_L)_2}$, have $x_{13}=1=x_{24}$ and $x_{ij}=0$ otherwise. Hence, for the boundary conditions $((\rho_L)_2,(\rho_R)_1)$, have
\begin{eqnarray*}
&&\mr{Aug}_m(T,\epsilon_{(\rho_L)_2)},(\rho_R)_1;k)
=\{\epsilon\in\mr{Aug}_m(T;k)|\epsilon_L=\epsilon_{(\rho_L)_2},\epsilon_R\in\mc{O}_m((\rho_R)_1;k)\}\\
&=&\{\{(x_i)_{1\leq i\leq 3}\in k^3|1+x_2x_3\neq 0\}\\
&=&k^2\sqcup k\times(k^*)^2
\end{eqnarray*}
where in the decomposition of the last equality, the subvarieties are $\{x_2=0,(x_1,x_3)\in k^2\}$ and $\{x_2\neq 0, 1+x_2x_3\neq 0,x_1\in k\}$ respectively. Hence, by Definition \ref{def:augmentation number} and Example \ref{ex:bordered trefoil knot}, the augmentation number is:
\begin{eqnarray*}
&&\mr{aug}_m(T,(\rho_L)_1,(\rho_R)_2;q)=q^{-3}(q^2+q(q-1)^2)\\
&=&q^{-1}(1+z^2)=q^{-\frac{d}{2}}<(\rho_L)_2|R_T^m(z)|(\rho_R)_1>
\end{eqnarray*}
where $z=q^{\frac{1}{2}}-q^{-\frac{1}{2}}$ and $2=d=\mr{max.deg}_z<(\rho_L)_2|R_T^m(z)|(\rho_R)_1>$.

\noindent{}\textbf{(4).} For the boundary conditions $((\rho_L)_2,(\rho_R)_2)$, have
\begin{eqnarray*}
&&\mr{Aug}_m(T,\epsilon_{(\rho_L)_2)},(\rho_R)_2;k)
=\{\epsilon\in\mr{Aug}_m(T;k)|\epsilon_L=\epsilon_{(\rho_L)_2},\epsilon_R\in\mc{O}_m((\rho_R)_2;k)\}\\
&=&\{\{(x_i)_{1\leq i\leq 3}\in k^3|1+x_2x_3=0,x_2\neq 0\}\\
&=&k\times k^*
\end{eqnarray*}
Hence, by Definition \ref{def:augmentation number} and Example \ref{ex:bordered trefoil knot}, the augmentation number is:
\begin{eqnarray*}
&&\mr{aug}_m(T,(\rho_L)_2,(\rho_R)_2;q)=q^{-2}q(q-1)\\
&=&q^{-\frac{1}{2}}z=q^{-\frac{d}{2}}<(\rho_L)_2|R_T^m(z)|(\rho_R)_2>
\end{eqnarray*}
where $z=q^{\frac{1}{2}}-q^{-\frac{1}{2}}$ and $1=d=\mr{max.deg}_z<(\rho_L)_2|R_T^m(z)|(\rho_R)_2>$.

Altogether, the calculation matches with Theorem \ref{thm:counting for tangles} in each case.
\end{example}

\section{Augmentations for elementary Legendrian tangles}\label{sec:Augmentations for elementary Legendrian tangles}

The main goal of this section is to show Lemma \ref{lem:augmentation varieties for elementary tangles}. More generally, we also obtain a finer structure of the augmentation varieties $\mr{Aug}_m(T,\epsilon_L,\rho_R;k)$ (see Theorem \ref{thm:augmentation varieties for Legendrian tangles}).

\subsection{The identification between augmentations and A-form MCSs}\label{subsec:aug vs A-form MCS}

Let $(E,\mu)$ be any elementary Legendrian tangle: a single crossing $q$, a left cusp $q$, a (marked or unmarked) right cusp $q$, or $n$ parallel strands with a single base point $q$. Assume $E$ has $n_L$ left endpoints and $n_R$ right endpoints, and denote $\mu_L:=\mu|_{E_L}, \mu_R:=\mu|_{E_R}$.

Let $\epsilon$ be any $m$-graded $k$-augmentation of $\mc{A}(E)$. Denote $\epsilon_L:=\epsilon|_{E_L},\epsilon_R=\epsilon|_{E_R}$, where $\epsilon_R$ is induced from $\epsilon$ via $\iota_R:\mc{A}(E_R)\rightarrow \mc{A}(E)$. By Example \ref{ex:aug. variety for lines}, we can identify $\epsilon_L$ and $\epsilon_R$ with some $\mb{Z}/m$-graded filtered complexes $(C(E_L),d_L)$ and $(C(E_R),d_R)$ respectively. We know $d_R$ is completely determined by $d_L$ and the information of $\epsilon$ near $q$. To make this precise, we firstly introduce the following

\begin{definition}
A \emph{handleslide} is a \emph{vertical line segment} $H_r$ lying on two strands of $(T,\mu)$, equipped with a \emph{coefficient} $r\in k$, where $(T,\mu)$ is some trivial Legendrian tangle of $n$ parallel strands. For simplicity, we denote such a handleslide by $H_r$. A handleslide $H_r$ is \emph{$m$-graded} if either $r=0$ or its end-points belong to 2 strands having the same Maslov potential value modulo $m$.
\end{definition}
A $\mb{Z}/m$-graded handleslide $H_r$ with coefficient $r$ between strands $j<k$, is also equivalent to an $\mb{Z}/m$-graded filtered \emph{elementary transformation} $H_r:C((H_r)_L)\xrightarrow[]{\sim}C((H_r)_R)$ (closely related to Morse complex sequences (MCSs) in \cite{HR15}):
\begin{eqnarray*}
H_r(e_i)=\left\{\begin{array}{ll}
e_i & i\neq j\\
e_j-re_k & i=j
\end{array}\right.
\end{eqnarray*}

Now, by a direct calculation we have
\begin{lemma}\label{lem:aug and A-form MCS for elementary tangles}
Given $(E,\mu)$ and $\epsilon$ as above.

\begin{enumerate}
\item
If $E$ is a single crossing $q$ between strands $k,k+1$. Then there's an isomorphism of $Z/m$-graded (not necessarily filtered) complexes $\varphi: (C(E_L),d_L)\xrightarrow[]{\sim}(C(E_R),d_R)$ given by $\varphi=s_k\circ H_{r}$ for $r=-\epsilon(q)$, where $s_k:C(E_L)\xrightarrow[]{\sim}C(E_R)$ is the $\mb{Z}/m$-graded \emph{elementary transformation}
\begin{eqnarray*}
s_k(e_i)=\left\{\begin{array}{ll}
e_i & i\neq k, k+1\\
e_{k+1} & i=k\\
e_k & i=k+1
\end{array}\right.
\end{eqnarray*}
and $H_{r}: C(E_L)\xrightarrow[]{\sim}C(E_L)$ is the handleslide between strands $k,k+1$ of $E_L$. Note: $<d_Le_k,e_{k+1}>=0=<d_Re_k,e_{k+1}>$.\\
Pictorially, we can \emph{represent} $s_k$ by the front diagram $E$ with a crossing between strands $k,k+1$, hence $\varphi$ is represented by the front diagram $E$ with a handleslide of coefficient $r$ between strands $k,k+1$ to the left of $q$.

\item
If $E$ is a left cusp $q$ connecting strands $k,k+1$ of $E_R$. Then as a $\mb{Z}/m$-graded complex, $(C(E_R),d_R)$ is a direct sum of $(C(E_L), d_L)$ and the acyclic complex $(\mr{Span}\{e_k,e_{k+1}\}$, $d_Re_k=(-1)^{\mu_R(k)}e_{k+1})$, via the morphism $\varphi:(C(E_L),d_L)\hookrightarrow (C(E_R),d_R)$:
\begin{eqnarray*}
\varphi(e_i)=\left\{\begin{array}{ll}
e_i & i < k\\
e_{i+2} & i \geq k
\end{array}\right.
\end{eqnarray*}
\noindent{}Pictorially, we can simply \emph{represent} $\varphi$ by the front diagram $E$.

\item
If $E$ is a right cusp $q$ connecting strands $k,k+1$ of $E_L$. Let $t$ be the generator corresponding to the base point in $\mc{A}(E)$ if $q$ is marked, and $1$ otherwise. Then there's an morphism of complexes $\varphi: (C(E_L),d_L)\rightarrow (C(E_R),d_R)$ given by $\varphi=\varphi_c\circ Q\circ H_r$, where $H_r:(C(E_L),d_L)\xrightarrow[]{\sim}(C(E_L),d_L')$ is the handleslide with coefficient $r=-\epsilon(q)$ between strands $k,k+1$ of $E_L$, $Q:(C(E_L),d_L')\rightarrow (C(E_L),d_L')/\mr{Span}\{e_k,d_L'e_k\}$ is the natural quotient map, and
$\varphi_c: (C(E_L),d_L')/\mr{Span}\{e_k,d_L'e_k\}\xrightarrow[]{\sim}(C(E_R),d_R)$ is the isomorphism defined by
\begin{eqnarray*}
\varphi_c([e_i])=\left\{\begin{array}{ll}
e_i & i < k\\
e_{i-2} & i > k+1
\end{array}\right.
\end{eqnarray*}
Note: $<d_Le_k,e_{k+1}>=<d_L'e_k,e_{k+1}>=(-1)^{\mu_L(k)}(-\epsilon(t)^{\sigma(q)})$ (see Definition \ref{def:sign for a right cusp} for $\sigma(q)$), this ensures that the quotient $(C(E_L),d_L')/\mr{Span}\{e_k,d_L'e_k\}$ is freely generated by $[e_i], i\neq k, k+1$ as a $k$-module.\\
Pictorially, we can \emph{represent} $\varphi_c\circ Q$ by the front $E$ (with coefficient $\epsilon(t)^{\sigma(q)}$ attached to the base point if $q$ is marked), then $\varphi$ is represented by the front $E$ with a handleslide between strands $k,k+1$ of $E_L$ to the left of $q$.

\item
If $E$ is a single base point $q$ on the strand $k$. Let $\lambda:=\epsilon(q)$ (resp. $\epsilon(q)^{-1}$) if the orientation of the strand $k$ is right moving (resp. left moving). Then there's an isomorphism of complexes $\varphi:(C(E_L),d_L)\rightarrow (C(E_R),d_R)$ via
\begin{eqnarray*}
\varphi(e_i)=\left\{\begin{array}{ll}
e_i & i\neq k\\
\lambda e_k & i=k
\end{array}\right.
\end{eqnarray*}
\noindent{}Pictorially, we can simply \emph{represent} $\varphi$ by the front $E$ with the coefficient $\lambda$ attached to the base point.
\end{enumerate}
\end{lemma}

\begin{corollary}\label{cor:nondegeneracy of augmentations}
There's an isomorphism $H_*(C(E_L),d_L)\xrightarrow[]{\sim}H_*(C(E_R),d_R)$ of $\mb{Z}/m$-graded $k$-modules .
In particular, if $\epsilon_L$ is acyclic, then so is $\epsilon_R$.
By induction, this result then generalizes to all Legendrian tangles.
\end{corollary}
\begin{proof}
By the previous lemma, the only nontrivial case is when $E$ is a single right cusp, when we obtain a short exact sequence of $\mb{Z}/m$-graded complexes
$0\rightarrow \mr{Span}\{e_k,d_Le_k\}\rightarrow (C(E_L),d_L)\rightarrow (C(E_R),d_R)\rightarrow 0$ with the first term acyclic. Pass to the long exact sequence of homologies, we then obtain the desired isomorphism from $H_*((C(E_L),d_L))$ to $H_*(C(E_R),d_R)$.
\end{proof}

\begin{definition}\label{def:A-form MCS for elementary Legendrian tangles}
Given any elementary Legendrian tangle $(E,\mu)$, a \emph{$m$-graded $A$-form MCS} for $E$ is a triple $((C(E_L),d_L), \varphi,(C(E_R),d_R))$, where $(C(E_L),d_L), (C(E_R),d_R)$ are $\mb{Z}/m$-graded filtered complexes, $\varphi:(C(E_L),d_L)\rightarrow (C(E_R),d_R)$ is a $\mb{Z}/m$-graded morphism of complexes (or equivalently, the diagram $\varphi$), such that they satisfy the conditions in each case of Lemma \ref{lem:aug and A-form MCS for elementary tangles}. In particular, $(C(E_R),d_R)$ is determined by $(C(E_L),d_L)$ and $\varphi$.

\end{definition}

\begin{remark}\label{rem:aug and A-form MCS}
With the definition above, Lemma \ref{lem:aug and A-form MCS for elementary tangles} then shows that, there's an \emph{identification} between the augmentation variety $\mr{Aug}_m(E;k)$ and the set of $m$-graded A-form MCSs $MCS_m^A(E;k)$ for $E$, for any elementary Legendrian tangle $(E,\mu)$.

For any Legendrian tangle $(T,\mu)$, by cutting $T$ into elementary Legendrian tangles, one can define a $m$-graded A-form MCS for $T$ as a ``composition" of $m$-graded A-form MCSs for the elementary parts of $T$. We can similarly define the set $MCS_m^A(T;k)$ of all $m$-graded A-form MCSs for $T$. The lemma then shows by induction that, there's an identification $\mr{Aug}_m(T;k)\cong MCS_m^A(T;k)$.
\end{remark}

\subsection{Handleslide moves}

There're some identities involving the elementary transformations (represented by handleslides $H_r$ or crossings $s_k$ as in Lemma \ref{lem:aug and A-form MCS for elementary tangles}) between $\mb{Z}/m$-graded complexes. They can be represented by the \emph{local moves} (or \emph{handleslide moves}) of diagrams as in Figure \ref{fig:Handleslide_moves}: Each diagram represents a composition of elementary transformations with the maps going from left to right, and each local move represents an identity between 2 different compositions.

\begin{figure}[!htbp]
\begin{center}
\minipage{\textwidth}
\includegraphics[width=\linewidth, height=1.2in]{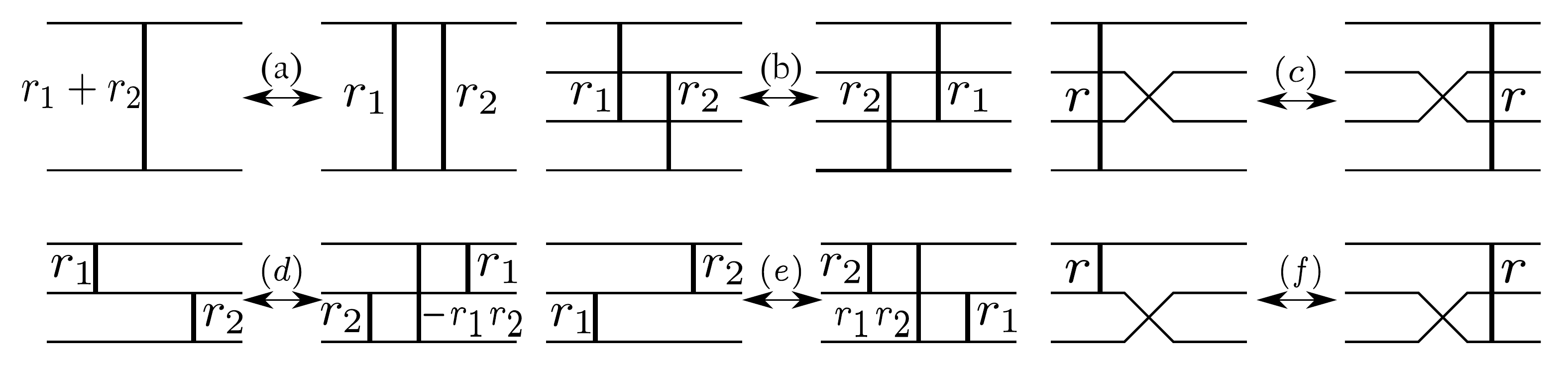}
\endminipage\hfill
\end{center}
\caption{Local moves of handleslides in a Legendrian tangle $T$ $=$ identities between different compositions of elementary transformations. The moves shown do not illustrate all the possibilities.}
\label{fig:Handleslide_moves}
\end{figure}

More precisely, the possible local moves in a Legendrian tangle $(T,\mu)$ are as follows (see also \cite[Section.6]{HR15}):

\begin{enumerate}
\item[\textbf{Type 0:}] 
(Introduce or remove a trivial handleslide) Introduce or remove a handleslide with coefficient $0$ and endpoints on two strands with the same Maslov potential value modulo $m$.
\item[\textbf{Type 1:}]
(Slide a handleslide past a crossing) Suppose $T$ contains one single crossing between strands $k$ and $k+1$, and exactly one handleslide $h$ between strands $i<j$, with $(i,j)\neq (k,k+1)$. We may slide $h$ (either left or right) past the crossing such that the endpoints of $h$ remain on the same strands of $T$. See Figure \ref{fig:Handleslide_moves} (c),(f) for two such examples.

\item[\textbf{Type 2:}]
(Interchange the positions of two handleslides) If $T$ contains exactly two handleslides $h_1$, $h_2$ between strands $i_1<j_1$, and $i_2<j_2$, with coefficients $r_1, r_2$ respectively. If $j_1\neq i_2$ and $i_1\neq j_2$, we may interchange the positions of the handleslides, see Figure \ref{fig:Handleslide_moves} (b) for an example; If $j_1=i_2$ (resp. $i_1=j_2$) and $h_1$ is to the left of $h_2$, we may interchange the positions of $h_1,h_2$, and introduce a new handleslide between strands $i_1, j_2$ (resp. $i_2,j_1$), with coefficient $-r_1r_2$ (resp. $r_1r_2$), see Figure (d) (resp. (e)).

\item[\textbf{Type 3:}]
(Merge two handleslides) Suppose $T$ contains exactly two handleslides $h_1, h_2$ between the same two strands, with coefficients $r_1, r_2$, respectively. We may merge the two handleslides into one between the same two strands, with coefficient $r_1+r_2$, see Figure \ref{fig:Handleslide_moves} (a).

\item[\textbf{Type 4:}]
(Introduce two canceling handleslides) Suppose $T$ contains no crossings, cusps or handleslides. We may introduce two new handleslides between the same two strands, with coefficients $r, -r$, where $r\in k$.
\end{enumerate}

Suppose $T$ contains no crossings or cusps, recall that as usual the strands of $T$ are labeled from top to bottom as $1,2,\ldots,s$. Given a handleslide $h$ in $T$, denote by $t_h<b_h$ the top and bottom strands of $h$.

\begin{definition}
Given 2 handleslides $h,h'$ in $T$, we say $h<h'$ if either $t_h>t_{h'}$ or $t_h=t_{h'}$ and $b_h<b_{h'}$.

\noindent{}A collection of handleslides $V$ in $T$ is called \emph{properly ordered} if given any 2 handleslides $h, h'$ in $V$, with $h$ to the left of $h'$, then $h<h'$.

\noindent{} Given a collection of handleslides $V$ in $T$, define $V^t$ to be the collection obtained from reversing the ordering of the $x$-coordinates of the handleslides in $V$.
\end{definition}

Assume $V$ is a collection of handleslides in $T$ such that either $V$ or $V^t$ is \emph{properly ordered}. There're 2 additional types of moves involving $V$, as a composition of Type 0, 2 and 3 moves:

\begin{enumerate}
\item[\textbf{Type 5:}]
(Incorporate a handleslide $h$ into a collection $V$) Suppose $h$ is a handleslide in $T$ immediately to the right of $V$, with coefficient $r$. We move $h$ into $V$ via Type 0, 2 and 3 moves to create a new collection $\overline{V}$, so that $\overline{V}$ has the \emph{same ordering property} as $V$:\\
If necessary, use a \emph{Type 0} move to introduce a trivial handleslide in $V$ with endpoints on the same strands as $h$, such that $V$ has the same ordering property as before. In this way, $V$ contains a unique handleslide $h'$ with endpoints on the same strands as $h$ and say, with coefficient $r'$ ($r'$ may be 0); Label the handleslides between $h$ and $h'$ from right to left by $h_1,h_2,\ldots,h_n$. For each $1\leq j\leq n$, move $h$ past $h_j$ via a \emph{Type 2} move, which possibly creates a new handleslide $h_j'$ (See Figure \ref{fig:Handleslide_moves} (d) and (e)); Merge $h_j'$ with the existing handleslides in $V$ with the same endpoints as $h_j'$ via \emph{Type 2} moves and one \emph{Type 3} move. The ordering property of $V$ ensures this does not introduce any new handleslides; After the above moves, $h$ and $h'$ are next to each other, use a \emph{Type 3} move to merge $h$ and $h'$. The resulting handleslide has coefficient $r+r'$.

\noindent{} When $h$ is immediately to the left of $V$, a similar procedure can be used to incorporate the handleslide $h$ into $V$ such that the resulting collection $\overline{V}$ has the same ordering property as $V$.

\item[\textbf{Type 6:}]
(Remove a handleslide $h$ from a collection $V$) Suppose $h$ is a handleslide with coefficient $r$ in $V$. Use \emph{Type 2} moves, we can remove $h$ from $V$ with coefficient unchanged, so that it appears either to the left or right of the remaining handleslides (with possibly new handleslides, see Figure \ref{fig:Handleslide_moves} (d), (e)), denoted by $\overline{V}$; Use \emph{Type 2} and \emph{Type 3} moves to reorder and merge handleslides in $\overline{V}$ so that $\overline{V}$ has the same ordering property as $V$. The ordering property of $V$ ensures this can be done without introducing any new handleslides.
\end{enumerate}

\subsection{Augmentation varieties for elementary Legendrian tangles}\label{subsec:aug vars for elementary Leg tangles}

Now, we're able to prove Lemma \ref{lem:augmentation varieties for elementary tangles}.

\begin{proof}[Proof of Lemma \ref{lem:augmentation varieties for elementary tangles}]
The result is trivial if $E$ is a left cusp, a marked right cusp or $2n$ parallel strands with a single base point. Also , if $E$ is a single crossing $q$ and $|q|\neq 0 (\mr{mod} m)$, then by definition, $\chi(\rho)=r(\rho)=B=0$ and $\mr{Aug}_m^{\rho}(E,\epsilon_L;k)=\{(\epsilon_L,\epsilon(q))|\epsilon(q)=0\}$ is a single point, the result follows. Now, assume $E$ is a single crossing $q$ between strands $k,k+1$ of $E_L$ and $|q|=0 (\mr{mod} m)$. Since $B=0$ and $-\chi(\rho)=s(\rho)-c_R=s(\rho)$, we need to show: $\mr{Aug}_m^{\rho}(E,\epsilon_L;k)\cong (k^*)^{s(\rho)}\times k^{r(\rho)}$.

Let $(C(E_L),d_L)$ be the complex corresponding to $\epsilon_L$, under the identification in Remark \ref{rem:aug and A-form MCS}, we have
\begin{equation}\label{eqn:aug vs A-form MCS}
\mr{Aug}_m^{\rho}(E,\epsilon_L;k)=\{r\in k|(C(E_R),d_R):=s_k\circ H_r((C(E_L),d_L))\in \mc{O}_m(\rho_R;k)\}
\end{equation}
Here $r=-\epsilon(q)$ corresponds to $\epsilon\in\mr{Aug}_m^{\rho}(E,\epsilon_L;k)$ and $H_r$ is the handleslide (which represents an elementary transformation) with coefficient $r$ between strands $k,k+1$ of $E_L$. Use the identification (\ref{eqn:aug vs A-form MCS}) above, given any $r$ in $\mr{Aug}_m^{\rho}(E,\epsilon_L;k)$, denote $(C(E_R),d_R):=s_k\circ H_r((C(E_L),d_L))$. For simplicity, we simply write $d_R=(s_k\circ H_r) \cdot d_L$.

Firstly, we show the lemma in the case \emph{when $\epsilon_L$ is a standard augmentation}, or equivalently the complex $(C(E_L),d_L)$ is standard (See Remark \ref{rem:standard Morse complex via augmentations} for the definition).

Notice that $\rho_L(k)\neq k+1$. Let $A=\{k,k+1,\rho_L(k),\rho_L(k+1)\}$, $\alpha=\min A, \beta= \min (A\setminus \{\alpha,\rho_L(\alpha)\})$. Let $a=<d_Le_{\alpha},e_{\rho_L(\alpha)}>, b=<d_Le_{\beta},e_{\rho_L(\beta)}>$. Notice that both $a$ and $b$ are nonzero, as $(C(E_L),d_L)$ is standard with respect to $\rho_L$ (see Remark \ref{rem:standard Morse complex via augmentations}).
Given any $r$ in $\mr{Aug}_m^{\rho}(E,\epsilon_L;k)$, we divide the discussion into several cases:

\begin{enumerate}[label=(\arabic*)]
\item
$\rho_L(k)<k<k+1<\rho_L(k+1)$. If $r=0$, then $d_R$ is standard and $q$ is a $m$-graded departure of $\rho$.\\
If $r\neq 0$, then $H\cdot d_R$ is standard for some composition of $m$-graded handleslides $H$ (See Figure \ref{fig:SR-form MCS} (S1)). Notice that any $m$-graded handleslide represents a $m$-graded filtration preserving automorphism of a $m$-graded filtered complex, which doesn't change the isomorphism class, hence the $m$-graded normal ruling determined by the complex. It follows that $H\cdot d_R$ and $d_R$ determines the same $m$-graded normal ruling of $E_R$. Hence, $q$ is a (S1) switch of $\rho$ (See Figure \ref{fig:SR-form MCS} (S1)).

\item
\begin{enumerate}[label=(\alph*)]
\item
$\rho_L(k+1)<\rho_L(k)<k<k+1$ or
\item
$k<k+1<\rho_L(k+1)<\rho_L(k)$.
\end{enumerate}

\noindent{}If $r=0$, then $d_R$ is standard (with respect to $\rho_R$), and $q$ is a \emph{$m$-graded departure} of $\rho$.\\
If $r\neq 0$, then $H\cdot d_R$ is standard for $H=$ some composition of handleslides (See Figure \ref{fig:SR-form MCS} (S2) (resp. (S3))). And it follows that $q$ is a \emph{(S2) (resp. (S3)) switch} in the case (2a) (resp. (2b)).

\item
$\rho_0(k+1)<k<k+1<\rho_0(k)$. Then $d_R$ is standard and $q$ is a $m$-graded \emph{(R1) return} (See Figure \ref{fig:SR-form MCS} (R1)).

\item
\begin{enumerate}[label=(\alph*)]
\item
$\rho_0(k)<k<k+1<\rho_0(k+1)$ or
\item
$k<k+1<\rho_0(k)<\rho_0(k+1)$.
\end{enumerate}

\noindent{}Then $d_R$ is standard and $q$ is a $m$-graded \emph{(R2) (resp. (R3)) return} in the case (4a) (resp. (4b)) (See Figure \ref{fig:SR-form MCS} (R2) (resp. (R3))).
\end{enumerate}

\begin{figure}[!htbp]
\begin{center}
\minipage{\textwidth}
\includegraphics[width=\linewidth, height=2in]{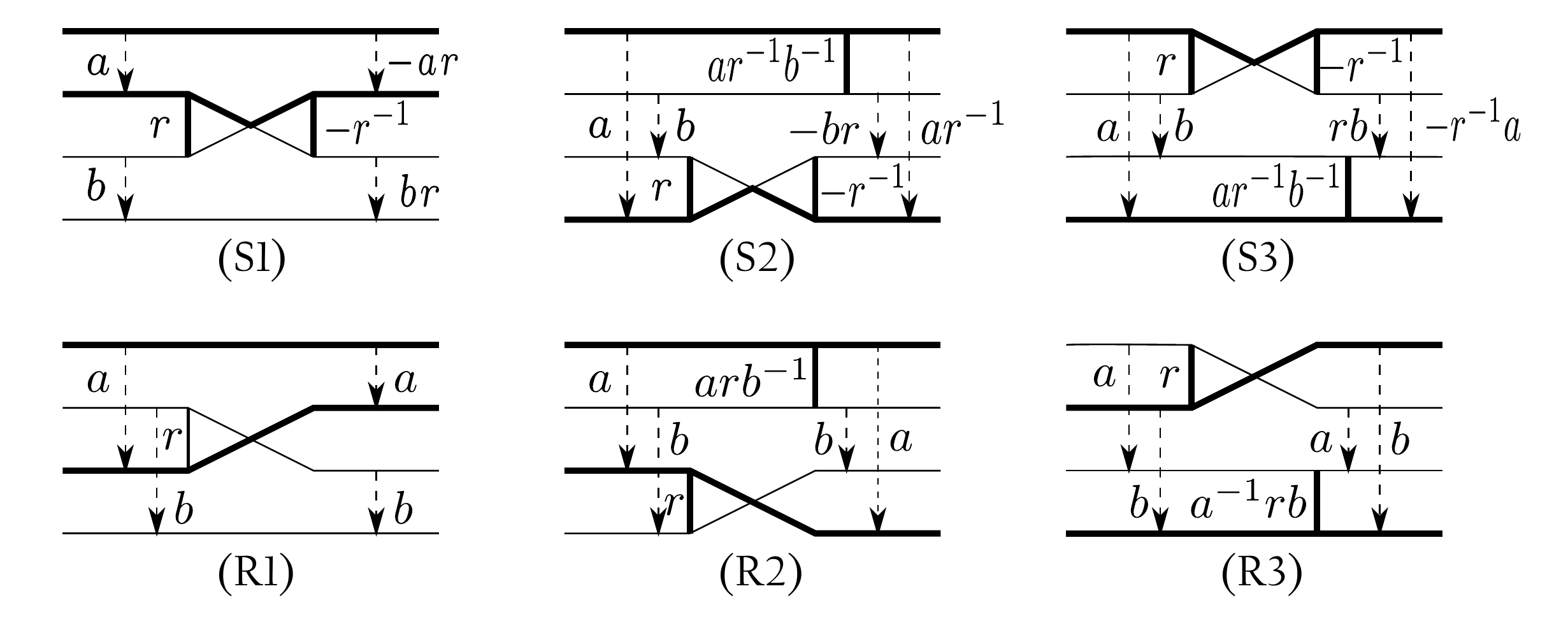}
\endminipage\hfill
\end{center}
\caption{The handleslide moves which preserve standard complexes: the crossing is a switch on the top row and
a $m$-graded return on the bottom row respectively, of the normal ruling determined by the complexes on the 2 ends. The dashed arrows correspond to the nonzero coefficients $<de_i,e_j>$ of the complexes associated to the 2 ends.}
\label{fig:SR-form MCS}
\end{figure}

As a consequence, via the identification (\ref{eqn:aug vs A-form MCS}) we obtain that
\begin{eqnarray*}
\mr{Aug}_m^{\rho}(E,\epsilon_L;k)=
\left\{\begin{array}{lll}
\{r|r=0\} & \text{If $q$ is a $m$-graded departure of $\rho$;}\\
\{r|r\neq 0\} & \text{If $q$ is a switch of $\rho$;}\\
\{r|r\in k\} & \text{If $q$ is a $m$-graded return of $\rho$.}
\end{array}\right.
\end{eqnarray*}
If follows that $\mr{Aug}_m^{\rho}(E,\epsilon_L;k)\cong (k^*)^{s(\rho)}\times k^{r(\rho)}$, as desired.

In the general case, by Remark \ref{rem:standard Morse complex via augmentations}, there exists an unipotent automorphism $\varphi_0$ of $C(E_L)$ such that $d_0:=\varphi_0^{-1}\cdot d_L$ is standard. We can represent $\varphi_0$ by a collection of handleslides $V$ which is properly ordered. Hence, $d_R=(s_k\circ H_r\circ V)\cdot d_0$. Pictorially, the morphism $s_k\circ H_r\circ V$ is represented by the Legendrian tangle front $E$, with a handleslide $H_r$ to the left of $q$ and a collection of handleslides $V$ to the left of $H_r$. Let $h'$ be the handleslide between strands $k,k+1$ in $V$, with coefficient $r'$, where $r'=r'(\varphi_0)$ is a constant depending on $\varphi_0$.

Use a \emph{type 5} move to incorporate $H_r$ into $V$, to obtain a collection of handleslides $V_1$. The handleslide $h_1$ between strands $k,k+1$ in $V_1$ has coefficient $r+r'$. Use a \emph{Type 6} move to remove $h_1$ from $V_1$ so that it appears to the left of the remaining handleslides, denoted by $V_2$. Now, $V_2$ contains no handleslides between strands $k,k+1$. Hence, we can use \emph{Type 2} moves to slide $V_2$ past the crossing $q$ to obtain a collection of handleslides $V_3$ to the right of $q$. The meaning of the procedure is that $s_k\circ H_r\circ V=V_3\circ s_k\circ h_1$ as a $\mb{Z}/m$-graded isomorphism from $C(E_L)$ to $C(E_R)$. Hence, $d_R=(s_k\circ H_r\circ V)\cdot d_0=(V_3\circ s_k\circ h_1)\cdot d_0$, which is isomorphic to $(s_k\circ h_1)\cdot d_0$, where $h_1=H_{r+r'}$ is the handleslide between strands $k,k+1$ of $E_L$. It follows that
\begin{eqnarray*}
\mr{Aug}_m^{\rho}(E,\epsilon_L;k)=\{r\in k|(s_k\circ H_{r+r'})\cdot d_0\in\mc{O}_m(\rho_R;k)\}.
\end{eqnarray*}
Now, $d_0$ is standard, we have reduced the problem to the previous case. More precisely, we have
\begin{eqnarray}\label{eqn:aug var for elementary Legendrian tangles}
\mr{Aug}_m^{\rho}(E,\epsilon_L;k)=
\left\{\begin{array}{lll}
\{r|r+r'=0\} & \text{If $q$ is a $m$-graded departure of $\rho$;}\\
\{r|r+r'\neq 0\} & \text{If $q$ is a switch of $\rho$;}\\
\{r|r+r'\in k\} & \text{If $q$ is a $m$-graded return of $\rho$.}
\end{array}\right.
\end{eqnarray}
and the desired result follows.

\end{proof}

In Remark \ref{rem:standard Morse complex via augmentations}, it turns out that there's canonical choice of the unipotent automorphism $\varphi_0$:

\begin{lemma}\label{lem:canonical automorphism via non-degenerate augmentation}
Let $(T,\mu)$ be a trivial Legendrian tangle of $n$ parallel strands, and $\rho$ be any $m$-graded isomorphism type (Definition \ref{def:isomorphism type for trivial Legendrian tangles}) of $(T,\mu)$. Equivalently, given the isomorphism class $\mc{O}_m(\rho;k)=\mr{Aut}(C)\cdot d_{\rho}$ of $\mb{Z}/m$-graded filtered complexes $(C=C(T),d)$ over $k$, determined by the Barannikov normal form $d_{\rho}$. There's a canonical algebraic map $\varphi_0:\mc{O}_m(\rho;k)\rightarrow \mr{Aut}(C)$ such that, $\varphi_0(d)$ is unipotent and $\varphi_0(d)^{-1}\cdot d$ is standard for all $d\in\mc{O}_m(\rho;k)$. Equivalently, the principal bundle $\mr{Aut}(C)\rightarrow \mc{O}_m(\rho;k)=\mr{Aut(C)}\cdot d_{\rho}$ has a canonical section $\varphi$.
\end{lemma}

\begin{proof}
As in the proof of Lemma \ref{lem:Barannikov normal form}, let $I=I(T)=\{1,2,\ldots,n\}$ be the set of left endpoints of $T$. Then $\rho$ determines a partition $I=U\sqcup L\sqcup H$ and a bijection $\rho:U\xrightarrow[]{\sim}L$, where $U$,$L$ and $H$ are the sets of \emph{upper},\emph{lower} and \emph{homological} indices in $I$ determined by the isomorphism class $\mc{O}_m(\rho;k)$.

As in Definition \ref{def:subsets of I determined by normal rulings}, define the subsets $I(i),i\in I$ and $A(i),i\in U\sqcup H$ of $I$. Clearly, for each $j\in A(i)$ have $A(j)\subsetneq A(i)$.

\noindent{}\textbf{Claim:} Given any $d$ in $\mc{O}_m(\rho;k)$ and any upper or homological index $i$ in $U\sqcup H$, there exists a unique $i$-admissible element in $C$ of the form $e_i'=e_i+\sum_{j\in A(i)}a_je_j$, such that $de_i'$ is $\rho(i)$-admissible. When $\rho(i)=\infty$, $de_i'$ is $\infty$-admissible means $de_i'=0$. Moreover, $e_i'=e_i'(d)$ depends on $d$ algebraically.

\begin{proof}[Proof of Claim:]
We firstly show the \emph{existence}. By the proof of Lemma \ref{lem:Barannikov normal form}, we know: $\rho(i)=\mr{max}\{\rho_d(x)|x \text{ is $i$-admissible}.\}$ for any $i\in U\sqcup H$. Hence, we can take an $i$-admissible element of the form $x=e_i+\sum_{j\in I(i)}a_je_j$, such that $\rho_d(x)=\rho(i)$, i.e. $dx$ is $\rho(i)$-admissible.

If $a_j=0$ for all $j\in I(i)\setminus A(i)$, then $e_i'=x$ is the desired element. Otherwise, for $x_0=x$ above, \emph{define} $j=j(x_0):=\mr{min}\{j|j\in I(i)\setminus A(i) \text{ and $a_j=<x_0,e_j>\neq 0$.}\}$. Then by definition of $A(i)$, either $j\in L\sqcup H$ or $j\in U$ and $\rho(j)>\rho(i)$. If $j\in L\sqcup H$, then there exists an $j$-admissible element of the form $y_j=e_j+\sum_{l>j}*_le_l$ such that $dy_j=0$. It follows that $x_1=x_0-a_jy_j$ is $i$-admissible and $dx_1$ is still $\rho(i)$-admissible, but $j(x_1)>j=j(x_0)$. Here, we define $\mr{min} \emptyset:=\infty$. If $j\in U$ and $\rho(j)>\rho(i)$, then there exists a $j$-admissible element of the form $y_j=e_j+\sum_{l>j}*_le_l$ such that $dy_j$ is $\rho(j)$-admissible. It follows again that $x_1=x_0-a_jy_j$ is $i$-admissible and $dx_1$ is still $\rho(i)$-admissible, but $j(x_1)>j(x_0)$.

If $j(x_1)=\infty$, then $e_i'=x_1$ is the desired element in the Claim. Otherwise, replace $x_0$ by $x_1$ and repeat the procedure above. Inductively, for some sufficiently large $N$, we obtain in the end an $i$-admissible element of the form $x_N=e_i+\sum_{j\in I(i)}a_je_j$ such that $dx_N$ is $\rho(i)$-admissible and $j(x_N)=\infty$. Now, $e_i'=x_N$ fulfils the Claim.

\noindent{}\emph{uniqueness}. We show the uniqueness by induction on $|A(i)|$. If $A(i)=\emptyset$, then $e_i'=e_i$, which is clearly unique. For the inductive procedure, assume the uniqueness holds when $|A(i)|<k$, and consider the case when $|A(i)|=k$. Let $e_i'=e_i+\sum_{j\in A(i)}a_je_j$ be any element satisfying the Claim. Since $A(j)\subsetneq A(i)$ for all $j\in A(i)$, by induction we can rewrite $e_i'=e_i+\sum_{j\in A(i)}b_je_j'$, where $e_j'$ is uniquely determined by $d$ for all $j\in A(i)$. We want to show the uniqueness of $b_j$'s.

Assume $A(i)=\{i_1<i_2<\ldots<i_k\}$ and $\rho(A(i))=\{j_1<j_2<\ldots<j_k\}\subset L$. By definition of $A(i)$, we know $\rho(i_l)<\rho(i)$ for all $1\leq l\leq k$. By the conditions of the Claim, $de_i'$ is $\rho(i)$-admissible, hence $<de_i',e_{\rho(i_l)}>=0$ for all $1\leq l\leq k$. That is, the following system of linear equations for $b_j$'s holds:
\begin{eqnarray*}
(<e_{\rho(i_p)},de_{i_q}'>)_{p,q}(b_q)_{q}=(-<e_{\rho(i_p)},de_i>)_p
\end{eqnarray*}
Notice that the coefficient matrix $(<e_{\rho(i_p)},de_{i_q}'>)$ is similar to $(<e_{j_p},de_{\rho^{-1}(j_q)}'>)$. And by definition, $de_{\rho^{-1}(j_q)}'$ is $j_q$-admissible, hence $<e_{j_p},de_{\rho^{-1}(j_q)}'>=0$ if $p<q$ and $<e_{j_q},de_{\rho^{-1}(j_q)}'>\neq 0$. That is, the square matrix $(<e_{j_p},de_{\rho^{-1}(j_q)}'>)$ is lower triangular and invertible, hence $(<e_{\rho(i_p)},de_{i_q}'>)$ is also invertible. It follows that
\begin{eqnarray*}
(b_q)_{q}=(<e_{\rho(i_p)},de_{i_q}'>)_{p,q}^{-1}(-<e_{\rho(i_p)},de_i>)_p
\end{eqnarray*}
where by induction $e_{i_q}'$'s are uniquely determined by $d$, hence so is the right hand side. The uniqueness in the Claim then follows. The previous equation also shows by induction that $e_i'=e_i'(d)$ depends algebraically on $d$.
\end{proof}

On the other hand, for each $j\in L$, so $j=\rho(i)$ for some $i\in U$. By the Claim, $de_i'=c_je_j'$ for some $c_j\neq 0$ and some $j$-admissible element of the form $e_j'=e_j+\sum_{l>j}*_le_l$. The claim shows that both $c_j$ and $e_j'$ are also uniquely determined and depend on $d$ algebraically. Now, define an unipotent isomorphism $\varphi_0(d)$ of $C$ by $\varphi_0(d)(e_i)=e_i'$ for all $1\leq i\leq n$. It follows that $\varphi_0:\mc{O}_m(\rho;k)\rightarrow \mr{Aut}(C)$ defines a canonical algebraic map. Moreover, for any $d\in\mc{O}_m(\rho;k)$ and any $i\in U$, have
\begin{eqnarray*}
(\varphi_0(d)^{-1}\cdot d)e_i&=&\varphi_0(d)^{-1}\circ d\circ \varphi_0(d)(e_i)
=\varphi_0(d)^{-1}\circ d(e_i')\\
&=&c_{\rho(i)}\varphi_0(d)^{-1}(e_{\rho(i)}')
=c_{\rho(i)}e_{\rho(i)}.
\end{eqnarray*}
Similarly, $(\varphi_0(d)^{-1}\cdot d)e_i=0$ for $i\in H\sqcup L$. So $\varphi_0(d)^{-1}\cdot d$ is standard.

Finally, for the canonical section of $\mr{Aut}(C)\rightarrow \mc{O}_m(\rho;k)$, we simply take $\varphi(d):=D(d)\circ \varphi_0(d)$, where $D(d)(e_i')=e_i'$ for $i\in U\sqcup H$ and $D(d)(e_j')=c_je_j'$ for $j\in L$.
\end{proof}

\begin{corollary}\label{cor:structures for oribits determined by normal rulings}
Let $(T,\mu)$ be a trivial Legendrian tangle of $n$ parallel strands, and $\rho$ be any $m$-graded isomorphism type of $(T,\mu)$. Then
\begin{eqnarray*}
\mc{O}_m(\rho;k)\cong (k^*)^{|L|}\times k^{A(\rho)}
\end{eqnarray*}
where $A(\rho)$ is defined as in Definition \ref{def:subsets of I determined by normal rulings}.
\end{corollary}

\begin{proof}
By the previous lemma, there's an identification between $d\in\mc{O}_m(\rho;k)$ and $\varphi(d)$, where $\varphi$ is the canonical section of $\mr{Aut}(C(T))\rightarrow \mc{O}_m(\rho;k)=\mr{Aut}(C(T))\cdot d_{\rho}$. But, use the notations in the proof of Lemma \ref{lem:canonical automorphism via non-degenerate augmentation} (see also Definition \ref{def:subsets of I determined by normal rulings}), the general form of $\varphi(d)$ is $\varphi(d)=D(d)\circ \varphi_0(d)$, where $\varphi_0(d)(e_i)=e_i'=e_i+\sum_{j\in A(i)}*_{ij}e_j$ for $i\in U\sqcup H$ and $\varphi_0(e_i)=e_i+\sum_{j\in I(i)}*_{ij}e_j$ for $i\in L$. Moreover, $D(d)(e_i')=e_i'$ for $i\in U\sqcup H$, and $D(d)(e_i')=c_ie_i'$ for $i\in L$ and some $c_i\in k^*$. It follows that
\begin{eqnarray*}
\mc{O}_m(\rho;k)
&\cong&\{(*_{ij},i\in U\sqcup H, j\in A(i),\text{ or $i\in L, j\in I(i)$},c_i,i\in L)|c_i\neq 0\}\\
&\cong& (k^*)^{|L|}\times k^{A(\rho)}
\end{eqnarray*}
by Definition \ref{def:subsets of I determined by normal rulings}.
\end{proof}

The previous lemma allows us to show a stronger result than Lemma \ref{lem:augmentation varieties for elementary tangles}:

\begin{lemma}\label{lem:fibration for augmentation varieties of elementary tangles}
Let $(E,\mu)$ be an elementary Legendrian tangle: a single crossing $q$, a left cusp $q$, a marked right cusp $q$ or $2n$ parallel strands with a single base point $*$. Let $\rho$ be a $m$-graded normal ruling of $E$, denote $\rho_L:=\rho|_{E_L},\rho_R:=\rho|_{E_R}$\footnote{Notice that $\rho$ is uniquely determined by $\rho_L,\rho_R$.}. Then the natural map $P: \mr{Aug}_m(E,\rho_L,\rho_R;k)\rightarrow\mc{O}_m(\rho_L;k)$ given by $\epsilon\rightarrow \epsilon_L=\epsilon|_{E_L}$ is algebraically a \emph{trivial fiber bundle} with fibers isomorphic to $(k^*)^{-\chi(\rho)+B}\times k^{r(\rho)}$, where $B$ is the number of base points, $\chi(\rho)=c_R-s(\rho)$ (see Definition \ref{def:switches and returns}).
\end{lemma}

\begin{proof}
As in the proof of Lemma \ref{lem:augmentation varieties for elementary tangles}, the only nontrivial case is when $E$ contains a single crossing $q$ and $|q|=0 (\mr{mod} m)$. In the proof of Lemma \ref{lem:augmentation varieties for elementary tangles} for the general case (i.e. for a general augmentation $\epsilon_L\in\mc{O}_m(\rho_L;k)$ or the corresponding differential $d_L$ for $C(E_L)$), the unipotent automorphism $\varphi_0$ can be taken to be canonical: $\varphi_0:=\varphi_0(\epsilon_L)$, by lemma \ref{lem:canonical automorphism via non-degenerate augmentation} above. Hence, the constant $r'=r'(\epsilon_l)$ in Equation (\ref{eqn:aug var for elementary Legendrian tangles}) depends algebraically on $\epsilon_L$. Use the identification in Remark \ref{rem:aug and A-form MCS}, it follows that
\begin{eqnarray*}
\mr{Aug}_m(E,\rho_L,\rho_R)=
\left\{\begin{array}{lll}
\{(\epsilon_L,r)|\epsilon_L\in\mc{O}_m(\rho_L), r+r'(\epsilon_L)=0\} & \text{$q$ is a departure;}\\
\{(\epsilon_L,r)|\epsilon_L\in\mc{O}_m(\rho_L), r+r'(\epsilon_L)\neq 0\} & \text{$q$ is a switch;}\\
\{(\epsilon_L,r)|\epsilon_L\in\mc{O}_m(\rho_L)\} & \text{$q$ is a  return.}
\end{array}\right.
\end{eqnarray*}
with the natural map $P$ given by $(\epsilon_L,r)\rightarrow \epsilon_L$. Here we have ignored the coefficient field $k$. Therefore, we obtain an isomorphism
\begin{eqnarray*}
P\times R:\mr{Aug}_m(E,\rho_L,\rho_R;k)\xrightarrow[]{\sim}
\mc{O}_m(\rho_L;k)\times ((k^*)^{-\chi(\rho)+B}\times k^{r(\rho)})
\end{eqnarray*}
where $R(\epsilon_L,r)=r+r'(\epsilon_L)$. The result then follows.
\end{proof}

As a consequence, we obtain
\begin{theorem}\label{thm:augmentation varieties for Legendrian tangles}
Let $(T,\mu)$ be any Legendrian tangle, with $B$ base points placed on $T$ so that \emph{each right cusp is marked}.
Fix $m$-graded normal rulings $\rho_L,\rho_R$ of $T_L, T_R$ respectively. Fix $\epsilon_L\in\mc{O}_m(\rho_L;k)$. Then
there's a decomposition of augmentation varieties into disjoint union of subvarieties
\begin{eqnarray*}
\mr{Aug}_m(T,\epsilon_L,\rho_R;k)=\sqcup_{\rho}\mr{Aug}_m^{\rho}(T,\epsilon_L,\rho_R;k)
\end{eqnarray*}
(See Definition \ref{def:augmentaion varieties via normal rulings}),
where $\rho$ runs over all $m$-graded normal rulings of $T$ such that $\rho|_{T_L}=\rho_L,\rho|_{T_R}=\rho_R$. Moreover,
\begin{eqnarray}
\mr{Aug}_m^{\rho}(T,\epsilon_L,\rho_R;k)\cong (k^*)^{-\chi(\rho)+B}\times k^{r(\rho)}.
\end{eqnarray}
\end{theorem}

\begin{proof}
Use the notations in Definition \ref{def:augmentaion varieties via normal rulings}, $T=E_1\circ\ldots\circ E_n$ is a composition of $n$ elementary tangles, and the canonical projection $P_n:\mr{Aug}_m^{\rho}(T_n,\epsilon_L;k)\rightarrow \mr{Aug}_m^{\rho|_{T_{n-1}}}(T_{n-1},\epsilon_L;k)$ is
a base change of the projection $\mr{Aug}_m(E_n,\rho_{n-1},\rho_n;k)\rightarrow\mc{O}_m(\rho_{n-1};k)$. The latter, by the previous lemma, is a trivial fiber bundle with fibers isomorphic to $(k^*)^{-\chi(\rho|_{E_n})+B(E_N)}\times k^{r(\rho|_{E_n})}$. Hence, so is the projection $P_n$ and
\begin{eqnarray*}
\mr{Aug}_m^{\rho}(T_n,\epsilon_L;k)\cong \mr{Aug}_m^{\rho|_{T_{n-1}}}(T_{n-1},\epsilon_L;k)\times (k^*)^{-\chi(\rho|_{E_n})+B(E_n)}\times k^{r(\rho|_{E_n})}
\end{eqnarray*}
By induction, the desired result then follows from Lemma \ref{lem:augmentation varieties for elementary tangles}.
\end{proof}

\begin{remark}\label{rem:structure thm for aug var}
We've defined varieties $\mr{Aug}_m^{\rho}(T,\rho_L,\rho_R;k)=\mr{Aug}_m^{\rho}(T,\rho_L;k)$ in Definition \ref{def:augmentaion varieties via normal rulings}. Use Lemma \ref{lem:fibration for augmentation varieties of elementary tangles}, a similar argument as in the proof above also shows that
\begin{eqnarray}
\mr{Aug}_m(T,\rho_L,\rho_R;k)=\sqcup_{\rho}\mr{Aug}_m^{\rho}(T,\rho_L,\rho_R;k)
\end{eqnarray}
with
\begin{eqnarray}
\mr{Aug}_m^{\rho}(T,\rho_L,\rho_R;k)
&\cong& \mc{O}_m(\rho_L;k)\times (k^*)^{-\chi(\rho)+B}\times k^{r(\rho)}\\
&\cong& (k^*)^{-\chi(\rho)+B+n_L'}\times k^{r(\rho)+A(\rho_L)}\nonumber
\end{eqnarray}
where $n_L=2n_L'$ is the number of left endpoints of $T$.
\end{remark}

\section{Ruling polynomials compute augmentation numbers: a generalization}\label{sec:Ruling vs aug:general case}

So far we have established the ``Ruling polynomials compute augmentation numbers" result for acyclic augmentations (Remark \ref{rem:nondegeneracy of augmentations}). For the purpose of studying Legendrian knots in $J^1\mb{R}$, that is all we need.
However, by the decomposition in Equation (\ref{eqn:decomp for the full aug var}), we see that in general the full augmentation variety $\mr{Aug}_m(T;k)$ contains also non-acyclic augmentations. It's natural to expect these additional augmentations are counted appropriately by some generalized Ruling polynomials for $T$. This is indeed the case, and the arguments we have for acyclic augmentations suggest what might be the right way to generalize the picture.

\subsection{Generalized normal rulings and Ruling polynomials}\label{subsec:generalized normal rulings}

Let $(T,\mu)$ be any Legendrian tangle in $J^1U$ with $n_L$ left endpoints and $n_R$ right endpoints. Notice that $n_L+2c_L=n_R+2c_R$, so $n_L, n_R$ have the same parity. Take any $m$-graded isomorphism type $\rho_L$ for $T_L$, as in Definition \ref{def:isomorphism type for trivial Legendrian tangles}, it determines a partition $I(T_L)=U\sqcup L\sqcup H$. $H$ determines an isomorphism class $H_*:=[H]$ of $\mb{Z}/m$-graded $k$-modules, represented by the module $\mr{Span}_k\{e_i, i\in H||e_i|=\mu(i) (\mr{mod} m)\}$. We will \emph{say} $\rho_L$ is of \emph{homology type $H_*$}.

Fix a homology type $H_*$, say $k:=\mr{dim}_k H_*$, and $k_i:=\mr{dim}_k H_i=(H_*)_i$ is the dimension of the degree $i$ part of $H_*$, for $i=0,\ldots,m-1 (\mr{mod} m)$. $H_*$ can also be identified with the collection $(k_0,\ldots,k_{m-1})$ of natural numbers. \emph{Define} a Legendrian tangle $\hat{T}$ by adding $k$ parallel strands $T(H_*)=\sqcup_iT(H_i)=\{1,2,\ldots,k\}$, above the front $T$ in the front plane $U\times \mb{R}_z$, labeled from top to bottom. We may assume that $T(H_{i})$ lies above $T(H_{i+1})$ for $i=0,\ldots,m-2 (\mr{mod} m)$.
Extend $\mu$ to a $\mb{Z}/2r$-valued Maslov potential $\mu$ on $\hat{T}$ so that $\mu$ takes the same value and $\mu=i+1 (\mr{mod} m)$ on the strands $T(H_i)$. We will fix such a $\mu$. Now, the numbers of left and right endpoints of $\hat{T}$ are even.

\begin{definition}\label{def:generalized normal rulings}
A \emph{$m$-graded normal ruling of homology type $H_*$} for $(T,\mu)$ is a $m$-graded normal ruling $\hat{\rho}$ of $(\hat{T},\mu)$ such that when restricted to $T_L$, $\hat{\rho}$ descends to a $m$-graded isomorphism type of homology type $H_*$ for $(T_L,\mu)$. In other words, $\hat{\rho}_L$ determines a partition $I(\hat{T}_L)=\hat{U}\sqcup \hat{L}$ and a bijection $\hat{\rho}_L:\hat{U}\xrightarrow[]{\sim}\hat{L}$, hence descends to a $m$-graded isomorphism type $\rho_L:=\hat{\rho}|_{T_L}$ given by $I(T_L)=U\sqcup L\sqcup H$ with $U:=\hat{U}\cap I(T_L)$, $L:=\hat{\rho}_L(U)$, $H=\hat{\rho}_L(T(H_*))$, and $\rho_L:=\hat{\rho}_L|_{U}\xrightarrow[]{\sim}L$.
\end{definition}

Similarly, for any generic $x$-coordinate in $U$, denote by $\hat{\rho}|_{T_a}$ the $m$-graded isomorphism type of $T_a=T|_{\{x=a\}}$ obtained by restricting $\hat{\rho}$ to $T_a$. Then by definition, it's clear that the homology type of $\hat{\rho}|_{T_a}$ is \emph{independent of $a$}.

\begin{definition}\label{def:generalized Ruling polynomials}
Given any Legendrian tangle $(T,\mu)$ and homology type $H_*$, so $(\hat{T},\mu)$ is defined. Given any $m$-graded normal rulings $\hat{\rho}_L,\hat{\rho}_R$ of homology type $H_*$ for $T_L,T_R$ respectively, the \emph{$m$-graded Ruling polynomial (with boundary conditions $(\hat{\rho}_L,\hat{\rho}_R$))} for $(T,\mu)$ is
\begin{eqnarray*}
<\hat{\rho}_L|R_T^m(z)|\hat{\rho}_R>:=<\hat{\rho}_L|R_{\hat{T}}^m(z)|\hat{\rho}_R>
\end{eqnarray*}
where the right hand side is the usual Ruling polynomial for $\hat{T}$ as in Definition \ref{def:Ruling polynomial}.
When $\hat{\rho}_L, \hat{\rho}_R$ are of different homology types, then we simply define $<\hat{\rho}_L|R_T^m(z)|\hat{\rho}_R>:=0$. This is compatible with the fact that $<\hat{\rho}_L|R_{\hat{T}}^m(z)|\hat{\rho}_R>=0$ if $\hat{\rho}_L|_{T_L}, \hat{\rho}_R|_{T_R}$ are of different homology types.
\end{definition}

By definition, it follows immediately from Theorem \ref{thm:invariance and composition of Ruling polynomials} that
\begin{corollary}\label{cor:invariance and composition of generalized ruling polynomials}
The Ruling polynomials $<\hat{\rho}_L|R_T^m(z)|\hat{\rho}_R>$ are Legendrian isotopy invariants for $(T,\mu)$.
Moreover, if $T=T_1\circ T_2$ is a composition of 2 Legendrian tangles, then the \emph{composition axiom} for $<\hat{\rho}_L|R_T^m(z)|\hat{\rho}_R>$ holds:
\begin{eqnarray*}
<\hat{\rho}_L|R_T^m(z)|\hat{\rho}_R>=
\sum_{\hat{\rho}_I}<\hat{\rho}_L|R_T^m(z)|\hat{\rho}_I>
<\hat{\rho}_I|R_T^m(z)|\hat{\rho}_R>
\end{eqnarray*}
where $\hat{\rho}_I$ runs over all $m$-graded normal rulings of $(T_1)_R=(T_2)_L$ of any homology types.
\end{corollary}

\subsection{Augmentation numbers for augmentations of given homology type}

Let $(T,\mu)$ be any Legendrian tangle, with base points $*_1,\ldots,*_B$ so that each connected component containing a right cusp has at least one base point. As usual, fix an nonnegative integer $m$ dividing $2r$. Recall that in Definition \ref{def:aug varieties with boundary conditions}, we have introduced augmentation varieties $\mr{Aug}_m(T,\epsilon_L,\rho_R;k)$ (resp. $\mr{Aug}_m(T,\rho_L,\rho_R;k)$) with boundary conditions, where $\rho_L,\rho_R$ are $m$-graded isomorphism types of $T_L,T_R$ respectively, and $\epsilon_L\in\mc{O}_m(\rho_L;k)$.

Use the identification in Lemma \ref{lem:augmentation varieties for elementary tangles}, if follows from Corollary \ref{cor:nondegeneracy of augmentations} that, if $\rho_L,\rho_R$ are of different homology type, then $\mr{Aug}_m(T,\rho_L,\rho_R;k)=\emptyset$, hence so is $\mr{Aug}_m(T,\epsilon_L,\rho_R;k)$. From now on, we assume $\rho_L,\rho_R$ are of the same fixed homology type $H_*$.
Define $(\hat{T},\mu)$ as in Section \ref{subsec:generalized normal rulings} above and \emph{fix} any $m$-graded normal ruling $\hat{\rho}_L$ of homology type $H_*$ for $(T,\mu)$ so that $\hat{\rho}|_{T_L}=\rho_L$. Take any $m$-graded augmentation $\hat{\epsilon_L}\in\mc{O}_m(\hat{\rho}_L;k)$ that restricts to $\epsilon_L$, where the restriction map is induced by the natural inclusion of DGAs $\mc{A}(T_L)\hookrightarrow\mc{A}(\hat{T}_L)$.

In fact, there's a canonical choice for $\hat{\epsilon_L}$: Denote by $d_L$ the differential for $C(T_L)$ corresponding to $\epsilon_L$, under the identification in Example \ref{ex:aug. variety for lines}.
By Lemma \ref{lem:canonical automorphism via non-degenerate augmentation}, there's a canonical automorphism $\varphi=\varphi(d_L)$ of $C(T_L)$ so that $\varphi(d_L)^{-1}\cdot d_L=d_{\rho_L}$ is the Barannikov normal form. Denote the isomorphism type determined by $\rho_L$ by $I(T_L)=U_L\sqcup L_L\sqcup H_L$ and $\rho_L:U_L\xrightarrow[]{\sim}L_L$ (see Definition \ref{def:isomorphism type for trivial Legendrian tangles}). Assume $\varphi(e_i)=e_i'$ for $i\in I(T_L)$, then $d_Le_i'=e_{\rho_L(i)}'$ for all $i\in U_L$ and $d_Le_i'=0$ for $i\in H_L$. We then take the differential $\hat{d_L}$ for $C(\hat{T}_L)=C(T_L)\oplus C(T(H_*))$ given by $\hat{d_L}|_{C(T_L)}=d_L$, and $\hat{d_L}(e_i)=e_{\hat{\rho}_L(i)}'$ for $i\in T(H_*)$. And define $\hat{\epsilon_L}$ to be the augmentation corresponding to $\hat{d_L}$.

So, from now on, we will always assume the augmentation $\hat{\epsilon_L}$ is the canonical one that restricts to $\epsilon_L$. The map $\epsilon_L\rightarrow\hat{\epsilon_L}$ thus defines a canonical embedding $\mc{O}_m(\rho_L;k)\hookrightarrow\mc{O}_m(\hat{\rho_L};k)$. Denote by $\hat{\mc{O}}_m(\rho_L;k)$ the image of $\mc{O}_m(\rho_L;k)$. By definition of $\hat{T}$, we have a push-out diagram of $\mb{Z}/2r$-graded DGAs
\begin{equation*}
\xymatrix{
\mc{A}(T_L) \ar[r] \ar[d] & \mc{A}(\hat{T}_L) \ar[d] \\
\mc{A}(T) \ar[r] & \mc{A}(\hat{T})
}
\end{equation*}
Pass to augmentation varieties, we then obtain an identification between the augmentation varieties $\mr{Aug}_m(T,\epsilon_L,\rho_R;k)$ and
\begin{eqnarray*}
\mr{Aug}_m(T,\hat{\epsilon_L},\rho_R;k):=
\{\hat{\epsilon}\in\mr{Aug}_m(\hat{T};k)|\hat{\epsilon}|_{\hat{T}_L}=\hat{\epsilon_L}, \hat{\epsilon}|_{T_R}\in\mc{O}_m(\rho_R;k)\}
\end{eqnarray*}
Similarly, we obtain an identification between $\mr{Aug}_m(T,\rho_L,\rho_R;k)$ and
\begin{eqnarray*}
\mr{Aug}_m(T,\hat{\rho_L},\rho_R;k):=
\{\hat{\epsilon}\in\mr{Aug}_m(\hat{T};k)|\hat{\epsilon}|_{\hat{T}_L}\in\hat{\mc{O}}_m(\rho_L;k), \hat{\epsilon}|_{T_R}\in\mc{O}_m(\rho_R;k)\}
\end{eqnarray*}

Under this identification, we obtain a decomposition of the augmentation varieties into subvarieties
\begin{eqnarray}\label{eqn:decomp for non-acyclic aug}
&&\mr{Aug}_m(T,\epsilon_L,\rho_R;k)=\sqcup_{\hat{\rho_R}}\mr{Aug}_m(T,\hat{\epsilon_L},\hat{\rho_R};k)\\
&&\mr{Aug}_m(T,\rho_L,\rho_R;k)=\sqcup_{\hat{\rho_R}}\mr{Aug}_m(T,\hat{\rho_L},\hat{\rho_R};k)\nonumber
\end{eqnarray}
where $\hat{\rho_R}$ runs over all $m$-graded normal rulings of homology type $H_*$ of $T$ that descends to $\rho_R$, and
\begin{eqnarray*}
&&\mr{Aug}_m(T,\hat{\epsilon_L},\hat{\rho_R};k):=
\{\hat{\epsilon}\in\mr{Aug}_m(\hat{T};k)|\hat{\epsilon}|_{\hat{T}_L}=\hat{\epsilon_L}, \hat{\epsilon}|_{\hat{T}_R}\in\mc{O}_m(\hat{\rho_R};k)\}\\
&&\mr{Aug}_m(T,\hat{\rho_L},\hat{\rho_R};k):=
\{\hat{\epsilon}\in\mr{Aug}_m(\hat{T};k)|\hat{\epsilon}|_{\hat{T}_L}\in\hat{\mc{O}}_m(\rho_L;k), \hat{\epsilon}|_{\hat{T}_R}\in\mc{O}_m(\hat{\rho_R};k)\}
\end{eqnarray*}
Note: $\mr{Aug}_m(T,\hat{\epsilon_L},\hat{\rho_R};k)=\mr{Aug}_m(\hat{T},\hat{\epsilon_L},\hat{\rho_R};k)$, where the latter is as in Definition \ref{def:aug varieties with boundary conditions}.

Inspired by the decomposition (\ref{eqn:decomp for non-acyclic aug}), we introduce
\begin{definition}\label{def:augmentation number:general case}
Given $(T,\mu)$ as usual, let $\hat{\rho_L},\hat{\rho_R}$ be any $m$-graded normal rulings of $T_L,T_R$ respectively of the same homology type $H_*$, so $\hat{T}$ is defined.
Then, the \emph{$m$-graded augmentation number with boundary conditions $(\hat{\rho_L},\hat{\rho_R})$} for $(T,\mu)$ is
\begin{eqnarray*}
&&\mr{aug}_m(T,\hat{\rho_L},\hat{\rho_R};q):=\mr{aug}_m(\hat{T},\hat{\rho_L},\hat{\rho_R};q)
=q^{-\mr{dim}\mr{Aug}_m(T,\hat{\epsilon_L},\hat{\rho_R};k)}|\mr{Aug}_m(T,\hat{\epsilon_L},\hat{\rho_R};\mb{F}_q)|
\end{eqnarray*}
\end{definition}

By Corollary \ref{cor:invariance of augmentation numbers}, $\mr{aug}_m(T,\hat{\rho_L},\hat{\rho_R};q)$'s are Legendrian isotopy invariants for $(T,\mu)$.
Moreover, by Theorem \ref{thm:counting for tangles}, we immediately obtain the following
\begin{corollary}\label{cor:ruling vs aug:general case}
Given Legendrian tangle $(T,\mu)$ with $B$ base points so that each connected component containing a right cusp has at least one base point, let $\hat{\rho}_L,\hat{\rho}_R$ be any $m$-graded normal rulings of any homology type $H_*$ for $T_L,T_R$ respectively. Then
\begin{eqnarray*}
\mr{aug}_m(T,\hat{\rho}_L,\hat{\rho_R};q)=q^{-\frac{d+B}{2}}z^{-B}<\hat{\rho}_L|R_T^m(z)|\hat{\rho}_R>
\end{eqnarray*}
\end{corollary}

Also, given any $m$-graded isomorphism types $\rho_L,\rho_R$ of any homology type $H_*$ of $T_L,T_R$ respectively, and $\epsilon_L\in\mc{O}_m(\rho;k)$, Theorem \ref{thm:augmentation varieties for Legendrian tangles}, together with the decomposition (\ref{eqn:decomp for non-acyclic aug}), immediately induces a structure theorem for $\mr{Aug}_m(T,\epsilon_L,\rho_R;k)$. Similarly as in Remark \ref{rem:structure thm for aug var}, there's also a structure theorem for $\mr{Aug}_m(T,\rho_L,\rho_R;k)$. That is,
\begin{eqnarray}
&&\mr{Aug}_m(T,\epsilon_L,\rho_R;k)=\sqcup_{\hat{\rho}}\mr{Aug}_m^{\hat{\rho}}(T,\hat{\epsilon_L};k)\\
&&\mr{Aug}_m(T,\rho_L,\rho_R;k)=\sqcup_{\hat{\rho}}\mr{Aug}_m^{\hat{\rho}}(T,\hat{\rho_L};k)\nonumber
\end{eqnarray}
where $\hat{\rho}$ runs over all $m$-graded normal rulings of homology type $H_*$ for $(T,\mu)$ such that
$\hat{\rho}|_{\hat{T}_L}=\hat{\rho_L}, \hat{\rho}|_{T_R}=\rho_R$, and
\begin{eqnarray*}
&&\mr{Aug}_m^{\hat{\rho}}(T,\hat{\epsilon_L};k):=\mr{Aug}_m^{\hat{\rho}}(\hat{T},\hat{\epsilon_L};k)\\
&&\mr{Aug}_m^{\hat{\rho}}(T,\hat{\rho_L};k):=
\{\hat{\epsilon}\in \mr{Aug}_m^{\hat{\rho}}(\hat{T},\hat{\rho_L};k)|\hat{\epsilon}|_{\hat{T}_L}\in\hat{\mc{O}}_m(\rho_L;k)\}
\end{eqnarray*}
where we have used Definition \ref{def:augmentaion varieties via normal rulings} for the right hand sides. Notice that $\hat{\rho_L}$ is fixed.
Moreover, have
\begin{eqnarray}
&&\mr{Aug}_m^{\hat{\rho}}(\hat{T},\hat{\epsilon_L};k)\cong
(k^*)^{-\chi(\hat{\rho})+B}\times k^{r(\hat{\rho})}\\
&&\mr{Aug}_m^{\hat{\rho}}(\hat{T},\hat{\rho_L};k)\cong
\mc{O}_m(\rho_L;k)\times(k^*)^{-\chi(\hat{\rho})+B}\times k^{r(\hat{\rho})}.\nonumber
\end{eqnarray}

\subsection{An alternative generalization}\label{subsec:alternative generalization}

There's an alternative generalization of the ``Ruling polynomials compute augmentation numbers" result.
Given a Legendrian tangle $(T,\mu)$ in $J^1U$ with $U=(x_L,x_R)$ with generic front, and given $m$-graded isomorphism types $\rho_L,\rho_R$ of $T_L,T_R$ respectively, we can instead define
the augmentation number $\mr{aug}_m(T,\rho_L,\rho_R;q)$ as in Definition \ref{def:augmentation number}. It turns out, this will also be a Legendrian isotopy invariant. This motivates us to give the alternative generalization of normal rulings and Ruling polynomials for $(T,\mu)$.

Notice that Definition \ref{def:normal_ruling} can be reformulated in terms of the language used in \cite[Def.2.1]{HR15}. That is, a $m$-graded normal ruling $\rho$ of $(T,\mu)$ is a family of \emph{fixed-point free involutions} $\rho_x$ of the strands of $T$ over the generic $x$-coordinates $x_L\leq x\leq x_R$ such that:
\begin{enumerate}[label=(\arabic*)]
\item
Label the strands of $T_x=T|_{\{x\}}$ from top to bottom by $1,2,\ldots,s_x$ over each generic $x$. Then $\rho_x:I(T_x)=\{1,2,\ldots,s_x\}\xrightarrow[]{\sim}I(T_x)$ pairs the strands, with $\mu(\text{ upper-strand})-\mu(\text{ lower-strand})=1 (\mr{mod} m)$ for each pair. Equivalently, $\rho_x$ can be identified with a $m$-graded isomorphism type (Definition \ref{def:isomorphism type for trivial Legendrian tangles}) for $T_x$: $I(T_x)=U_x\sqcup L_x\sqcup H_x$ with $H_x=\emptyset$, and $\rho_x:U_x\xrightarrow[]{\sim}L_x$.

\noindent{}Recall that the endpoints or strands in $U_x$, $L_x$ and $H_x$ are called \emph{upper}, \emph{lower} and \emph{homological} respectively.

\item
As $x$ goes from left to right, $\rho_x$ remains unchanged, except when we pass a cusp or a crossing: near a cusp, the only change is, the 2 strands connected to the cusp are paired with each other; near a crossing, the change of $\rho_x$ is indicated as in Figure \ref{fig:NR}.
\end{enumerate}

Using this language, we generalize the definition to also allow involutions with fixed-points:
\begin{definition}\cite[Def.2.4]{LR12}\label{def:generalized normal rulings_2}
A \emph{$m$-graded generalized normal ruling} $\rho$ of $(T,\mu)$ is a family of $m$-graded isomorphism types $\rho_x$ for the strands $T_x$ of $T$ over generic $x$-coordinates $x$, such that: As $x$ goes from left to right, $\rho_x$ remains unchanged except when we pass a cusp or a crossing. Near a cusp, the only change is, the 2 strands connected to the cusp are paired with each other (so the upper (resp. lower) strand is \emph{upper} (resp. \emph{lower})). Near a crossing $q$ between strands $k,k+1$, we require\\
\noindent{}either: the over and under strand of $q$ are both \emph{homological}, both before and after $q$. In that case, we say $q$ is \emph{homological};\\
\noindent{}or: No matter before or after $q$, at most one of the two strands near $q$ is homological, and we require that: For each generic $x$, so $\rho_x$ determines a partition $I(T_x)=U_x\sqcup L_x\sqcup H_x$. For each $i\in H_x$, denote $\rho_x(i):=\infty$, and pretend that the strand $i$ is paired with a fixed strand at $z=-\infty$. Then the behavior of the pairing of strands near $x$ is shown as in Figure \ref{fig:NR} (Now, the figures may also contain a strand at $z=-\infty$).
In this case, define the \emph{switches}, \emph{returns} and \emph{departures} of $\rho$ as usual.
\end{definition}
\noindent{}Note: When $T$ is a trivial Legendrian tangle, a $m$-graded generalized normal ruling is the same as a $m$-graded isomorphism type. Moreover, now the crossings of degree 0 modulo $m$ are divided into 4 types: $m$-graded homological crossings, switches, $m$-graded returns and $m$-graded departures.

\begin{definition}\label{def:generalized Ruling polynomials_2}
Given a $m$-graded generalized normal ruling $\rho$ for $(T,\mu)$, we define $s(\rho),r(\rho),d(\rho)$ as in Definition \ref{def:switches and returns}. Let $h(\rho)$ (resp. $r'(\rho)$) to be the number of $m$-graded homological crossings (resp. $m$-graded returns) of $q$. Define the \emph{Euler characteristic} of $\rho$ to be:
\begin{eqnarray}
\chi(\rho):=c_R-s(\rho)-h(\rho)
\end{eqnarray}
where $c_R$ is the number of right cusps in $T$. Moreover, given any $m$-graded generalized normal rulings $\rho_L,\rho_R$ for $T_L,T_R$ respectively, the \emph{$m$-graded generalized Ruling polynomial (with boundary conditions $(\rho_L,\rho_R)$)} for $(T,\mu)$ is:
\begin{eqnarray}
<\rho_L|\tilde{R}_T^m(z)|\rho_R>:=\sum_{\rho}z^{-\chi(\rho)}(\frac{q}{q-1})^{h(\rho)}
\end{eqnarray}
where $\rho$ runs over all $m$-graded generalized normal rulings of $T$ such that $\rho|_{T_L}=\rho_L,\rho|_{T_R}=\rho_R$, and $z=q^{\frac{1}{2}}-q^{-\frac{1}{2}}$.
\end{definition}

\noindent{}Note: When $\rho$ is a $m$-graded normal ruling, then $h(\rho)=0$, so by Remark \ref{rem:filling_surface_computation_formula}, $\chi(\rho)$ coincides with Definition \ref{def:Ruling polynomial}. Hence, $<\rho_L|\tilde{R}_T^m(z)|\rho_R>=<\rho_L|R_T^m(z)|\rho_R>$.

For generalized normal rulings, the analogue for Lemma \ref{lem:filling_surface} no longer holds. However, by looking at the change under Legendrian isotopies, we still have

\begin{corollary}\label{cor:inv of generalized ruling poly}
Given any Legendrian tangle $(T,\mu)$ and any $m$-graded generalized normal rulings $\rho_L,\rho_R$ for $T_L,T_R$ respectively, $<\rho_L|\tilde{R}_T^m(z)|\rho_R>$ defines a Legendrian isotopy invariant. Moreover, if $T=T_1\circ T_2$ is a composition of two Legendrian tangles, then the composition axiom for the generalized Ruling polynomials holds:
\begin{eqnarray*}
<\rho_L|\tilde{R}_T^m(z)|\rho_R>=\sum_{\rho_I}<\rho_L|\tilde{R}_T^m(z)|\rho_I><\rho_I|\tilde{R}_T^m(z)|\rho_R>
\end{eqnarray*}
where the sum is over all $m$-graded generalized normal rulings $\rho_I$ of $(T_1)_R=(T_2)_L$.
\end{corollary}

\begin{proof}
The composition axiom follows from the definition. The Legendrian isotopy invariance is slightly nontrivial. The proof is similar to that of Theorem \ref{thm:invariance and composition of Ruling polynomials}, we need to study the behavior of $<\rho_L|\tilde{R}_T^m(z)|\rho_R>$ under a smooth isotopy, or a Legendrian Reidemeister move of Type I, II, III respectively. The only nontrivial case is the Legendrian Reidemeister Type III move, when the bijection in Lemma \ref{lem:filling_surface} can fail. Otherwise, the proof is the same as that of Theorem \ref{thm:invariance and composition of Ruling polynomials}. By the composition axiom, it suffices to show $<\rho_L|\tilde{R}_T^m(z)|\rho_R>=<\rho_L|\tilde{R}_{T'}^m(z)|\rho_R>$, when $T, T'$ differ by Legendrian Reidemeister Type III move, and $T,T'$ both consist of exactly the 3 crossings appeared in the Type III move (Figure \ref{fig:LR} (right)). We only illustrate the proof by showing one such nontrivial case, as in Figure \ref{fig:Generalized Ruling under type III move}. The other cases are either trivial or similar. In the example of the figure, the 3 crossings are all of degree 0 modulo $m$. Say, the endpoints connected to one of the 3 crossings are $k-1,k,k+1$ (both in $T$ and $T'$), then $k,k+1$ are homological and $k-1$ is lower with respect to $\rho_L$ (and also $\rho_R$). With these fixed boundary isomorphism types $\rho_L,\rho_R$, $T$ admits 2 $m$-graded generalized normal rulings $\rho_1,\rho_2$ (Figure \ref{fig:Generalized Ruling under type III move} (left)),  and $T'$ admits a unique $m$-graded generalized normal ruling $\rho'$ (Figure \ref{fig:Generalized Ruling under type III move} (right)). Moreover, $s(\rho_1)=1,h(\rho_1)=0$, $s(\rho_2)=2,h(\rho)=1$, and $s(\rho')=1,h(\rho')=2$. Hence by Definition \ref{def:generalized normal rulings_2}, have $<\rho_L|\tilde{R}_T^m(z)|\rho_R>=z(\frac{q}{q-1})^0+z^3(\frac{q}{q-1})^1$ and $<\rho_L|\tilde{R}_{T'}^m(z)|\rho_R>=z^3(\frac{q}{q-1})^2$. One can see that they are equal precisely when $z=q^{\frac{1}{2}}-q^{-\frac{1}{2}}$.
\end{proof}

\begin{figure}[!htbp]
\begin{center}
\minipage{0.8\textwidth}
\includegraphics[width=\linewidth, height=1in]{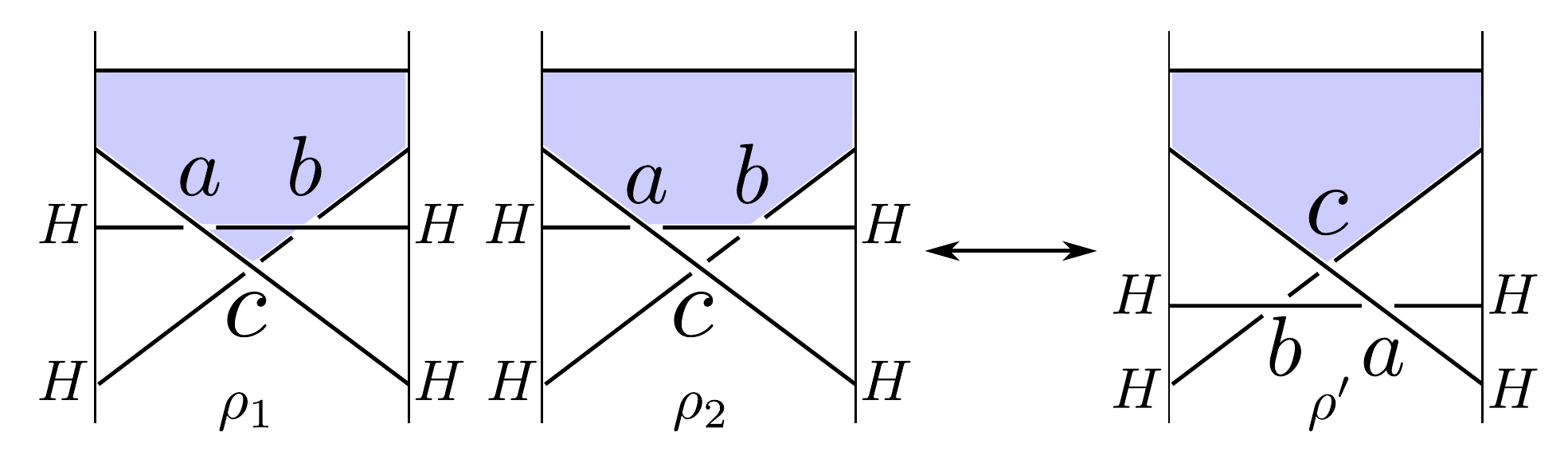}
\endminipage\hfill
\end{center}
\caption{In the example of the figure, the Reidemeister Type III move induces a correspondence between two $m$-graded generalized normal rulings $\rho_1,\rho_2$ for $T$ (left), and one $m$-graded generalized normal ruling $\rho'$ for $T'$ (right), with the fixed boundary isomorphism types. In the figure, the crossings all have degree 0 modulo $m$, ``$H$" indicates the homological endpoints, the switches are the crossings appeared in the boundary of the shadowed disks.}
\label{fig:Generalized Ruling under type III move}
\end{figure}

On the other hand, the main results of Section \ref{subsec:computation of aug number}-\ref{subsec:Ruling vs aug: acyclic case} admit a direct generalization to the case of generalized normal rulings, with completely the same arguments. The only exception is Corollary \ref{cor:invariance of augmentation numbers}, where the same argument there doesn't apply to the general case. However, this is not an issue since we didn't use the corollary anywhere else.

More specifically, fix $m$-graded generalized normal rulings $\rho_L,\rho_r$ for $T_L,T_R$ respectively and take any $\epsilon_L$ in $\mc{O}_m(\rho_L;k)$. As in Equation \ref{eqn:partition of aug var}, we obtain a decomposition of the augmentation variety
\begin{eqnarray}
\mr{Aug}_m(T,\epsilon_L,\rho_R;k)=\sqcup_{\rho}\mr{Aug}_m^{\rho}(T,\epsilon_L;k)
\end{eqnarray}
where $\rho_L$ runs over all $m$-graded generalized normal rulings of $T$ such that $\rho|_{T_L}=\rho_L,\rho|_{T_R}=\rho_R$, and $\mr{Aug}_m^{\rho}(T,\epsilon_L;k)$ is defined as in Definition \ref{def:augmentaion varieties via normal rulings}. Moreover, as in Theorem \ref{thm:augmentation varieties for Legendrian tangles}, have
\begin{eqnarray}
\mr{Aug}_m^{\rho}(T,\epsilon_L;k)\cong (k^*)^{-\chi(\rho)-h(\rho)+B}\times k^{r(\rho)+h(\rho)}
\end{eqnarray}
The proof is completely similar. Finally, as in Theorem \ref{thm:counting for tangles}, we obtain the following direct generalization:
\begin{eqnarray}
\mr{aug}_m(T,\rho_L,\rho_R;q)=q^{-\frac{d+B}{2}}z^B<\rho_L|\tilde{R}_T^m(z)|\rho_R>.
\end{eqnarray}
In particular, this also shows the Legendrian isotopy invariance of $\mr{aug}_m(T,\rho_L,\rho_R;q)$ without using the arguments in the proof of Corollary \ref{cor:invariance of augmentation numbers}. Similarly, as in Remark \ref{rem:structure thm for aug var}, one also obtains a structure theorem for $\mr{Aug}_m(\rho_L,\rho_R;k)$ for any $m$-graded generalized normal rulings $\rho_L,\rho_R$ of $T_L,T_R$ respectively.

\end{document}